%% file: main.tex
\documentclass{article}
\input{preamble.tex}

\usepackage{morefloats}

\usepackage{scalerel}
\usepackage[ruled]{algorithm2e}

\usepackage{url}
\usepackage{float}
\usepackage[backend=biber,style=alphabetic,sorting=nyt,maxbibnames=99,backref=true]{biblatex}
\usepackage{tikz}
\usetikzlibrary{arrows,decorations.markings,decorations.pathmorphing,backgrounds,positioning,fit,petri,knots,shapes,automata, arrows.meta, positioning, decorations.pathreplacing,calligraphy}
\usetikzlibrary{hobby}

\newcommand{\Xs}{\raisebox{-.3ex}{\includegraphics[scale=0.019]{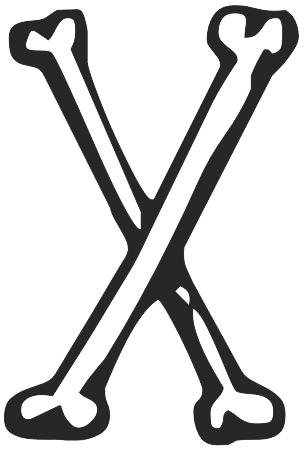}}}

\begin{document}

\author{Nathalie Aubrun,\  Nicolás Bitar}

\date{}

\title{Self-Avoiding Walks on Cayley Graphs Through the Lens of Symbolic Dynamics}
\maketitle
\begin{abstract}
\hskip -.2in
\noindent
We study dynamical and computational properties of the set of bi-infinite self-avoiding walks on Cayley graphs, as well as ways to compute, approximate and bound their connective constant. To do this, we introduce the skeleton $\Xs_{G,S}$ of a finitely generated group $G$ relative to a generating set $S$, which is a one-dimensional subshift made of configurations on $S$ that avoid all words that reduce to the identity. We provide a characterization of groups which have SFT skeletons and sofic skeletons: first, there exists a finite generating set $S$ such that $\Xs_{G,S}$ is a subshift of finite type if and only if $G$ is a plain group; second, there exists $S$ such that $\Xs_{G,S}$ is sofic if and only if $G$ is a plain group, $\mathbb{Z}\times\mathbb{Z}/2\mathbb{Z}$ or $\mathcal{D}_{\infty}\times\mathbb{Z}/2\mathbb{Z}$. We also characterize finitely generated torsion groups as groups whose skeletons are aperiodic.

For connective constants, using graph height functions and bridges, we show that Cayley graphs of finitely generated torsion groups do not admit graph height functions, and that for groups that admit transitive graph height functions, the connective constant is equal to the growth rate of periodic points of the skeleton. Finally, we take a brief look at the set of bi-infinite geodesics and introduce an analog of the connective constant for the geodesic growth.
\vskip .1in
\noindent\keywords{Self-avoiding walk, Cayley graph, connective constant, subshift of finite type, sofic subshift, graph height function, geodesics, entropy.}

\end{abstract}

%=-=-=-=-=-=-=-=-=-=-=-=-=-=-=-=-=-=-=-=-=-=-=-=-=-=-=-=-=-=-=-=-=-=-=-=-=-=-=
%
%=-=-=-=-=-=-=-=-=-=-=-=-=-=-=-=-=-=-=-=-=-=-=-=-=-=-=-=-=-=-=-=-=-=-=-=-=-=-=

\section{Introduction}

In this article we study bi-infinite self-avoiding walks on Cayley graphs of finitely generated groups from the point of view of symbolic dynamics and group theory.
A self-avoiding walk is a path on a graph that visits a vertex at most once. Figure~\ref{fig:introsaw} shows an example of a self-avoiding walk on the hexagonal grid. These walks were originally introduced by Flory for the study of long-chain polymers~\cite{flory1949configuration}. Although the original setting was the infinite square grid, self-avoiding walks are more generally studied in the context of infinite quasi-transitive graphs, intersecting with areas such as combinatorics, probability and statistical physics. The fundamental problem in this area is the study of the asymptotic growth rate of the number of self-avoiding walks of a given length, called the connective constant. See~\cite{grimmet2017self} for a recent survey on this problem. Recently, there has been increasing interest in the study of the set of all self-avoiding walks on edge-labeled graphs from the point of view of formal language theory~\cite{lindorfer2020language,lehner2023saw}. We take this study further by focusing on both bi-infinite self-avoiding walks and bi-infinite geodesics on Cayley graphs of finitely generated groups.

\begin{figure}[!ht]
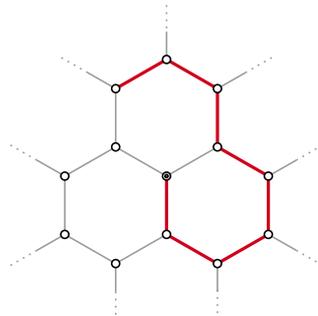

    \centering
    \includestandalone[scale=0.8]{figures/sawIntro}
    \caption{A self-avoiding walk, marked in red, on the hexagonal grid.}
    \label{fig:introsaw}
\end{figure}

Given a finitely generated group $G$, and a symmetric finite generating set $S$, we introduce the skeleton of $G$ with respect to $S$, denoted $\Xs_{G,S}$, as the set of labels of bi-infinite self-avoiding walks on the Cayley graph $\Gamma(G,S)$. The name, \emph{skeleton}, comes from the study of the decidability of tiling problems on groups, particularly the infinite snake problem~\cite{aubrun2023domino}. The skeleton $\Xs_{G,S}$ is equivalently defined as the set of bi-infinite words $x\in S^{\Z}$ that contain no factors representing the identity in $G$ (see Sections~\ref{subsec:snake} and~\ref{subsec:SAWs} for formal definitions). This set admits a $\Z$-action through the shift operation, making it a symbolic dynamical system, commonly refered to as a subshift. Our present goal is to establish connections between the dynamical properties $\Xs_{G,S}$, and geometric and algebraic properties of the underlying group and Cayley graph. We will explore how the skeleton provides a way to translate results from groups to shift spaces.\\

The article is divided into three parts. First, we look at the relation between the skeletons' dynamical properties and the underlying group's properties. We study which groups admit skeletons that are subshifts of finite type, sofic, effective, or minimal, and study their periodic points. We completely characterize the first two properties and provide partial results for the rest. Second, we look at the skeleton's entropy. Using a result independently established by Rosenfeld~\cite{rosenfeld2022finding} for subshifts and Grimmet et al.~\cite{grimmett2014extendable} for self-avoiding walks, we show the topological entropy of the skeleton is given by
$$h(\Xs_{G,S}) = \log(\mu(G,S)),$$
where $\mu(G,S)$ is the connective constant of the Cayley graph $\Gamma(G,S)$. This answers one of the questions from Problem 108 posed by Rufus Bowen in his notebook of problems~\cite{bowen108}. With this connection, we give some results on the approximation of connective constants and the existence of graph height functions. Lastly, we focus our attention on a subset of the skeleton composed of all bi-infinite geodesics, which we call the geodesic skeleton. We study its dynamic properties and introduce a geodesic analog of the connective constant.

% -=-=-=-=-=-=-=-=-=-=-=-=-=-=-=-=-=-=-=-=-=-=-=-=-=-=-=-=-=-=-
\subsection*{Characterizing Classes of Subshifts}

An important class of subshifts in the theory of symbolic dynamics is the class of subshifts of finite type. These subshifts are defined as sets of bi-infinite sequences that avoid a finite set of forbidden patterns. In this article, we completely classify groups that admit a skeleton that is a subshift of finite type.

\begin{thmx}
\label{intro:SFT}
    Let $G$ be a finitely generated group. Then, there exists a finite generating set $S$ such that
    $\Xs_{G,S}$ is an SFT if and only if $G$ is a plain group.
\end{thmx}

Plain groups are a subset of virtually free groups defined as free products of a finite number of finite groups and a free group (Definition~\ref{def:plain}). The connective constants of some Cayley graphs of plain groups were previously studied by Gilch and M{\"u}ller~\cite{gilch2017counting}. We refine their result by showing that, when the skeleton is a subshift of finite type, the connective constant is a non-negative rational power of a Perron number.\\

A bigger class of subshifts is the class of sofic shifts. A subshift is said to be sofic if the language of its finite factors is regular. For the skeleton, this is equivalent to the language of labels of bi-infinitely extendable self-avoiding walks being regular. Lindorfer and Woess showed that the set of labels of finite length self-avoiding walks on a connected quasi-transitive locally finite deterministically-labeled infinite graph is regular if and only if the graph has only thin ends, all of them of size 1~\cite{lindorfer2020language}. For Cayley graphs, through a theorem of Haring-Smith (Theorem~\ref{thm:HS}), this can be shown to be exactly the class of plain groups. Nevertheless, when working with bi-infinitely extendable self-avoiding walks, there exist Cayley graphs (in particular the bi-infinite ladder graph) with ends of size 2 where the language of bi-infinite extendable walks is regular, i.e. its skeleton $\Xs_{G,S}$ is sofic. We classify all groups that admit a Cayley graph for which the skeleton is sofic.

\begin{thmx}
\label{intro:sofic}
    Let $G$ be a finitely generated group. There exists $S$ such that $\Xs_{G,S}$ is sofic if and only if $G$ is a plain group, $\Z\times  \Z/2\Z$ or $\D_{\infty}\times  \Z/2\Z$.
\end{thmx}

This result should be contrasted with Proposition~\ref{prop:not_sofic} that shows that every group admits a generating set such that the skeleton is not sofic. In other words, the property of having a sofic skeleton is dependent on the Cayley graph. As is the case for subshifts of finite type, when the skeleton is sofic its connective constant is a non-negative rational power of a Perron number.\\

Lastly, from our study of effective skeletons, we provide an alternative proof of the existence of effective one-dimensional subshift that has no computable configurations (Proposition~\ref{prop:monster}).

% -=-=-=-=-=-=-=-=-=-=-=-=-=-=-=-=-=-=-=-=-=-=-=-=-=-=-=-=-=-=-
\subsection*{Connective Constants and Graph Height Functions}

The next part of the article is concerned with finding ways to approximate connective constants through periodic configurations and finding lower bounds when they are not available.\\

We show that if a Cayley graph $\Gamma(G,S)$ admits a graph height function (see Definition~\ref{def:graphHeight}), then the skeleton contains periodic points. As shown in~\cite{aubrun2023domino}, torsion groups can be characterized as groups whose skeletons never contain periodic points, independently of the generating set. This allows us to state the following result.
\begin{thmx}
\label{intro:height_function}
        The Cayley graphs of infinite torsion finitely generated groups do not admit graph height functions.
\end{thmx}

This theorem generalizes a result by Grimmet and Li who showed the Grigorchuk group does not admit graph height functions, and more generally, Cayley graphs of torsion groups with certain conditions on the stabilizer of the identity~\cite{grimmett2015self}.\\

In contrast, using the General Bridge Theorem by Lindorfer~\cite{lindorfer2020bridge}, we can extend a result of Clisby~\cite{clisby2013endless} on the approximation of the connective constant through periodic self-avoiding walks for the square lattice, to groups admitting a particular type of graph height function.
\begin{thmx}
\label{intro:approximation}
    Let $G$ be a finitely generated group and $S$ a finite generating set. If $\Gamma(G,S)$ admits a graph height function $(h, H)$ such that $H$ acts transitively on $\Gamma(G,S)$, then 
    $$\mu(G,S) = \lim_{n\to \infty} \sqrt[n]{e_n},$$
    where $e_n$ denotes the number of periodic points of period $n\in\N$ in $\Xs_{G,S}$ .
\end{thmx}

The next step is to look at torsion groups to find ways of approximating their connective constants without the use of periodic points. To do this, we use a counting method popularized by Rosenfeld~\cite{rosenfeld2020colorings,rosenfeld2022finding}, that provides us with a method to find lower bounds on any graph by finding solutions to an inequality dependent on the number of simple cycles on the graph (see Proposition~\ref{prop:rosenfeld}).

% -=-=-=-=-=-=-=-=-=-=-=-=-=-=-=-=-=-=-=-=-=-=-=-=-=-=-=-=-=-=-
\subsection*{Bi-infinite Geodesics}

A sub-class of self-avoiding walks that are of special interest are geodesics. A geodesic is a shortest path between two points on the Cayley graph. We define the geodesic skeleton, $\Xs^g_{G,S}$ of a group $G$ with respect to the generating set $S$ as the set of all bi-infinite geodesics on the corresponding Cayley graph.
We obtain partial results on the classification of groups that admit geodesic skeletons that are subshifts of finite type, sofic and effective. In addition, we show that the characterization of torsion groups through periodic points also holds in this case.
\begin{thmx}
    Let $G$ be a finitely generated group. The following are equivalent,
    \begin{itemize}
        \item  $G$ is a torsion group,
        \item $\Xs_{G,S}$ is aperiodic for all (some) generating sets $S$,
        \item $\Xs^g_{G,S}$ is aperiodic for all (some) generating sets $S$.
    \end{itemize}
\end{thmx}

Finally, we introduce a geodesic analog of the connective constant. If we take $\Gamma_{G,S}$ the geodesic growth function of $G$ with respect to $S$, we define the geodesic connective constant of a Cayley graph as the limit,
$$\mu^g(G,S) = \lim_{n\to\infty}\sqrt[n]{\Gamma_{G,S}(n)}.$$
The geodesic growth of groups has been extensively studied, especially in the case of virtually nilpotent groups~\cite{bridson2012groups,bishop2020virtually,bishop2021geodesic,bodart2023intermediate}. As was the case with the skeleton, the entropy of the geodesic skeleton is the logarithm of the geodesic connective constant. We are able to explicitly compute the geodesic connective constants for lattices with known (or well approximated) connective constants such as the square grid, ladder graph ($\mathbb{L}$) and hexagonal grid ($\mathbb{H}$):
\begin{itemize}
    \item $\mu^{g}(\Z^2) = 2$,
    \item  $\mu^{g}(\mathbb{L}) = 1$,
    \item $\mu^{g}(\mathbb{H})=\sqrt{2}$.
\end{itemize}

% -=-=-=-=-=-=-=-=-=-=-=-=-=-=-=-=-=-=-=-=-=-=-=-=-=-=-=-=-=-=-
\paragraph{Structure of the Article}

The paper is organized as follows. Section~\ref{section.preliminaries} is devoted to definitions and background on symbolic dynamics, combinatorial group theory ans self-avoiding walks. Section~\ref{section.general_properties} surveys general properties of the skeleton subshift, and shows how its entropy corresponds to the logarithm of the connective constant of the corresponding Cayley graph. In Section~\ref{section.forbid_periodic_computational} we investigate how dynamical and computational properties of the skeleton subshift -- existence of periodic configurations in $\Xs_{G,S}$, minimality of $\Xs_{G,S}$, being SFT or effective -- relate to properties on the group $G$ itself. Next, in Section~\ref{section.sofic_skeleton} we provide a characterization of groups that admit sofic skeletons. To do this we also introduce notions from the study of thin and thick ends of graphs, and automorphisms of graphs. In Section~\ref{section.entropy_CC} we use the skeleton to get new results on entropy and connective constant. We begin by looking at graph height functions and bridges, and their relation to periodicity in the skeleton. Then, we use Rosenfeld's counting method to provide lower bounds on the connective constant based on the number of simple cycles on the Cayley graph of a given length. Finally, Section~\ref{seciton.geodesic} is devoted to the study of the geodesic skeleton and the geodesic connective constant.

%%=-=-=-=-=-=-=-=-=-=-=-=-=-=-=-=-=-=-=-=-=-=-=-=-=-=-=-=-=-=-=-
\section{Background and Definitions}
\label{section.preliminaries}

Given an alphabet $A$, we denote by $A^n$ the set of words on $A$ of length $n$, $A^{\leq n}$ the set of words of length at most $n$, and $A^*$ the set of all finite length words including the empty word $\epsilon$. Furthermore, we denote by $A^+ = A^*\setminus\{\epsilon\}$ the set of non-empty finite words over $A$. A factor $v$ of a word $w$ is a contiguous subword of $w$; we denote this by $v \factor w$. For a bi-infinite word $x\in A^\Z$, given $i,j\in\Z$, $x_{[i,j]}$ denotes the word $x_ix_{i+1}\ ...\ x_j$, $x_{[j,+\infty)}$ the infinite word stating at $j$, and $x_{(-\infty,i]}$ the infinite word finishing at $i$. For a word $w\in A^*$, the expression $w^{\infty}$ denotes the infinite word obtained by repeating $w$. We denote the free group defined by the free generating set of size $n$ by $\F_n$, and $\F_S$ the free group generated by $S$. The commutator of two group elements $g,h$ is denoted by $[g,h]=ghg^{-1}h^{-1}$.

%=-=-=-=-=-=-=-=-=-=-=-=-=-=-=-=-=-=-=-=-=-=-
\subsection{Symbolic Dynamics}
\label{subsec:symbolic}
Given a finite alphabet $A$, we define the \define{full-shift} over $A$ as the set of maps $A^{\Z} = \{x:\Z\to A\}$. We call maps $x:\Z\to A$ \define{configurations}. There is a natural $\Z$-action on the full-shift called the \define{shift}, $\sigma:A^{\Z}\to A^{\Z}$, given by $\sigma(x)_{i} = x_{i+1}$. The full-shift is also endowed with the prodiscrete topology, making it a compact space.

Let $F\subseteq\Z$ be a finite connected subset. A \define{pattern} of support $F$ is an element $p\in A^F$. We say a pattern $p\in A^F$ appears in a configuration $x\in A^{\Z}$ if there exists $k\in\Z$ such that $x_{k+i} = p_i$ for all $i\in F$. Given a set of patterns $\Fo$, we define the set of configurations where no pattern from $\Fo$ appears as,
$$\X_{\Fo} \coloneqq \{x\in A^{\Z}\mid \forall p\in\mathcal{F}, \  p \text{ does not appear in } x\}.$$

A \define{subshift} is a subset of the full-shift $X\subseteq A^\Z$ such that there exists a set of patterns $\Fo$ that verifies $X = \X_{\Fo}$. Subshifts are equivalently defined as closed $\sigma$-invariant subsets of the full-shift. We say a subshift $X$ is
\begin{itemize}
    \item a \define{subshift of finite type} (SFT) if there exists a finite $\Fo$ such that $X = \X_{\Fo}$,
    \item \define{sofic} if there exists a regular $\Fo$ such that $X = \X_{\Fo}$,
    \item \define{effective} if there exists a decidable $\Fo$ such that $X = \X_{\Fo}$.
\end{itemize}
Each class is strictly contained in the next.\\

The \define{language} of a subshift $X$, denoted $\lang(X)$, is defined as the set of all contiguous patterns that appear within some configuration from $X$. Formally, 
$$\lang(X) = \{w\in A^*\mid \exists x\in X, w\factor x\}.$$
Any subshift can be defined by taking the complement of its language as forbidden patterns, that is, $X = \X_{\mathcal{L}(X)^c}$. Furthermore, a subshift is sofic if and only if its language is regular. Similarly, for a set of forbidden patterns we define the language of \define{locally admissible patterns}, $\lang_{loc}(\Fo)$, as the set of words $w\in A^*$ which contain no patterns from $\Fo$. Notice that $\lang(\X_{\Fo})\subseteq \lang_{loc}(\Fo)$, but is often not equal, and $\X_{\Fo} = \X_{\lang_{loc}(\Fo)^c}$.\\

An important measure of the combinatorial properties of the subshift is its \define{complexity}, $p_{X}$, defined as $p_X(n)=|\lang(X)\cap A^{n}|$, which counts the amount of words of a given length present in the subshift's language. Because the complexity is a submultiplicative function, that is, $p_X(m+n)\leq~p_X(m)p_X(n)$, Fekete's Lemma allows us to define the \define{asymptotic word growth rate} as 
$$\alpha^{\infty}(X) = \lim_{n\to\infty}\sqrt[n]{p_X(n)}.$$
Analogously, we define the function $q(n) = |\lang_{loc}(X)\cap A^n|$, which is also submultiplicative, and thus $\alpha(X) = \lim_{n\to\infty}\sqrt[n]{q(n)}$ exists. Rosenfeld showed this quantity is equal to the asymptotic word growth rate.

\begin{lemma}[\cite{rosenfeld2022finding}]
\label{lem:global_local}
    For $\Fo\subseteq A^+$, $\alpha(\X_{\Fo}) =\alpha^{\infty}(X_{\Fo})$.
\end{lemma}

An important quantity in symbolic dynamics is the (topological) \define{entropy} of the subshift, defined as 
$$h(X) = \lim_{n\to\infty}\frac{1}{n}\log(p_{X}(n)) = \log(\alpha^{\infty}(X)).$$
For example, in the case of the full-shift on $A$, $h(A^{\Z}) = \log(|A|)$. Entropy informs many dynamical properties of the subshift, and is invariant under shift-commuting continuous bijections. Furthermore, for two subshifts $X$ and $Y$ such that $Y\subseteq X$, $h(Y)\leq h(X)$. A classical result by Lind~\cite{lind1984entropies} shows that the entropies of SFTs and sofic subshifts are exactly the set of non-negative rational multiple of Perron numbers. In contrast, the set of entropies of effective subshifts is the set of right computable real numbers~\cite{hertling2008shifts}.\\

 We say a configuration $x\in X$ is \define{periodic} if there exists $k\in\Z\setminus\{0\}$ such that $\sigma^{k}(x) = x$. We say the subshift $X$ is \define{aperiodic} if it contains no periodic configurations. If a non-empty subshift is sofic, it always contains periodic configurations.\\

 We say a subshift is \define{minimal} if it does not contain non-empty proper subshifts. Equivalently, a subshift is minimal if every orbit under the shift action is dense in the subshift. Finally, a minimal subshift with a periodic configuration is always finite.\\

 % We say a word $w\in A^*$ is \define{uniformly recurrent} in a configuration $x$ if there exists $N\geq 1$ such that $w$ is a factor of $x_{[i, i+N]}$ for every $i\in\Z$. Uniform recurrence also characterizes minimality. A subshift $X$ is minimal if and only if every $w\in\lang(X)$ is uniformly recurrent in every configuration of $x$. 

 For a comprehensive introduction to one-dimensional symbolic dynamics we refer the reader to \cite{lind2021introduction}, were proof of our assertions can be found.
%=-=-=-=-=-=-=-=-=-=-=-=-=-=-=-=--=-=-=-=-=-=-=-=-=-
\subsection{Combinatorial Group Theory}

Let $G$ be a finitely generated (f.g.) group and $S$ a finite generating set. In this article we will only consider finite \define{symmetric} generating sets, that is, generating sets that verify $S=S^{-1}$, that never contain the identity. Elements in the group are represented as words on the alphabet $S$ through the evaluation function $w\mapsto \overline{w}$. Two words $w$ and $v$ represent the same element in $G$ when $\overline{w} = \overline{v}$, and we denote this by $w =_G v$. We say a word is \define{reduced} if it contains no factor of the form $ss^{-1}$ or $s^{-1}s$ with $s\in S$. We denote the identity of a group $G$ by $1_G$.
 
 \begin{definition}
  Let $G$ be a group. We say $(S, R)$ is a \define{presentation} of $G$, denoted $G = \langle S \mid R\rangle$, if the group is isomorphic to  $\langle S \mid R\rangle = \F_S/\llangle R\rrangle$, where $\llangle R\rrangle$ is the normal closure of $R$, i.e. the smallest normal subgroup containing $R$. We say $G$ is \define{recursively presented} if there exists a presentation $(S,R)$ such that $S$ is finite and $R$ is recursively enumerable.
 \end{definition}
For a group $G$ and a generating set $S$, we define:
  $$\WP(G, S) \coloneqq \{w\in S^{*} \ | \ w=_{G} \epsilon\}.$$

 \begin{definition}
  The \define{word problem} of a group $G$ with respect to a set of generators $S$ asks to determine, given a word $w\in S^*$, if $w\in \WP(G,S)$.
 \end{definition}
 We say a word $w\in S^+$ is $G$\define{-reduced} if $w$ contains no factor in $\WP(G,S)$.  We say a word $w\in S^*$ is a \define{geodesic} if for all words $v\in S^*$ such that $\overline{w} = \overline{v}$ we have $|w|\leq |v|$. For a given group $G$ and generating set $S$, we denote its \define{language of geodesics} by $\Geo(G,S)$. The length of an element $g\in G$ with respect to $S$ is defined as $\|g\|_{S} = |w|$ where $w$ is any geodesic representing $g$. This length also defines a $G$-invariant metric $d_S:G\times G\to \N$ given by $d_S(g,h) = \|g^{-1}h\|_S$.

% \todo{Add strict growth function here}
\begin{definition}
    Let $G$ be a finitely generated group, with generating set $S$. The \define{growth function} $\gamma_{G,S}:\N\to\N$ of $G$ with respect to $S$ is defined for a given $n\in\N$ as the amount of elements of length at most $n$. In other words, $\gamma_{G,S}(n) = |\{g\in G\mid \|g\|_S\leq n\}|$.
\end{definition}

We say an element $g\in G$ has \define{torsion} if there exists $n\geq 1$ such that $g^n = 1_G$. If there is no such $n$, we say $g$ is \define{torsion-free}. Analogously, we say $G$ is a \define{torsion group} if all of its elements have torsion. Otherwise, if the only torsion element is the identity, we say the group is \define{torsion-free}.\\

The \define{free product} of two groups $G$ and $H$ given by presentations $\langle S_G\mid R_G\rangle$ and $\langle S_H \mid R_H\rangle$ respectively, is the group given by the presentation,
$$G\ast H = \langle S_G\cup S_H \mid R_G \cup R_H\rangle.$$

Finally, let $\mathcal{P}$ be a class of groups (for example abelian groups, free groups, etc). We say a group $G$ is \define{virtually} $\mathcal{P}$, if there exists a finite index subgroup $H\leq G$ that is in $\mathcal{P}$.

%%=-=-=-=-=-=-=-=-=-=-=-=-=-=-=-=-=-=-=-=-=-=-=-=-=-=-=-=-=-=-=-
\subsection{The Skeleton, Infinite Domino Snakes and Analogies between Groups and Subshifts}
\label{subsec:snake}
Let $G$ be a finitely generated group with $S$ a set of generators. We define the \define{skeleton} of $G$ with respect to $S$ as the subshift,
$$\Xs_{G,S} = \{x\in S^{\Z} \mid \forall w\factor x, \ w\notin\WP(G,S)\} = \X_{\WP(G,S)}.$$

In other words, $\Xs_{G,S}$ is the set of all bi-infinite words that do not contain factor that evaluate to the group's identity. This subshift was originally defined in the context of the \define{infinite snake problem}~\cite{aubrun2023domino}, where its properties where shown to inform the decidability of the problem. In particular, if $\Xs_{G,S}$ is sofic, then the infinite snake problem for $(G,S)$ is decidable.
\begin{example}
    Take $\Z^2$ with its standard generating set $S = \{\tt{a}^{\pm}, \tt{b}^{\pm}\}$. Its skeleton is given by 
    $$\Xs_{\Z^2, S} = \{x\in S^{\Z} \mid \forall w\factor x,\ |w|_{\tt{a}} \neq 0\ \vee\ |w|_{\tt{b}} \neq 0 \},$$
    where $|w|_s$ is the sum of exponents of the generator $s$.
\end{example}
\begin{example}
\label{ex:dihedral}
    Let $\D_{\infty}$ be the infinite dihedral group. The skeletons of this group can be radically different depending on the generating set. For instance, if we take the presentation 
    $$\D_{\infty} = \langle \tt{a}, \tt{b} \mid \tt{a}^2, \tt{b}^2\rangle,$$
    the corresponding skeleton is the finite subshift $\{(\tt{ab})^{\infty}, (\tt{ba})^{\infty}\}$. On the other hand, if we take the presentation,
    $$\D_{\infty} = \langle \tt{r}, \tt{s} \mid \tt{s}^2, \tt{srsr}\rangle,$$
    the skeleton is infinite: for every $n\in\Z$ it contains a configuration $x$ defined by $x(n)=\tt{s}$ and $x(k) = \tt{r}$ for all $k\neq n$.
\end{example}

The skeleton is also present in Rufus Bowen's notebook of problems~\cite[Problem 108]{bowen108}, where he asks what can be said about $\Xs_{G,S}$, what is its entropy and if it is intrisically ergodic. In Section~\ref{subsec:entropy} we tackle the second question. Further still, this subshift is inserted in the larger project of understanding the analogies between multidimensional subshifts and finitely generated groups. Jeandel and Vanier observed~\cite{jeandel2019characterization} that the forbidden patterns of a subshift play a similar role to relations of a group. A summary of this comparison is shown in Table~\ref{tab:dict}.
\begin{table}[h!]
\begin{center}
\begin{tabular}{ |c|c| } 
 \hline
 Group &  Subshift \\
 \hline
 Group with $n$ generators & Subshift on $n$ symbols \\ 
 Word problem $\WP(G)$ & co-language $\mathcal{L}(X)^c$\\
 Finitely presented group & SFT\\
 Recursively presented group & Effectively closed subshift\\
 Simple group & Minimal subshift\\
 $H$ is a quotient of $G$ & $Y\subseteq X$\\
 \hline
\end{tabular}
\caption{A part of the Jeandel-Vanier dictionary between groups and subshifts as introduced in~\cite{jeandel2019characterization}}
\label{tab:dict}
\end{center}
\end{table}
These comparison has been further strengthened through results such as Higman and Boone-Higman Theorems for subshifts~\cite{jeandel2019characterization}.\\

The skeleton is an attempt to establish these analogies explicitly, by using the generators as an alphabet, and $\WP(G,S)$ as the set of forbidden patterns. We will see that in this case some of these analogies hold and some do not.

%=-=-=-=-=-=-=-=-=-=-=-=-=-=-=-=-=-=-=-=-=-=-=-=-=-=-=-=-=-=-=-
\subsection{Cayley Graphs and Self-Avoiding Walks}
\label{subsec:SAWs}

Let $G$ be a finitely generated group along with a finite symmetric generating set $S$. The \define{Cayley graph} of $G$ with respect to $S$, denoted $\Gamma(G,S)$, is defined by the set of vertices $V_{\Gamma} = G$ and the set of labeled edges $E_{\Gamma} = \{(g, s, gs) \mid g\in G, s\in S\}\subseteq G\times S\times G$. Each edge $e = (g,s,h)\in E_{\Gamma}$ has an initial vertex $\init(e) = g$, a terminal vertex $\ter(e) = h$ and a label $\lambda(e) = s$. The graph is also endowed with an involution $e\mapsto e^{-1}=(h, s^{-1}, g)\in E_{\Gamma}$. If a generator has order $2$, that is, if $s\in S$ satisfies $s^2 = 1_G$, we take a unique edge between $g$ and $gs$ for every $g\in G$. Notice that every Cayley graph is $|S|$-regular, locally finite, transitive and deterministically labeled, that is, for every vertex there is a unique out-going edge for each label $S$. The group $G$ acts by translation on $\Gamma(G,S)$ by left multiplication, in other words, the action of $g\in G$ over a vertex $h\in V_{\Gamma}$ is given by $g\cdot h = gh$. Through this action, we can identify $G$ with a subgroup of the automorphism group of the Cayley graph. 

We also consider the \define{undirected Cayley graph} $\hat{\Gamma}(G,S)$, where we collapse each edge $e$ and $e^{-1}$ to a single undirected edge between $\init(e)$ to $\ter(e)$. In other words, $\hat{\Gamma}(G,S)$ is the graph with vertex set $G$ such that $g,h\in G$ are adjacent if $gh^{-1}\in S$. 

\begin{example}
\label{ex:Coxeter}
    The hexagonal grid $\mathbb{H}$ is the Cayley graph of the affine Coxeter group $\tilde{A}_2$ given by the presentation,
    $$\tilde{A}_2 = \langle \tt{a}, \tt{b}, \tt{c} \mid \tt{a}^2, \tt{b}^2, \tt{c}^2, (\tt{ab})^3, (\tt{ac})^3, (\tt{bc})^3\rangle.$$
    This can be seen in Figure~\ref{fig:honeycomb}.

\begin{figure}[!ht]
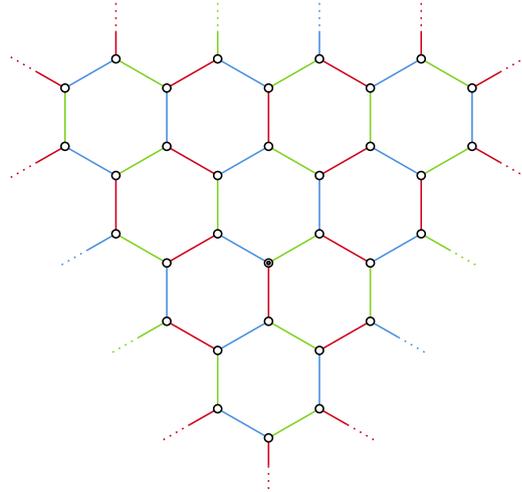

        \centering
        \includestandalone[scale=0.8]{figures/honeycomb}
        \caption{A Cayley graph of the affine Coxeter group $\tilde{A}_2$. The red edges represent $\tt{a}$, blue edges represent $\tt{b}$, and green edges $\tt{c}$.}
        \label{fig:honeycomb}
\end{figure}
\end{example}

\begin{example}
    The ladder graph $\mathbb{L}$ is the Cayley graph of $\Z\times \Z/2\Z$ with the presentation,
    $$\Z\times \Z/2\Z = \langle \tt{t}, \tt{s} \mid \tt{s}^2, \tt{tst}^{-1}\tt{s}\rangle.$$
    This can be seen in Figure~\ref{fig:ladder}.
    \begin{figure}[!ht]
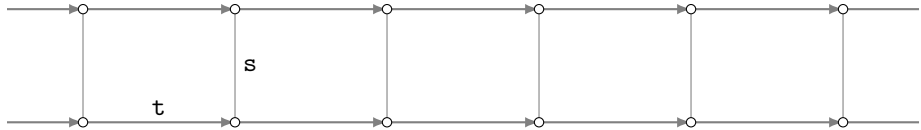

        \centering
        \includestandalone[scale=1]{figures/ladder}
        \caption{A Cayley graph of the group $\Z\times \Z/2\Z$. The generator $\tt{t}$ defines the horizontal right-pointing edges, and the generator $\tt{s}$ defines the vertical undirected edges.}
        \label{fig:ladder}
\end{figure}
This is not the only group that admits the ladder graph as a Cayley graph, this is also the case for the groups $\D_{\infty}$ and $\D_{\infty}\times \Z/2\Z$.
\end{example}

A \define{path} on $\Gamma(G,S)$ is a sequence of edges $\pi = (e_1, ..., e_n)$ such that for all $i\in\{1,...,n-1\}$ we have $\init(e_{i+1}) = \ter(e_i)$. We denote the initial vertex of the path by $\init(\pi) = \init(e_1)$ and its terminal vertex as $\ter(\pi)=\ter(e_n)$. The length of the path is given by $\ell(\pi) = n$, and its label is $\lambda(\pi) = \lambda(e_1)\ ...\ \lambda(e_n)\in S^*$. We also define the sequence of vertices visited by $\pi$ as the sequence $V(\pi) = (g_0, ..., g_n)$ with $g_i = \init(e_{i+1})$ for all $i\in\{0,..., n-1\}$ and $g_n = \ter(e_n)$. This formalism gives us a one-to-one correspondence between paths starting at $1_G$ and words in $S^*$. In particular, a path $\pi$ satisfies $\init(\pi) = \ter(\pi)$ if and only if $\lambda(\pi)\in\WP(G,S)$.\\

A path $\pi$ is a \define{self-avoiding walk} (SAW) if it never visits the same vertex twice. We define the language of self-avoiding walks over $\Gamma(G,S)$ as the set 
$$L_{SAW}(G,S) = \{\lambda(\pi) \mid \pi \text{ is a SAW with } \init(\pi)=1_G\}.$$
Remark that the language remains the same if we change the initial vertex from the identity to any other group element because the graph is transitive. Furthermore, because Cayley graphs are deterministically labeled, no two SAWs share the same label. A \define{bi-infinite SAW} centered at $g\in G$ is a sequence of edges $\pi=(e_i)_{i\in\Z}\in E^{\Z}$ such that $\init(e_{i+1}) = \ter(e_i)$, and $g = \init(e_0)$ such that $\pi$ never visits the same vertex twice. We can thus state the following.

\begin{lemma}
\label{lem:LSAW}
    Let $G$ be a group and $S$ a generating set. Then,
    \begin{align*}
        \Xs_{G,S} &= \{\lambda(\pi)\in S^{\Z} \mid \pi \text{ is a bi-infinite SAW centered at } 1_G\}, \\
        &= \X_{L_{SAW}(G,S)^c}.
    \end{align*}
    Moreover, $\lang_{loc}(\Xs_{G,S}) = L_{SAW}(G,S)$.
\end{lemma}

\begin{proof}
    This is a direct consequence of the definitions. %, since $L_{SAW}(G,S)=S^*\setminus\WP(G,S)$.
\end{proof}

Once again, as the graph is transitive, we can change the center for any other element of the group. 
Notice that any finite subwalk of a bi-infinite SAW is a SAW. The converse is not always true, that is, there are SAWs that do not appear in any bi-infinite SAW (see Figure \ref{fig:finite_SAW_Z2_non_infinite}).

\begin{figure}[!ht]
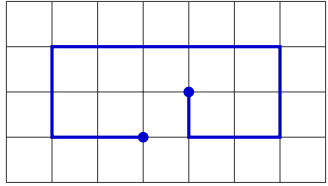

    \centering
    \include{figures/finite_SAW_Z2_non_infinite}
    \vspace{-0.3cm}
    \caption{On $\Z^2$ with standard generating set% $\left\{ (1,0),(0,1)\right\}$
    : a finite SAW  that does not appear in a bi-infinite SAW.}
    \label{fig:finite_SAW_Z2_non_infinite}
\end{figure}

\begin{remark}
    For the undirected Cayley graph $\hat{\Gamma}(G,S)$ we can define a labelling function $\lambda'$ from the set of self-avoiding walks over $\hat{\Gamma}(G,S)$ to $S^*$. Take a self-avoiding walk $\pi$ that passes through the sequence of vertices $(g_0, g_1, ..., g_n)$, its label $\lambda'(\pi) = s_0\ ...\ s_{n-1}$ is such that $s_i = g_i^{-1}g_{i+1}$. This way, both $L_{\SAW}(G,S)$ and $\Xs_{G,S}$ are equal to the corresponding labels of self-avoiding walks on the \emph{undirected} Cayley graph. Because of this, in what follows we do not distinguish between the directed and undirected Cayley graphs. 
\end{remark}

%%=-=-=-=-=-=-=-=-=-=-=-=-=-=-=-=-=-=-=-=-=-=-=-=-=-=-=-=-=-=-=-
\subsection{Connective Constants}

Let $c_n$ be the number of self-avoiding walks of length $n$ in the Cayley graph $\Gamma(G,S)$ starting at the identity. This sequence is submultiplicative, i.e. $c_{n+m}\leq c_n c_m$ for all $m,n\in\N$, so that by Fekete's Lemma, the limit of $\sqrt[n]{c_n}$ exists:
$$\mu(G,S) = \lim_{n\to\infty}\sqrt[n]{c_{n}} = \inf_{n\in\N}\sqrt[n]{c_n}\in[1,\infty).$$
This limit is known as the \define{connective constant} of the Cayley graph. For general quasi-transitive graphs, this limit was proved to be independent of the starting vertex by Hammersly and Morton~\cite{hammersley1954poor}. 

In general, connective constants are hard to compute. Nevertheless, the exact value of some connective constants is known. For instance, for the hexagonal grid (which as we saw, is a Cayley graph for $\tilde{A}_2$) its value is $\sqrt{2+\sqrt{2}}$~\cite{duminil2012connective}, for the bi-infinite ladder (as in Figure~\ref{fig:ladder}) it is the golden mean $\frac12(1+\sqrt{5})$~\cite{alm1990random}, and for some Cayley graphs of free products of finite groups it is the zero of a polynomial~\cite{gilch2017counting}. On the other hand, giving a closed form for the connective constant of $\Z^2$ with standard generators is still an open problem. The best estimate as of writing is
$$\mu(\Z^2) \approx 2.63815853032790(3),$$
obtained by Jacobsen, Scullard, and Guttman~\cite{jacobsen2016constant}. 

There are, however, bounds on the connective constant. We translate the following results from Grimmet and Li -- which were stated for larger classes of graphs -- to our Cayley graph context.
\begin{theorem}[\cite{grimmett2014strict, grimmett2015bounds}]
\label{thm:bounds_GL}
    Let $G$ be an infinite finitely generated group and $S$ a generating set.
    \begin{itemize}
        \item $\mu(G,S)\geq \sqrt{|S| - 1}$,
        \item For $w\neq_G 1_G$ and $N$ its normal closure in $G$, $\mu(G/N, S)<\mu(G,S)$,
        \item For $g\notin S$ a non-identity element of $G$ and $S' = S\cup\{g^{\pm}\}$, $\mu(G,S)<\mu(G,S')$.
    \end{itemize}
\end{theorem}

For more bounds and details, see \cite{grimmet2017self}.

%%=-=-=-=-=-=-=-=-=-=-=-=-=-=-=-=-=-=-=-=-=-=-=-=-=-=-=-=-=-=-=-

\section{General Properties}
\label{section.general_properties}

Let us begin by establishing properties of the skeleton that are common for all groups and generating sets. 

\subsection{Bi-infinite SAWs through Group Elements and Computability}

A first observation is that $\Xs_{G,S} = \varnothing$ for a generating set if and only if $G$ is a finite group; this is a consequence of Konig's Lemma (see~\cite{watkins1986infinite}). As we only consider infinite finitely generated groups, unless explicitly stated, the skeletons are never empty.  Next, if $\pi = (e_i)_{i\in\Z}$ is a bi-infinite SAW on the Cayley graph $\Gamma(G,S)$, its inverse $\pi^{-1} = (e_i^{-1})_{i\in\Z}$ is also a bi-infinite SAW. Therefore, for each configuration $x\in \Xs_{G,S}$, its inverse configuration $x^{-1}(k) = x(k)^{-1}$ belongs to $\Xs_{G,S}$.\\

Next, as shown in Figure~\ref{fig:finite_SAW_Z2_non_infinite}, the word $w = a^{-2}b^{2}a^5b^{-2}a^{-2}b$ in $\Z^2$ avoids factors from the word problem but does not define a configuration. Nevertheless, the word $ab$ defines the same group element, and can be extended to a complete configuration. Is this always the case? Is it true that for every non-trivial group element there is a word evaluating to it in $\mathcal{L}(\Xs_{G,S})$? For every finitely generated group $G$ and every group element $g\in G$, one can find a generating set $S$ such that this is true (Proposition~\ref{prop:non_trivial_in_one_skeleton}). 

\begin{proposition}\label{prop:non_trivial_in_one_skeleton}
Let $G$ be a finitely generated group. Then there exists $S$ a generating set for $G$ such that for every non trivial group element $g\in G$, there exists a word $w\in S^*$ such that $\overline{w} = g$ and $w\in\mathcal{L}(\Xs_{G,S})$.
\end{proposition}

\begin{proof}
    A theorem by Seward~\cite[Theorem 1.8]{Seward_2014} states that for every finitely generated group $G$, there exists a finite generating set $S$ such that the Cayley graph $\Gamma(G,S)$ has a regular spanning tree. In particular, this tree has no leaves, thus each path leading to a vertex can be continued (see Figure~\ref{fig:regular_spanning_tree}). Consider such an $S$ and one associated regular spanning tree of $\Gamma(G,S)$. Take a non trivial element $g\in G$ and consider the path connecting the identity $1_G$ to $g$ in the regular spanning tree. Then this finite simple path can be extended to an infinite simple path inside the spanning tree, leading to an infinite simple path going through $1_G$ and~$g$. Translating this path into a bi-infinite sequence of elements in $S$ give a configuration of the skeleton $\Xs_{G,S}$.

\begin{figure}[!ht]
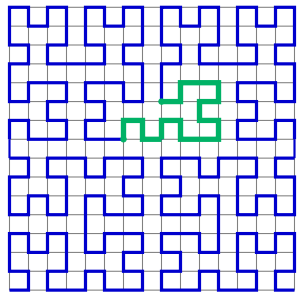

        \centering
        \include{figures/regular_spanning_tree}
        \caption{In blue, an example of regular spanning tree of degree $2$ -- thus a bi-infinite Hamiltonian path -- for $\Z^2$ with its standard presentation. In green a path from $(0,0)$ to $(2,2)$ extracted from this regular spanning tree. Note that this path is highly non-geodesic.}
        \label{fig:regular_spanning_tree}
\end{figure}
\end{proof}

This can be made more precise in the case of one and two ended groups. Cayley graph $\Gamma = \Gamma(G,S)$ has $k$ ends if $k$ is the supremum of the number of infinite connected components of the induced subgraph $\Gamma[V_{\Gamma}\setminus A]$ over every finite subset $A\subseteq V$ (see Section~\ref{subsec:EndsAndAuts}).

\begin{proposition}
    Let $G$ be a finitely generated group with one or two ends, with $S$ a generating set. Define $S' = \{g\in G \mid \|g\|_S\leq 3\}$. Then, for every non trivial group element $g\in G$, there exists a word $w\in (S')^*$ such that $\overline{w} = g$ and $w\in\mathcal{L}(\Xs_{G,S'})$.
\end{proposition}

\begin{proof}
    In~\cite[Theorem 1.3]{carrasco2023geometric} Carrasco-Vargas showed that for one and two ended groups, Sewards' Theorem holds for $S'$, that is, there is a Hamiltonian path on the Cayley graph $\Gamma(G,S')$. This directly implies our statement.
\end{proof}

By translating and joining Lemmas 4.4 and 3.7 from~\cite{carrasco2023geometric}, we can state results on the decidability of the language of the skeleton for particular sets of generators.
\begin{proposition}
    Let $G$ be a finitely generated group with one or two ends, with $S$ a generating set. Suppose $G$ has decidable word problem, and define $S' = \{g\in G \mid \|g\|_S\leq 3\}$. Then, $\lang(\Xs_{G,S'})$ is computable. 
\end{proposition}

For a skeleton, having a computable language means that there is an algorithm that determines if a finite SAW is bi-infinitely extendable. A particular class of subshifts that have computable language are sofic subshifts. In Sections~\ref{section.forbid_periodic_computational} and~\ref{section.sofic_skeleton} we will explore when skeletons belong to this class.

% %=-=-=-=-=-=-=-=-=-=-=-=-=-=-=-=-=-=-=-=-=-=-=-=-=-=-=-=-=-=-=-=-=-=-=-=-=- 

% \todo{what follows is now useless}
% \begin{lemma}\label{lem:skeleton_group_morphism}
% Let $G,H$ be two f.g. groups. Assume there exist $S,T$ generating sets for $G$ and $H$ and a bijection $\varphi:T\to S$ such that $\Xs_{G,S}=\varphi\left(X_{H,S}\right)$. Then for all $u,v\in \mathcal{L}(\Xs_{G,S})$ if $u\equiv_G v$ then $\varphi(u)\equiv_G \varphi(v)$.    
% \end{lemma}

% \begin{proof}
% Take $u,v\in \mathcal{L}(\Xs_{G,S})$ such that $u\equiv_G v$. First suppose that the words $u$ and $v$ represent paths starting from the identity $1_G$ in the Cayley graph that are disjoint except on their extremities $1_G$ and $g$. By definition of the language $\mathcal{L}(\Xs_{G,S})$ there exist two configurations $x,y\in \Xs_{G,S}$ such that $u\factor x$ and $v\factor y$. Up to a shift on $x$ and $y$, we can moreover assume that $u=x_{[1;|u|]}$ and $v=y_{[1;|v|]}$. Then $\varphi(u)\factor \varphi(x)$ and $\varphi(v)\factor \varphi(y)$.

% Assume that $\varphi(u)\not\equiv_H \varphi(v)$, said otherwise $\varphi(u)\cdot\varphi(v)^{-1}\not\equiv_H 1_H$.
% \todo{is it true that $w\not\equiv_H 1_H$ and $w$ reduced implies that $w\in\mathcal{L}(\Xs_{H,T})$ ?????}

% If $u$ and $v$ have more in common than their extremities, we split $u$ and $v$ into $u_1\dots u_n$ and $v_1\dots v_n$ such that each couple $(u_i,v_i)$ has this property.\todo{complete this case}

% \end{proof}

%=-=-=-=-=-=-=-=-=-=-=-=-=-=-=-=-=-=-=-=-=-=-=-=-=-=-=-=-=-=-=-=-=-=-=-=-=- 
\subsection{Entropy}
\label{subsec:entropy}

As seen in Lemma~\ref{lem:LSAW}, the skeleton shift is the set of labels of bi-infinite SAWs over a Cayley graph. Consequently, its complexity function counts the number of infinitely bi-extendable SAWs of length $n$, with $\alpha^{\infty}(\Xs_{G,S})$ being their asymptotic growth rate. Furthermore, the number of locally admissible words of length $n$ is exactly $c_n$, the number of finite SAWs of length $n$. Therefore, by Lemma~\ref{lem:global_local} we obtain the following.

\begin{lemma}
    For a finitely generated group $G$ and a generating set $S$, $h(\Xs_{G,S}) = \log(\mu(G,S))$.
\end{lemma}

This equality can also be deduced from \cite{grimmett2014extendable}, where Grimmet et al. show that the connective constant is equal to the growth rate of infinitely bi-extendable SAWs.\\

Let $(G,S)$ be a f.g. group and $\gamma_{G,S}:\N\to\N$ its growth function. We define its asymptotic growth rate as the value,
$$\mathfrak{H}_{G,S}= \lim_{n\to\infty}\frac{1}{n}\log(\gamma_{G,S}(n)).$$
As the growth function is sub-multiplicative (see \cite{ceccherini2010}), $\mathfrak{H}_{G,S}$ exists by Fekete's Lemma.

\begin{remark}    
\label{rem:strict_vs_normal}
    An alternative way to look at the growth of a group is the \define{strict growth function} $\sigma_{G,S}$, where $\sigma_{G,S}(n)$ is the number of elements of length \emph{exactly} $n$. As is the case for the growth function, $\sigma_{G,S}$ is sub-multiplicative. Its asymptotic growth rate is the same as that of $\gamma_{G,S}$, namely $\mathfrak{H}_{G,S}$. This can be seen through their generating functions. Take $F,f:\C\to\C$ defined as 
    $$F(z) = \sum_{n\in\N}\gamma_{G,S}(n)z^n, \ f(z) = \sum_{n\in\N}\sigma_{G,S}(n)z^n.$$
    By the Cauchy-Hadamard theorem we know that the asymptotic growth rate of $\gamma_{G,S}$ (resp. $\sigma_{G,S}$) is the reciprocal of the radius of convergence of $F$ (resp. $f$). Because strict growth can be expressed as $\sigma_{G,S}(n) = \gamma_{G,S}(n) - \gamma_{G,S}(n-1)$, with the convention that $\gamma_{G,S}(-1) = 0$, we get that $f(x) = (1-x)F(x)$. As the term $(1-x)$ does not change the radius of convergence of the series, we have that
    $$\lim_{n\to\infty}\frac{\log(\sigma_{G,S}(n))}{n} = \mathfrak{H}_{G,S}.$$ 
\end{remark}

\begin{proposition}
\label{prop:entropy}
Let $G$ be a f.g. group with $S$ a finite generating set. Then, 
$$\mathfrak{H}_{G,S} \leq h(\Xs_{G,S})\leq \log(|S|-1).$$
\end{proposition}

\begin{proof}
    Let $p(n)$ be the complexity function for $\Xs_{G,S}$ and $k = |S|$. For the upper bound, notice that the total number of reduced words of length $n$ over $S$ is exactly the number of elements of length $n$ in free group $\F_{m}$, with $m = \lceil\frac{k}{2}\rceil$. Therefore,
    $$p(n) \leq \gamma_{\F_{m}}(n) - \gamma_{\F_{m}}(n-1) = k(k - 1)^{n-1}.$$

    On the other hand, every element of length $n$ has a geodesic representative of length $n$, which by definition is $G$-reduced. In particular, this representative is a SAW of length $n$. Thus, $\sigma_{G,S}(n)\leq c_n$ and 
    $$\mathfrak{H}_{G,S} = \lim_{n\to\infty}\frac{\log(\sigma_{G,S}(n))}{n}\leq\log(\mu(G,S)).$$
\end{proof}

\begin{remark}
     The bounds from Proposition \ref{prop:entropy} are tight in general, as free groups with free generating sets satisfy $\mathfrak{H}_{\F_m,S} = h(\Xs_{\F_m,S}) = \log(2m-1)$. Nevertheless, by Theorem~\ref{thm:bounds_GL} we know that for non-free groups $h(\Xs_{G,S}) < \log(2m-1)$ for all generating sets. This same theorem also tells us that for groups with polynomial growth the lower bound is strict, as
     $$0 = \mathfrak{H}_{G,S} < \frac12\log(|S|-1)\leq h(\Xs_{G,S}).$$
\end{remark}

Another straightforward bound we find from algebraic considerations is the following.
\begin{proposition}
\label{prop:entropia_grande}
    Take $G$ a finitely generated group and $S$ generating set. If $\{s_1, ..., s_n\}\subseteq S$ is a subset of generators such that there induced semigroup $\langle s_1, ..., s_n\rangle_+$ does not contain the identity, then $h(\Xs_{G,S})\geq\log(n)$.
\end{proposition}

\begin{proof}
    If $\langle s_1, ..., s_n\rangle_+$ does not contain the identity, any combination of these generators will give a word that does not contain factors that evaluate to the identity. In other words, the skeleton contains the full-shift $\{s_1, ..., s_n\}^{\Z}$. Consequently, $h(\Xs_{G,S_n})\geq\log(n)$. 
\end{proof}

\begin{example}
    Take $\Z^2$ with its standard generating set $\{\tt{a}^\pm, \tt{b}^\pm\}$. Then, the semigroup generated by $\tt{a}$ and $\tt{b}$ does not contain the identity. Then, $h(\Xs_{\Z^2,\{\tt{a}^\pm, \tt{b}^\pm\}})\geq \log(2)$. Similarly, if we take the discrete Heisenberg group $H_3$ with generating set $\{\tt{a}^\pm, \tt{b}^\pm, \tt{c}^\pm\}$ through the presentation,
    $$H_3 = \langle \tt{a}, \tt{b}, \tt{c} \mid [\tt{a}, \tt{c}], [\tt{b}, \tt{c}], [\tt{a}, \tt{b}]\tt{c}^{-1}\rangle,$$
    the semigroup given by the three generators $\tt{a}, \tt{b}$ and $\tt{c}$ does not contain the identity. Then, by the previous proposition $h(\Xs_{H_3,\{\tt{a}^\pm, \tt{b}^\pm, \tt{c}^\pm\} })\geq \log(3)$.
\end{example}

\begin{remark}
    Given a group $G$, the entropy of its skeleton can be made arbitrarily large. This can be done be taking bigger and bigger generating sets and using the lower bound $\sqrt{|S|-1}$ given by Theorem~\ref{thm:bounds_GL}. This can also be done in torsion-free groups by taking a torsion-free element $g\in G$, a generating set containing $\{g, g^2, ..., g^n\}$ and using the previous proposition.
\end{remark}

In Section~\ref{section.entropy_CC} we will see methods to approximate entropy and connective constants for different classes of groups.

%=-=-=-=-=-=-=-=-=-=-=-=-=-=-=-=-=-=-=-=-=-=-=-=-=-=-=-=-=-=-=-=-=-=-=-=-=-
\section{Dynamic and Computational Aspects}
\label{section.forbid_periodic_computational}

The goal of this section is to explore the multiple dynamical and computational properties of skeletons, and how they interact with the algebraic properties of the underlying group. We look at groups that admits SFT, sofic, effective or minimal skeletons, as well as their periodic points.\\

A subshift can be defined by various different sets of forbidden patterns. We saw, from its definition and Lemma~\ref{lem:LSAW}, that $\Xs_{G,S}$ is defined by at least two different sets, namely $L_{SAW}(G,S)$ and $\WP(G,S)$. We begin by describing an additional set that will help us better understand the structure of forbidden patterns.\\

We begin by looking at the set of patterns that define simple cycles (also called embedded cycles) in the Cayley graph. We define the set of labels of simple cycles of a group $G$ with respect to a finite generating set $S$ as,
\begin{align*}
    \Ou_{G,S} &= \{w\in \WP(G,S) \mid w \text{ defines a simple cycle in } \Gamma(G,S)\}\\
    &= \{w\in \WP(G,S) \mid \forall w'\sqsubset w,\  w'\notin\WP(G,S)\}.
\end{align*}

\begin{example}
    Consider $\Z^2$ with its standard presentation $\langle \tt{a},\tt{b} \mid [\tt{a},\tt{b}]\rangle$. Then the word $\tt{aba}^{-3}\tt{b}^{-1}\tt{abab}^{-1}$ is in $\WP(\Z^2,\{\tt{a},\tt{b}\})$ but not in $\Ou_{\Z^2,\{\tt{a},\tt{b}\}}$ since there are repeated vertices in the path it represents in $\Gamma(\Z^2,\{\tt{a},\tt{b}\})$. 
    \begin{center}
    \begin{tikzpicture}[scale=0.75]

    \tikzset{
  % style to add an arrow in the middle of a path
  mid arrow/.style={postaction={decorate,decoration={
        markings,
        mark=at position .75 with {\arrow[#1]{stealth}}
      }}},
}

    %\draw[black!75] (-3,-1) grid (2,2);
    
    \draw[very thick,bleu,decorate,postaction={mid arrow=bleu}](0,0) -- (1,0);
    \draw[very thick,bleu,decorate,postaction={mid arrow=bleu}](1,0) -- (1,1.05);
    \draw[very thick,bleu,decorate,postaction={mid arrow=bleu}](1,1.05) -- (0,1.05);
    \draw[very thick,bleu,decorate,postaction={mid arrow=bleu}](0,1.05) -- (-1,1.05);
    \draw[very thick,bleu,decorate,postaction={mid arrow=bleu}](-1,1.05) -- (-2,1.05);
    \draw[very thick,bleu,decorate,postaction={mid arrow=bleu}](-2,1.05) -- (-2,0);
    \draw[very thick,bleu,decorate,postaction={mid arrow=bleu}](-2,0) -- (-1,0);
    \draw[very thick,bleu,decorate,postaction={mid arrow=bleu}](-1,0) -- (-1,0.95);
    \draw[very thick,bleu,decorate,postaction={mid arrow=bleu}](-1,0.95) -- (0,0.95);
    \draw[very thick,bleu,decorate,postaction={mid arrow=bleu}](0,0.95) -- (0,0);
    \filldraw[color = bleu] (0,0) circle (3pt);
    
    \end{tikzpicture}    
    \end{center}
\end{example}

We call elements of $\Ou_{G,S}$ \define{self-avoiding polygons} (SAPs) of the Cayley graph of $\Gamma(G,S)$.

\begin{lemma}
\label{lem:ouro}
    Let $\Fo = \Ou_{G,S}\cup\{ss^{-1}\mid s\in S\}$. Then, $\Xs_{G,S} = \X_{\Fo}$.
\end{lemma}

\begin{proof}
    Since $\Ou_{G,S}\subseteq\WP(G,S)$ and $\{ss^{-1}\mid s\in S\}\subseteq\WP(G,S)$, we have that $\Fo\subseteq\WP(G,S)$. So, the subshifts defined by two sets respect the reciprocal inclusion, and we have $\Xs_{G,S} \subseteq \X_{\Fo}$.

    Reciprocally, take some configuration $x\in \X_{\Fo}$ and assume it contains some pattern $w\in \WP(G,S)$. Without loss of generality we assume that $w=s_1\dots s_n$ for some $n\in\N$, so that $\overline{s_1\dots s_n} = 1_G$. Consider the group elements $g_i$ defined by $g_i = \overline{s_1\dots s_i}$ for $i\in\{1\dots n\}$ and $g_0=1_G$. Since $x\in \X_{\Fo}$ the pattern $w$ does not belong to $\Ou_{G,S}\cup\{ss^{-1}\mid s\in S\}$. So necessarily $n>1$ and there are some repetitions among the $g_i$'s in addition to $g_0=g_n$. Take two indices $i,j$ such that $i<j$ and $\{i,j\}\neq\{0\},\{n\},\{0,n\}$ (at least one of the two indices is neither $0$ nor $n$) and $i,j$ are minimal. Then the word $s_i\dots s_j$ defines a cycle in $\Gamma(G,S)$, which contradicts our original assumption. Finally, $x\in\Xs_{G,S}$, which concludes the proof.
\end{proof}

This alternative set of forbidden patterns for $\Xs_{G,S}$ will be particularly helpful in the proof of Theorem~\ref{thm:SFT}, where we characterize groups $G$ which admit a generating set $S$ such that $\Xs_{G,S}$ is an SFT and also in Section~\ref{subsection.lower_bounds_SAP}.

%=-=-=-=-=-=-=-=-=-=-=-=-
\subsection{SFT Skeletons}

To find SFTs, we start with a warm-up lemma that contains the central idea used in our classification of groups that admit skeleton SFTs.

\begin{lemma}
    $\Xs_{\Z^d, S}$ is not an SFT for $d\geq 2$ and any generating set $S$.
\end{lemma}

\begin{proof}
    Let $S$ be a generating set for $\Z^d$ and suppose $\Fo'$ is a finite set of forbidden patterns such that $\Xs_{\Z^d,S} = \X_{\Fo'}$. Then, as $S$ generates the group, there must exist $s_1, s_2\in S$ such that $\langle s_1\rangle\cap\langle s_2\rangle = \{1_{\Z^d}\}$, and $\langle s_1, s_2\rangle\simeq\Z^2$. Let us denote $N =\max_{w\in\Fo'}|w|$. Take the SAP defined by the square of length $2N$ on the first two generators $w = s_1^{2N}s_2^{2N}s_1^{-2N}s_2^{-2N}$. Notice that no factors of $w^2$ of length $N$ belong to $\Fo'$ as they are all globally admissible in $\Xs_{\Z^d,S}$. Let $x =w^{\infty}$. Clearly $x\notin \Xs_{\Z^d,S}$, as it contains $w$ which satisfies $\overline{w}=1_G$. Nevertheless, no factor of $x$ of length $N$ is contained in $\Fo'$. Therefore $x\in \X_{\Fo'}$, which is a contradiction.
\end{proof}

The main idea of this lemma is using arbitrarily large cycles that are locally self-avoiding. This way, it is not possible to detect that the path eventually crosses itself using a finite window. Which groups admit generating sets that define SFT skeletons then? Let us show that this is the case of a specific class of virtually free groups.

\begin{definition}
\label{def:plain}
    A group $G$ is \define{plain} if there exist finite groups $\{G_i\}_{i=1}^k$ and $m\geq 1$ such that $G$ is isomorphic to the free product
    $$\left(\Conv_{i=1}^k G_i\right) \ast \mathbb{F}_m.$$ 
We say a finite generating set $S$ for such $G$ is \define{standard} if it can be written as the disjoint union $S = S_1 \cup ... \cup S_k \cup S_{k+1}$ where $S_i$ is a generating set for $G_i$ and $S_{k+1}$ is a free generating set for $\F_m$.
\end{definition}

\begin{theorem}[Theorem~\ref{intro:SFT}]
\label{thm:SFT}
    Let $G$ be a finitely generated group. Then, there exists a finite generating set $S$ such that $\Xs_{G,S}$ is an SFT if and only if $G$ is a plain group.
\end{theorem}

In order to prove this theorem we use a characterization of plain groups with respect to their simple cycles. The \define{diameter} of a simple cycle is the greatest distance between to vertices in the cycle. A vertex $v$ in a graph $\Gamma$ is said to be a \define{cut vertex} if $\Gamma\setminus\{v\}$ is disconnected. A graph is said to be \define{2-connected} if it contains no cut vertices. A maximal 2-connected subgraph is called a \define{block}.

\begin{theorem}[\cite{haring1983groups}]
\label{thm:HS}
Let $G$ be a group and $m\in\N$. Then, the following are equivalent
\begin{itemize}
    \item $G$ admits a finite generating set $S$ such that all simple cycles in the undirected Cayley graph $\Gamma(G,S)$ have diameter at most $m$,
    \item $G$ admits a finite generating set $S$ such that all blocks in the undirected Cayley graph $\Gamma(G,S)$ have diamater at most $m$,
    \item $G$ is a plain group.
\end{itemize}
\end{theorem}

Proofs of this Theorem can be found in \cite{haring1983groups,elder2022rewriting}.

\begin{proof}[Proof of Theorem \ref{thm:SFT}]
    Let $G$ be a plain group decomposed as $\left(\Conv_{i=1}^k G_i\right) \ast \F_m$ with $S = S_1 \cup ... \cup S_{k+1}$ a standard generating set. Due to its free product structure, any word $w\in S^*$ can be uniquely decomposed as $w = w_{1} w_{2} .... w_{r}$ where,
	\begin{itemize}
		\item $w_{j}\in S_{l}^{*}$ for some $l$, for all $j\in\{1,\ ...,\ r\}$;
		\item $w_{j}$ and $w_{j+1}$ are words over different alphabets for all $j\in\{1,\ ...,r-1\}$.
	\end{itemize}    
	If $\overline{w} = 1_{G}$, by our decomposition, $\overline{w}_j = 1_G$ for every $j$. This means every SAP from $G$ must be entirely contained in one of the finite groups $G_i$, as $\F_m$ has no SAPs with its free generating set. Therefore, $\Ou_{G,S}$ is finite because the number of SAPs in each finite group is finite. By Lemma \ref{lem:ouro}, $\Xs_{G,S}$ is an SFT.\\
	
	Now, let $G$ be a finitely generated group with $S$ such that $\Xs_{G,S}$ is an SFT, defined by the finite set of forbidden patterns $\Fo$. If $G$ is not a plain group, by Theorem \ref{thm:HS}, the Cayley graph $\Gamma(G,S)$ contains arbitrarily large simple cycles, and therefore arbitrarily large SAPs. Next, we can assume without loss of generality that every word in $\Fo$ has the same length, say $N\geq 1$. If $\Fo\subseteq \WP(G,S)$, take a SAP $W$ of length greater than $N+1$. Because SAPs contain no strict factors that belong to $\WP(G,S)$ and every cyclic permutation of the word defining a SAP is itself a SAP, the configuration $x=W^{\infty}$ does not contain any word from $\Fo$. Therefore $x\in \X_{\Fo}\setminus \Xs_{G,S}$, which is a contradiction.
	Suppose there are elements in $\Fo$ that are not in $\WP(G,S)$. As $\Fo$ contains forbidden patterns and $\X_{\Fo} = \Xs_{G,S}$, for every word $w\in\Fo$ there exists $N_{w}\in\N$ such that for every $u,v\in S^{N_w}$ the word $uwv$ contains a factor from $\WP(G,S)$. Let $M = \max_{w\in\Fo}N_{w}$ and take $W$ a SAP of length $2M+N+4$. Once again, because every SAP contains no strict factors that belong to $\WP(G,S)$ and every cyclic permutation of the word defining the SAP is itself a SAP, the configuration $x = W^{\infty}$ contains no factors from $\Fo$. Indeed, if there is a $w\in\Fo$ such that $w\factor W$ we can take the cyclic permutation of $W$ such that $w$ is at the middle. Thus, $w$ can be extended by words $u, v$ of length $M+1$ such that $uwv$ contains no factor in $\WP(G,S)$. As a consequence $x\in \X_{\Fo}\setminus \Xs_{G,S}$, which is a contradiction.
\end{proof}

As plain groups admit SFT skeletons, we have an effective procedure to calculate the connective constant of their Cayley graphs. As mentioned in Section~\ref{subsec:symbolic}, entropies of SFTs are non-negative rational multiples of logarithms of Perron numbers (see Theorem 4.4.4 \cite{lind2021introduction}). Thus, we can slightly improve Corollary 3.4 from \cite{gilch2017counting} in the case of (plain) groups.
\begin{corollary}
    Let $G$ be a plain group with $S$ a standard set of generators. Then, $\mu(G,S)$ is a non-negative rational power of a Perron number.
\end{corollary}

Let us sketch how to compute the connective constant using SFTs. Let $\Xs_{G,S}$ be the skeleton of the plain group $G = (\Conv_{i=1}^k G_i) \ast \F_m$ and $\Fo$ be the finite set of patterns defining it. Recall from Theorem~\ref{thm:SFT} that this set corresponds to the SAPs on each individual group $G_i$ as well as the words $ss^{-1}$ for all $s\in S$. Let $N$ be the length of the biggest word in $\Fo$. We can extend $\Fo$ to $\Fo'$ so that all words have length $N$. The Rauzy graph $R_N(G,S)$ of order $N$ of $\Xs_{G,S}$ is the finite directed graph whose vertices are labeled by the language of size $N$ of the skeleton $\mathcal{L}_N(\Xs_{G,S})$, and edges are labeled with $\mathcal{L}_{N+1}(\Xs_{G,S})$. There is an edge labeled by $w$ from $u$ to $v$ if $u$ is the prefix of length $N$ of $w$ and $v$ is the suffix of length $N$ of $w$. We denote the adjacency matrix of the graph $R_N$ by $M_N$, that is, if $\mathcal{L}_N(\Xs_{G,S})=\{ u_1,\dots,u_\ell\}$, the entry $M_N(i,j)$ represents the number of edges in $R_N$ from $u_i$ to $u_j$. Then the connective constant of $\Gamma(G,S)$ is the logarithm of the dominant eigenvalue of $M_N$, which exists by Perron-Frobenius' Theorem.

\begin{example}
Take $S_3$ the symmetric group on 3 elements with generating set $\tt{s}_1 = (1 \ 2)$ and $\tt{s}_2 = (1 \ 3)$, and the cyclic group $ \Z/3\Z = \langle \tt{t} \rangle$. Then, the skeleton of the plain group $G = S_3\ast  \Z/3\Z$ with respect to $S = \{\tt{s}_1, \tt{s}_2, \tt{t}^{\pm 1}\}$ is defined by the forbidden patterns,
$$\Fo = \{\tt{s}_1^2, \tt{s}_2^2, (\tt{s}_1\tt{s}_2)^3, \tt{t}^3, \tt{tt}^{-1}, \tt{t}^{-1}\tt{t}\}.$$
We obtain that the connective constant $\mu = \mu(G,\{\tt{s}_1,\tt{s}_2,\tt{t}^{\pm1}\})$ is the solution of the polynomial equation $x^7-4x^5-8x^4-8x^3-8x^2-8x-4=0$ obtained from the characteristic polynomial of the matrix described above, which is approximately $\mu\approx~2.8698315$.

%\todo{Figura of the Cayley graph: mejorar eso\dots}
\begin{figure}[H]
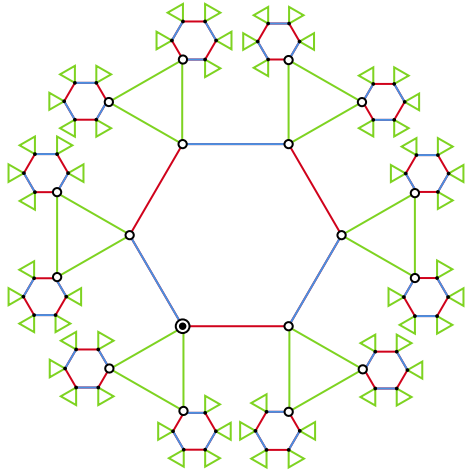

        \centering
        \include{figures/Cayley_graph_S3_Z3Z}
        \caption{A portion of the Cayley graph of the plain group $S_3\ast\Z/3\Z$. The two generators for $S_3$ are pictured in red and blue and alternate; the generator for $\Z/3\Z$ is pictured in green.}
        \label{fig:Cayley_graph_S3_Z3Z}
\end{figure}
\end{example}

The skeletons of plain groups with respect to their standard generating sets also have nice dynamical properties. We say a subshift $X\subseteq A^\Z$ is \define{irreducible} if for every $w_1, w_2\in\lang(X)$, there exists some $w\in\lang(X)$ such that $w_1ww_2\in\lang(X)$.

\begin{proposition}
    Let $G$ be a plain group with standard generating set $S$. Then, $\Xs_{G,S}$ is irreducible.
\end{proposition}
\begin{proof}
    Decompose $G$ as $\left(\Conv_{i=1}^k G_i\right) \ast \F_m$ with $S = S_1 \cup ... \cup S_{k+1}$ a standard generating set. Take $w_1,w_2\in\lang(\Xs_{G,S})$ appearing at position 0 of the configurations $x^{(1)},x^{(2)}\in \Xs_{G,S}$ respectively. There is a unique decomposition $w_i = w^i_{1} w^i_{2} .... w^i_{r_i}$ where,
    \begin{itemize}
        \item $w^i_{j}\in S_{l}^{+}$ for some $l$, for all $j\in\{1,\ ...,\ r_i\}$;
        \item $w^i_{j}$ and $w^i_{j+1}$ are words over different alphabets for all $j\in\{1,\ ...,r-1\}$.
    \end{itemize}
    If $w^1_{r_1}, w^2_1\in S^*_i$, take any generator $s\in S_j$ for $j\neq i$ and define $x = x^{(1)}_{(-\infty,0]}w_1sw_2x^{(2)}_{[|w_2|,+\infty)}$. Because we chose a generator that does not belong to $G_i$, and $x^{(1)}$ and $x^{(2)}$ belong to the skeleton, $x$ must also belong to the skeleton. This implies, $w_1sw_2\in\lang(\Xs_{G,S})$. If instead $w^1_r\in S^*_i$ and $w^2_1\in S^*_j$ for $i\neq j$, define
    $y = x^{(1)}_{(-\infty,0]}w_1ss'w_2x^{(2)}_{[|w_2|,+\infty)}$. As before, $y$ must belong to $\Xs_{G,S}$. This means, $w_1w_2\in\lang(\Xs_{G,S})$.
\end{proof}

\begin{corollary}
\label{cor:irreducible}
    For $G$ a plain group with standard generating set $S$, the set of periodic configurations of $\Xs_{G,S}$ is dense in $\Xs_{G,S}$. In other words, any bi-infinitely extendable SAW on $\Gamma(G,S)$ appears in a periodic bi-infinite SAW. Furthermore,
    $$\mu(G,S) = \lim_{n\to \infty} \sqrt[n]{e_n},$$
    where $e_n$ denotes the number of periodic points in $\Xs_{G,S}$ of period $n\in\N$.
\end{corollary}

This corollary states a general property of irreducible subshifts of finite type, namely, its set of periodic configurations is dense and its entropy is approximated through its periodic points~\cite{lind2021introduction}. We obtain a similar expression for the connective constants of Cayley graphs whose skeletons is not an SFT in Section~\ref{section.entropy_CC}.

%=-=-=-=-=-=-=-=-=-=-=-=-
\subsection{Effective Skeletons}

Let us briefly look at the case of effective skeletons. We know that recursively presented groups have recursively enumerable word problem. $\WP(G,S)$ is thus recursively enumerable for all finite generating sets. This enumeration gives us an enumeration of the forbidden patterns of our skeleton.

\begin{lemma}[\cite{aubrun2023domino}]
    Let $G$ be a recursively presented group. Then, for every generating set $S$, the subshift $\Xs_{G,S}$ is an effective subshift. In particular, $\mu(G,S)$ is a right computable real number.
\end{lemma}

In order to approach the converse, we give a computational upper bound of the word problem of the group in terms of the computability of finite SAWs on the Cayley graph.

\begin{lemma}
    The word problem for $G$ with respect to a generating set $S$ is decidable given an algorithm for $L_{SAW}(G,S)^c$.
\end{lemma}

\begin{proof}
	We describe a procedure to compute the word problem of $G$ given an algorithm that determines if a word belongs to $L = L_{SAW}(G,S)^c$. We begin with an algorithm that computes all words $w\in S^{\leq n}$ such that $w\in\WP(G,S)$ given $n$. This algorithm, which we call $\M$, is shown in Algorithm~\ref{alg:WP}.
	
	\begin{algorithm}[!ht]
		\caption{$\M$}
		\KwIn{$n \geq 2$}
		% \KwOut{$\WP(G,S)\cap S^n$}
		$T \gets \varnothing$\;
		\For{$w\in S^2$}{
			\If{$w\in L$}{
				$T\gets T\cup\{w\}$\;
			} 
		}
		\For{$i\in\{3, ..., n\}$}{
			\For{$w\in S^i$}{
				\If{$w\in L$}{
					\If{$w$ contains no factors from $T$}
					{
						$T\gets T\cup\{w\}$\;
					}
					\For{$v\in T$}{
						Delete $v$ from $w$ if present, to obtain $w'$\;
						\If{$w'\in T$}{
							$T\gets T\cup\{w\}$\;
		}}}}}
		\Return{T}\;
		\label{alg:WP}
	\end{algorithm}
	Let us show the output of $\M$ on $n$ is $\WP(G,S)\cap S^{\leq n}$. Let $T_i$ be the set $T$ in the algorithm after the first $i$ iterations of the \textbf{for} loop, for $i\in\{2, ..., n\}$. We claim $T_i = \WP(G,S)\cap S^{\leq i}$. First off, every non-self avoiding path of length two must represent the identity. Thus, $T_2 = \WP(G,S)\cap S^2$. Now, suppose we have the equality for $T_i$. Take $w\in\WP(G,S)\cap S^{i+1}$. This implies $w\in L$, and as seen in Lemma~\ref{lem:ouro}, it must represent a simple cycle, contain a shorter simple cycle, or a word of the form $ss^{-1}$. In the first case, $w$ contains no factors from $T_i$ and is therefore added to $T_{i+1}$. In the other two cases, it contains a factor from $T_i$ that after being deleted creates a word that belongs to $\WP(G,S)\cap S^{\leq i} = T_i$. Therefore, $w\in T_{i+1}$. Conversely, if $w'\in T_{i+1}\setminus T_i$ we know $w\in L$. If $w$ was added to $T_{i+1}$ because it contains no factors from $T_i$, it must represent a simple loop and is therefore in $\WP(G,S)$. On the other hand, if $w$ was added after deleting a factor from $T_i$, $w$ is made up of a word representing the identity with a factor representing the identity inserted into it. This means, $w\in\WP(G,S)$ and therefore $T_{i+1} = \WP(G,S)\cap S^{\leq i+1}$.
	
	Finally, to determine if a given word $w$ belongs to $\WP(G,S)$, we run $\M$ on the input $|w|$, and see if it is present in $T$.
\end{proof}

As a consequence, if $L_{SAW}(G,S)$ is co-recursively enumerable, the word problem of $G$ must be in $\Delta^0_2$ on the arithmetical hierarchy. This is the case when $\Xs_{G\ast H, S}$ is effective for $H$ any f.g. group, and $S = S_G\cup S_H$ with $S_G, S_H$ generating sets for $G$ and $H$ respectively. 

\begin{conjecture}
    A group is recursively presented if and only if there exists a finite generating set $S$ such that $\Xs_{G,S}$ is effective.
\end{conjecture}

Even though recursively presented groups define subshifts that are effective, if the structure of the underlying group is computationally complex, the configurations of the skeleton may be uncomputable. We say a configuration $x\in S^\Z$ is \define{computable} if there is an algorithm that on input $n\in\Z$ computes $x_n\in S$.

\begin{definition}
    A finitely generated group $G$ and generating set $S$ are said to be \define{algorithmically finite} if for every infinite recursively enumerable set $L\subseteq\F_{S}$ , there exist infinitely many pairs of distinct words $u,v\in L$ such that $\pi(u) = \pi(v)$, where $\pi:\F_S\to G$ is the canonical projection. We say $G$ is a \define{Dehn Monster} if it is infinite, recursively presented and algorithmically finite.
\end{definition}

This class of groups was introduced by Myasnikov and Osin in \cite{myasnikov2011algorithmically}, where they showed that Dehn Monsters exist. Furthermore, they showed that being algorithmically finite does not depend on the generating set.

\begin{proposition}
\label{prop:monster}
Let $G$ be a Dehn Monster. Then $\Xs_{G,S}$ is effective for any finite generating set $S$, but no configuration in $\Xs_{G,S}$ is computable.
\end{proposition}

\begin{proof}
    As the properties of being infinite, recursively presented and algorithmically finite are independent of the generating set, we take any generating set $S$ for $G$. If there existed a computable configuration $x\in \Xs_{G,S}$, we could recursively enumerate the set of words $L = \{x_{[0,n-1]}\in S^*\mid n\geq 1\}$. Then for any $u,v\in L$, $\pi(u)\neq\pi(v)$. If not, we would arrive at $x_{[0,n-1]} =_G x_{[0,m-1]}$ for some $n > m\geq 1$, which implies $x_{[n,m-1]} =_G \epsilon$. Therefore, any pair of elements in $L$ maps to a different element through $\pi$, which contradicts the algorithmic finiteness of $G$.
\end{proof}

%=-=-=-=-=-=-=-=-=-=-=-=-=-=-=-=-=-=-=-=-=-=-=-=-=-=-=-=-=-=-=-=-=-=-=-=-=-
\subsection{Periodic Configurations}

Configurations of particular importance in the study of subshifts are periodic configurations. Recall that a configuration $x\in\Xs_{G,S}$ is periodic if there exists $k\in\Z\setminus\{0\}$ such that $x_{i+k} = x_i$ for all $i\in\Z$. Such a configuration has a rigid structure, if we take $w = x_{[0,k-1]}$ the configuration $x$ is equal to the bi-infinite repetition of $w$, i.e., $x = w^{\infty}$. We will see that the existence of periodic points in the skeleton imposes strong restrictions on the structure of the underlying group. 

In~\cite{aubrun2023domino} it was shown that for any finitely generated group with a torsion-free element, the skeleton contains a periodic point. The periodic configuration was obtained by iterating any geodesic of the torsion-free element with the smallest length in the group. By re-interpreting the proof of~\cite[Theorem 7]{halin1973automorphisms} we obtain the following generalization.

\begin{proposition}
\label{prop:torsion-free}
    Let $G$ be a finitely generated group. Take a generating set $S$ and a torsion-free element $g\in G$, and
    $k = \argmin\{\|g^n\|_S \mid n\geq 1\}$. Then, for any geodesic $w\in S^*$ representing $g^k$, $w^{\infty}\in\Xs_{G,S}$ and is a bi-infinite geodesic. 
    % Furthermore, there exists $w'\in\lang(\Xs_{G,S})$ such that $g = \overline{w'}$.
\end{proposition}

\begin{proof}
    Fix a generating set $S$ and $g\in G$ torsion-free. Let $k\geq 1$ be as in the statement of the result, and denote $h = g^k$. Take a geodesic $w\in S^*$ for $h$ and let $\pi = (e_0, ..., e_{n-1})$ be the (self-avoiding) walk starting at the identity in the Cayley graph, of label $w$ and $|w| = n$. Let $\Pi$ be the bi-infinite walk made by concatenating the paths $h^m\cdot\pi$, for all $m\in\Z$. Thus, $\lambda(\Pi) = w^{\infty}$. We claim $\Pi$ is self-avoiding. Suppose it is not, and take the smallest $m\in\N$ such that $w^m$ does not represent a SAW. Let $\pi_i$ denote the walk $h^i\cdot\pi$ where $\init(\pi_i) = h^{i}$ and $\ter(\pi_i)=h^{i+1}$. As $m$ is minimal, we know the concatenated walks $\pi_0\ ...\ \pi_{m-1}$ and $\pi_{1}\ ...\ \pi_{m}$ are self-avoiding, and therefore the first intersection must occur between $\pi_0$ and $\pi_m$. Then, there exists $v,u\factor w$ prefixes, and $f\in\pi_0\cap\pi_{m}$ such that $f = \overline{v} = h^{m-1}\overline{u}$. Once again, because $m$ is minimal, $f\neq h, h^{m-1}$. If we compute the distance,
    $$d_S(f, h^{m-1}f) = d_S(\overline{u}, \overline{v})\leq |w|,$$
    as $k$ was chosen to minimize $\|g^{k}\|_{S}$, the distance between $f$ and $h^{m-1}f$ must be $|w|$. As both vertices are in $\pi_m$, this is only possible if $f = h^{m}$ and $h^{m-1}f = h^{m-1}$. Thus $h^{m} = 1_G$, which is a contradiction as $h$ is torsion-free. Therefore, $w^{\infty}\in\Xs_{G,S}$. 
    Finally, as we chose $k$ to minimize the distance to the identity of powers of $g$, $w^n$ must be a geodesic for all $n\in\N$. 
\end{proof}

\begin{theorem}[\cite{aubrun2023domino}]
\label{thm:aperiodic}
Let $G$ be a finitely generated group. Then, $G$ is a torsion group if and only if $\Xs_{G,S}$ is aperiodic for every (any) generating set.
\end{theorem}

\begin{proof}
    Suppose $G$ is a torsion group and let $x\in \Xs_{G,S}$ be a periodic configuration that infinitely repeats the word $w$. Let $g = \overline{w}$. By definition of the skeleton, $g^n = \overline{w^n} \neq 1_G$ for all $n\in\N$. This contradicts the fact that $G$ is a torsion group. Conversely, if $G$ has a torsion-element, by Proposition \ref{prop:torsion-free}, $\Xs_{G,S}$ contains a periodic point.
\end{proof}

\begin{corollary}
\label{cor:not_sofic_torsion}
    If $G$ is a finitely generated torsion group, then for all generating sets $\Xs_{G,S}$ is not sofic.
\end{corollary}

\begin{proof}
    If $G$ is a finitely generated torsion group, Theorem \ref{thm:aperiodic} tells us that none of its skeletons contain periodic configurations. Because non-empty sofic shifts always contain periodic configurations, no skeleton of $G$ can be sofic.
\end{proof}

%=-=-=-=-=-=-=-=-=-=-=-=-=-=-=-=-=-=-=-=-=-=-=-=-=-=-=-=-=-=-=-=-=-=-=-=-=-
\subsection{Minimality}
\label{subsec:minimality}

Our next objective is to find sufficient and necessary properties for the skeleton to be minimal. We begin by identifying possible subshifts of $\Xs_{G,S}$.

\begin{lemma}
\label{lem:sk_quotient}
    Let $G$ be a finitely generated group. Then,
    \begin{itemize}
        \item For a symmetric subset $S'\subseteq S$ and $H = \langle S'\rangle$, $\Xs_{H,S'}$ is a subshift of $\Xs_{G,S}$.
        \item For $N \trianglelefteq G$ a normal subgroup, $\Xs_{G/N, S'}$ is a subshift of $\Xs_{G,S}$, where $S' = S\setminus N$.
    \end{itemize}
\end{lemma}

\begin{proof}
    The first statement follows from the fact that any configuration from $\Xs_{H,S'}$ avoids all words from $\WP(G,S)$, as $H$ is a subgroup of $G$. For the second statement, let $x\in \Xs_{G/N, S}$ and $\{w_i\}_{i}\subseteq S^*$ a set of generators for $N$. Then by definition no factor $w\factor x$ belongs to $\WP(G/N, S) = \llangle \WP(G,S)\cup \{w_i\}\rrangle$. In particular, it does not belong to $\WP(G,S)$. Therefore $x\in \Xs_{G,S}$.
\end{proof}

Because every non-finite quotient gives us a non-empty subshift of $\Xs_{G,S}$, if we want to find a minimal skeleton, it is reasonable to look at the class of just infinite groups. A group $G$ is said to be \define{just infinite} if it is infinite and every proper quotient is finite.

\begin{proposition}
    \label{prop:minimal}
    Let $G$ be a finitely generated group with a generating set $S$. If $\Xs_{G,S}$ is minimal, then $G$ is a just infinite group.
\end{proposition}

\begin{proof}
    If $\Xs_{G,S}$ is minimal, every subshift of the form $\Xs_{G/N, S}$ must be either empty or equal to $\Xs_{G,S}$. Let $N$ be a proper normal subgroup, that is, non trivial and not equal to $G$. By Theorem \ref{thm:bounds_GL}, the connective constants satisfy $\mu(G/N,S) < \mu(G,S)$. Thus, the entropy of $\Xs_{G/N,S}$ is strictly less than that of $\Xs_{G,S}$, so they cannot be equal. Then $\Xs_{G/N, S} =\varnothing$, meaning $G/N$ is finite. Therefore, $G$ is just infinite.
\end{proof}

\begin{proposition}
    Let $G$ be a finitely generated group with a generating set $S$. If $\Xs_{G,S}$ is minimal, for every symmetric subset $S'\subsetneq S$, the subgroup $\langle S'\rangle$ is finite. In particular, torsion-free groups do not admit minimal skeletons.
\end{proposition}

\begin{proof}
If $\Xs_{G,S}$ is minimal, every subshift of the form $\Xs_{H, S'}$, for $H=\langle S'\rangle$, must be either empty or equal to $\Xs_{G,S}$. If $H = G$, then by Theorem~\ref{thm:bounds_GL}, $\mu(G,S') < \mu(G,S)$ meaning $\Xs_{G, S'}$ is empty, which is a contradiction. Therefore, $H\lneq G$. Now, take $s\in S\setminus H$ and $x\in\Xs_{H,S'}$. Define the configuration $x' = x_{(-\infty,-1]}sx_{[0,+\infty)}\in S^{\Z}$. Because $x$ is in $H$'s skeleton, we know neither $x_{(-\infty,-1]}$ nor $x_{[0,+\infty)}$ contain subwords from $\WP(G,S)$. Next, if there exist $i,j\in\N$ such that $x_{[i,-1]}sx_{[0,j]}\in\WP(G,S)$, then $s=_G~\!(x_{[i,-1]})^{-1}(x_{[0,j]})^{-1}$ which implies $s\in H$. This is a contradiction. Therefore, $x'\in\Xs_{G,S}\setminus\Xs_{H,S'}$. As $\Xs_{G,S}$ is minimal, $\Xs_{H,S'}=\varnothing$ and thus $H$ is finite. Finally, if a group is torsion-free, each generator generates $\Z$ which is not possible if the skeleton is minimal.
\end{proof}

\begin{remark}
    Both conditions are not sufficient to characterize minimal skeletons. Take the group $\tilde{A}_2$ with generating set $\{\tt{a}, \tt{b}, \tt{c}\}$ as defined in Section~\ref{subsec:SAWs}. This group is just infinite~\cite{moller2024quotients}, every pair of different generators generates a subgroup isomorphic to the finite group $S_3$, and every generator generates a copy of $ \Z/2\Z$. Nonetheless, its skeleton is not minimal. Take the periodic configuration $x = (\tt{abcb})^\infty$ which belongs to the skeleton. Then, the closure of the orbit of $x$ is finite and contains exactly periodic configurations defined by cyclic permutations of $\tt{abcb}$. But, the skeleton also contains the periodic configuration $y = (\tt{bcac})^\infty$, which is not one of the cyclic permutations.
\end{remark}

As the remark shows, if a minimal skeleton contains periodic configurations, it must be finite. This is the case of $\D_{\infty}$ with generating set $\{\tt{a},\tt{b}\}$, as seen on Example~\ref{ex:dihedral}, which defines a minimal skeleton.

%=-=-=-=-=-=-=-=-=-=-=-=-=-=-=-=-=-=-=-=-=-=-=-=-=-=-=-=-=-=-=-=-=-=-=-=-=-
%
%=-=-=-=-=-=-=-=-=-=-=-=-=-=-=-=-=-=-=-=-=-=-=-=-=-=-=-=-=-=-=-=-=-=-=-=-=-
\section{Sofic Skeletons}
\label{section.sofic_skeleton}

Let us tackle the question of which groups admit skeletons that are sofic. Since SFTs are sofic subshifts, from Theorem \ref{thm:SFT} we already know that plain groups admit sofic skeletons. But are there groups that admit sofic skeletons which are not SFTs? The first naive strategy would be to ask when the word problem of the group is regular, as this is the set of forbidden patterns used in the definition of the skeleton. Unfortunately, Anisimov showed in~\cite{anisimov1971regular} that $\WP(G,S)$ is regular if and only if $G$ is a finite group. We must therefore find other sets of forbidden patterns to study. Lemma \ref{lem:LSAW} tells us that we can look at the classes of groups where the language of SAWs is regular. The class of groups with such property have already been classified.

\begin{theorem}[\cite{lindorfer2020language}]
\label{thm:lindorfer_woess}
    Let $G$ be a f.g. group with $S$ a finite generating set. Then, $L_{\SAW}(G,S)$ is regular if and only if $\Gamma(G,S)$ has more than one end and all ends are thin of size 1.
\end{theorem}
As Lindorfer and Woess show, if $\Gamma(G,S)$ has only thin ends of size 1 its blocks are finite~\cite[Lemma 5.3]{lindorfer2020language}. Combining this fact with Haring-Smith's characterization of plain groups (Theorem \ref{thm:HS}), we see that groups where $L_{\SAW}(G,S)$ is regular are exactly plain groups. Nevertheless, when considering bi-infinitely extendable SAWs, the situation is different.
\begin{lemma}
\label{lem:ZyZ2}
    The group $G = \Z\times \Z/2\Z$ given by the presentation $\langle s,t\mid s^2, [t,s]\rangle$ has a sofic skeleton.
\end{lemma}

\begin{proof}
    We will exhibit a regular set of forbidden patterns for $\Xs_{G,S}$, with $S = \{s^{\pm 1},t^{\pm 1}\}$. Take the set of forbidden patterns
    $$\Fo = \{st^{\pm n}st^{\mp1}\mid n\in\N\}\cup \{t^{\pm1}st^{\mp n}s\mid n\in\N\}\cup\{s^2, t^{\pm1}t^{\mp1}\}.$$
    It is a simple exercise to show that $\Fo$ is a regular language. Let us show $\Xs_{G,S} = \X_{\Fo}$. Suppose there is a configuration $x\in \Xs_{G,S}\setminus\X_{\Fo}$. Because $x$ is in the skeleton, we know it does not contain factors of the form $s^2$ or $t^{\pm 1}t^{\mp1}$. Therefore it must contain a factor of the form $st^{\pm n}st^{\mp1}$ or $t^{\pm 1}st^{\pm n}s$. Suppose, $x$ contains the word $w=st^nst^{-1}$, for some $n\in\N$. There is no way to extend this word to the right, as $ws$ contains the factor $tst^{-1}s$ which evaluates to the identity, extending by $t^{-k}$ with $k\geq n$ creates the factor $st^nst^{-n}$ which evaluates to the identity, and extending by $t^{-k}s$ with $k\leq n-1$ creates the factor $t^{k}st^{-k}s$ which also evaluates to the identity. This leads to a contradiction. The other cases being analogous, we have $\Xs_{G,S}\subseteq \X_{\Fo}$.

    Now, suppose there is a configuration $x\in \X_{\Fo}\setminus \Xs_{G,S}$. By Lemma \ref{lem:ouro} and the definition of $\Fo$, $x$ must contain a SAP. Nevertheless, all SAPs in $G$ are cyclic permutations of words of the form $st^{n}st^{-n}$ for some $n\in\N$. Thus, each SAP contains a factor from $\Fo$, leading to a contradiction and proving $\X_{\Fo}\subseteq \Xs_{G,S}$.
\end{proof}

The Cayley graph of $\Z\times \Z/2\Z$ with respect to the before mentioned generating set is the bi-infinite ladder, which is a graph with two thin ends of size 2. An analogous proof can be done for the Cayley graph of $\Z$, which is a plain group, with respect to the generating set $\{\pm1, \pm2\}$, which also has thin ends of size 2.

To characterize groups which admit sofic skeletons we will make use of the fact that the language of a sofic subshift is regular. Our main tool in this regard will be the following version of the Pumping Lemma.

\begin{lemma}[Pumping Lemma]
    Let $L$ be a regular language. Then, there exists $p>0$ such that every word $w\in L$ with $|w|\geq p$ can be decomposed as $w = w'uv$ with $|u|>1$ and $|uv|\leq p$, such that for all $n\in\N$, $w'u^nv\in L$.
\end{lemma}

This allows us to show that being sofic is a property of skeletons that depends on the generating set.
\begin{proposition}
\label{prop:not_sofic}
        Every group $G$ admits a generating set $S$ such that $\Xs_{G,S}$ is not sofic. 
\end{proposition}

\begin{proof}
    By Corollary \ref{cor:not_sofic_torsion}, if $G$ is a torsion group, no skeleton is sofic. We can therefore suppose $G$ has a torsion-free element. Let $S'$ be any generating set for $G$, and $g$ a torsion-free element. We denote $s = g^2$, $t = g^3$, and define $S = S'\cup\{s,t\}$. Suppose $\Xs_{G,S}$ is sofic. Then, its language $\lang(\Xs_{G,S})$ is regular. Take $p > 0$ given by the pumping lemma. The word $w = ts^{p+1}t^{-1}s^{-p}$ is contained in $\lang(\Xs_{G,S})$ as it is globally admissible through the configuration $s^{\infty}ts^{p+1}t^{-1}s^{-p}t^{-1}(s^{-1})^{\infty}$ (see Figure \ref{fig:not_sofic}).

    \begin{figure}[!ht]
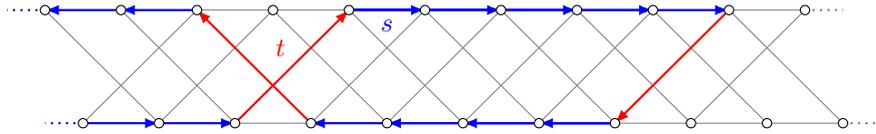

        \centering
        \include{figures/not_sofic}
        \caption{The configuration used for the Pumping Lemma (with $p=4$) depicted in the Cayley graph of the subgroup $\langle s,t\rangle$. The blue edges represent $s$ and the red edges $t$.}
        \label{fig:not_sofic}
\end{figure}
Now, by the Pumping Lemma we can decompose $w$ as $w = w'uv$ with $|uv|\leq p$. Thus, $u = s^{-k}$ with $k\geq 1$. Therefore, the word $w'u^2v=ts^{p+1}t^{-1}s^{-(p+k)}$ belongs to $\lang(\Xs_{G,S})$, which is a contradiction as $ts^{p+1}t^{-1}s^{-(p+1)} =_G \epsilon$. We conclude that $\Xs_{G,S}$ is not sofic.
\end{proof}

%=-=-=-=-=-=-=-=-=-=-=-=-=-=-=-=-=-=-=-=-=-=-=-=-=-=-=-
\subsection{Ends and Automorphisms}
\label{subsec:EndsAndAuts}

To go towards a characterization we must make a brief detour through the theory of ends and automorphisms of infinite quasi-transitive graphs.
Let us begin by taking a look at the theory of ends of connected graphs as introduced by Halin~\cite{halin1964unendliche}.

For a connected graph $\Gamma = (V_{\Gamma},E_{\Gamma})$, and a subset of vertices $A\subseteq V_{\Gamma}$ we denote by $\Gamma\setminus A$ the graph obtained by removing the vertices from $A$ and all their incident edges. We define a \define{ray} $\rho$ to be an infinite sequence of distinct vertices $\pi = (v_0, v_1, ...)\in V_{\Gamma}^{\N}$ such that there is an edge between $v_i$ and $v_{i+1}$. Analogously, a \define{double ray} to be a bi-infinite sequence of distinct vertices $\pi = (..., v_{-1},v_0,v_1, ...)\in V_{\Gamma}^{\Z}$ such that each successive vertex is connected by an edge. Two rays are said to be \define{equivalent} if for any finite set $A\subseteq V_{\Gamma}$ all but finitely many of their vertices are contained in the same connected component of $\Gamma\setminus A$. The equivalence classes of this relation are called the \define{ends} of the graph. Given an end $\omega$ and a finite set $A\subseteq V_{\Gamma}$, we define $C(\omega, A)$ to be the connected component of $\Gamma\setminus A$ where all the rays defining $\omega$ eventually end up in.

A \define{defining sequence} for an end $\omega$ is a sequence of finite subsets $(A_i)_{i\in\N}$ such that for all $i\geq 1$, $A_i\cup C(\omega,A_i)\subseteq C(\omega, A_{i-1})$. We say that an end $\omega$ is \define{thin} if there exist $m\geq 1$ and a defining sequence $(A_i)_{i\in\N}$ such that $|A_i|=m$ for all $i\in\N$. The smallest $m$ verifying this condition is called the \define{size} of $\omega$. An end is called \define{thick} if its size is infinite. Thomassen and Woess \cite{thomassen1993vertex} showed using Menger's Theorem that an end of size $m\in\N\cup\{\infty\}$, seen as an equivalence class of rays, contains a maximum of $m$ vertex disjoint rays.

Let $\Aut(\Gamma)$ denote the set of automorphisms of $\Gamma$, that is, bijections $f:V_{\Gamma}\to V_{\Gamma}$ that preserve edge adjacency. We say a subgroup $G\leq \Aut(\Gamma)$ acts \define{quasi-transitively} on $\Gamma$ if the set of orbits of the action $G\curvearrowright \Gamma$ is finite. We say $G$ acts \define{transitively} if there is a unique orbit. Freudenthal and Hopf independently showed~\cite{freudenthal1944enden,hopf1943enden} that a quasi-transitive graph has either $0$, $1$, $2$ or an infinite amount of ends.

In our setting, all Cayley graphs $\Gamma(G,S)$ are transitive under the action of the group $G$ by left translations. Furthermore, this action preserves the labeling given by the generating set.

Take $\Gamma$ to be locally finite and connected. Following~\cite{halin1973automorphisms}, automorphisms of $\Gamma$ can be classified into three classes. An automorphism $g\in\Aut(\Gamma)$ is,
\begin{itemize}
    \item \define{elliptic} if it fixes a finite subset of $V_{\Gamma}$,
    \item \define{parabolic} if it fixes a unique end, and
    \item \define{hyperbolic} if it fixes a unique pair of ends.
\end{itemize}

Halin showed~\cite{halin1973automorphisms} that for a non-elliptic automorphism $g\in\Aut(\Gamma)$ and vertex $v\in V$ the sequence $(v, g\cdot v, g^2\cdot v, ...)$ uniquely defines and fixes an end which we call the \define{direction} of $g$, and denote $D(g)$.

\begin{theorem}[Halin, \cite{halin1973automorphisms} Theorem 9]
\label{thm:halin}
    Let $g$ be a non-elliptic automorphism acting on a connected locally finite graph $\Gamma$. Then,
    \begin{itemize}
        \item $D(g)$ and $D(g^{-1})$ have the same size $m$.
        \item $D(g)\neq D(g^{-1})$ if and only if $m<\infty$. In this case $g$ is hyperbolic.
        \item There are $m$ disjoint double rays $\{\pi_i\}_{i=1}^m$ that are invariant by some positive power of $g$.
        \item If $g$ is hyperbolic, there exists a set $A\subseteq V_{\Gamma}$ of size $m$ and $k\in\N$ such that $(g^{kn}\cdot A)_{n\in\N}$ and $(g^{-kn}\cdot A)_{n\in\N}$ are defining sequences for $D(g)$ and $D(g^{-1})$ respectively, that intersect each $\pi_i$ in exactly one vertex.
    \end{itemize}
\end{theorem}

To precisely understand thin ends, we study the following graphs. 
\begin{definition}
    A connected locally finite graph is called a \define{strip} if it is two ended and quasi-transitive.
\end{definition}

We present general facts about strips that can be found in \cite{lindorfer2020language} and can be partly deduced from Theorem \ref{thm:halin}. For every strip $Q$, there exits a hyperbolic automorphism $g\in\Aut(Q)$ that fixes both ends $\omega^+$ and $\omega^-$. Both ends have the same size, for instance $m$, which entails the existence of a finite set $A$ of size $m$ such that $(g^n\cdot A)$ and $(g^{-n}\cdot A)$ are defining sequence for $\omega^+$ and $\omega^-$ respectively. In addition, there are $m$ disjoint double rays intersecting every $g^n\cdot A$ at exactly one vertex. We call such a strip a $g$-strip of size $m$. When working with a $g$-strip, up to taking a power of $g$, we can assume that the subgraph induced by $C(\omega^+, A)\setminus C(\omega^+, g\cdot A)$, which we call $P(\omega^+)$, is connected and finite.

The following results show that quasi-transitive graphs contain strips, under conditions on their ends and automorphisms.

\begin{lemma}[Lindorfer, Woess, \cite{lindorfer2020language} Lemma 3.3]
\label{lem:ends_to_strip}
    Let $\Gamma$ be a connected and locally finite graph where $G\leq\Aut(\Gamma)$ acts quasi-transitively. If $\Gamma$ has a thin end of size $m$, then it contains a $g$-strip of size $m$ for some $g\in G$.
\end{lemma}

\begin{lemma}[Lindorfer, Woess, \cite{lindorfer2020language} Lemma 3.4]
\label{lem:parabolic_to_strip}
    Let $\Gamma$ be a connected and locally finite graph where $G\leq\Aut(\Gamma)$ acts quasi-transitively. If $G$ contains a parabolic element, then for every $m\geq 1$, $\Gamma$ contains a $g$-strip of size at least $m$ for some $g\in G$.
\end{lemma}

%=-=-=-=-=-=-=-=-=-=-=-=-=-=-=-=-=-=-=-=-=-=-=-=-=-=-=-
\subsection{Characterizing Sofic Skeletons}

We provide the following characterization.
\begin{theorem}[Theorem~\ref{intro:sofic}]
\label{thm:sofic}
    Let $G$ be a finitely generated group. There exists $S$ such that $\Xs_{G,S}$ is sofic if and only if $G$ is a plain group, $\Z\times  \Z/2\Z$ or $\D_{\infty}\times \Z/2\Z$.
\end{theorem}

The idea of the proof is as follows. First, we use the same constructions of Lindorfer and Woess~\cite{lindorfer2020language} to find ladder-like structures on strips that will allow us to use the Pumping Lemma, and then conclude that all ends of the graph must be thin and of size at least 2. Next, by using similar ideas, we show that if the graph has an end of size two and the skeleton is sofic, then the group must be virtually $\Z$. Finally, we characterize virtually $\Z$ groups with sofic skeletons, completing the proof.

\begin{lemma}
\label{lem:2strip}
    Let $G$ be a finitely generated group with a generating set $S$, such that $\Gamma(G,S)$ contains an $g$-strip $Q$ for some $g\in G$. If $\Xs_{G,S}$ is sofic, then $Q$ has size at most $2$.
\end{lemma}

\begin{proof}
    Suppose $Q$ is of size greater or equal than $3$. Then, $Q$ contains three disjoint double rays which we call $\pi_1 = (v_i)_{i\in\Z}$, $\pi_2 = (u_i)_{i\in\Z}$ and $\pi_3 = (v'_i)_{i\in\Z}$, that are $g$-invariant.  Recall we took our subgraph $P(\omega^+)$ to be connected and finite. Therefore, there is a path $p_1$ that connects two of the rays. Suppose without loss of generality that $p_1$ connects $\pi_1$ and $\pi_2$ from $v_0$ to $u_0$ with no other vertices from $\pi_i$ for $i\in\{1,2,3\}$. Analogously, $g\cdot P(\omega^+)$ will connect $\pi_3$ with another of the rays through a path $p_2$. Up to rearranging indices, suppose $p_2$ connects $\pi_2$ to $\pi_3$ starting at $u_k$ and ending at $v'_k$, for some $k\in\N$ such that there are no other vertices from $\pi_i$ for $i\in\{1,2,3\}$. Because the vertex set of every element of the sequence $(g^n\cdot P(\omega^+))_{n\in\N}$ is pairwise disjoint, no walks in $\{g^{2n}\cdot p_1 \mid n\in\Z\}\cup\{g^{2n}\cdot p_2 \mid n\in\Z\}$ intersect. This way, the subgraph induced by the three paths $\{\pi_i\}_{i=1}^3$ and all $g^{2n}\cdot p_1$ and $g^{2n}\cdot p_2$, $Q'\subseteq Q$ is a periodic subdivision of the bi-infinite 3-ladder (see Figure~\ref{fig:3ladder}).

    \begin{figure}[!ht]
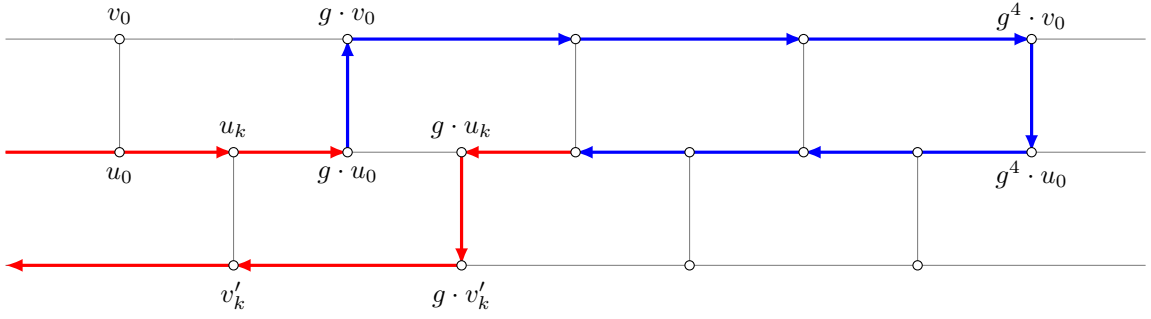

        \centering
         \includestandalone[scale=1]{figures/3ladder}
        \caption{The periodic subdivision of the $3$-ladder with the configuration $x$ highlighted in red and blue. The word $\lambda_3 \lambda_4^3 \lambda_3^{-1} \lambda_1^{-2}$ is marked in blue, whereas the infinite prefix and suffix of $x$ are marked in red.}
        \label{fig:3ladder}
    \end{figure}
    
    Now, let us give names to the labels of the different portions of the subdivision. Denote $\lambda_1$ the label from $u_k$ to $g\cdot u_k$, $\lambda_2$ the label from $u_k$ to $g\cdot u_0$, $\lambda_3$ the label from $u_0$ to $v_0$, $\lambda_4$ the label from $v_0$ to $g\cdot v_0$, $\lambda_5$ the label from $u_k$ to $v'_k$, and finally $\lambda_6$ the label from $g\cdot v'_k$ to $v'_k$. Then, for every $n\geq 1$ and $k<n$ the configuration
    $$x = \lambda_1^{\infty}\lambda_2.\lambda_3 \lambda_4^n \lambda_3^{-1} \lambda_1^{-k} \lambda_2^{-1}\lambda_5 \lambda_6^{\infty},$$
    belongs to the skeleton (See Figure~\ref{fig:3ladder}). Thus, $\lambda_3 \lambda_4^n \lambda_3^{-1} \lambda_1^{-k}\in\lang(\Xs_{G,S})$ for every $n\geq 1$ and $k<n$. Notice that the language $L = \{\lambda_3 \lambda_4^n \lambda_3^{-1} \lambda_1^{-k}\in S^* \mid k,n\in\N\}$ is regular. If $\Xs_{G,S}$ is sofic, its language $\lang(\Xs_{G,S})$ is regular. Then by the closure properties of regular languages, 
    $$L' = L\cap\lang(\Xs_{G,S})= \{\lambda_3 \lambda_4^n \lambda_3^{-1} \lambda_1^{-k}\in S^* \mid k<n\},$$
    is regular. By the Pumping Lemma, there exists a pumping length $p>0$. Take $\lambda_3 \lambda_4^{p+1} \lambda_3^{-1} \lambda_1^{-p}\in L'$. This word decomposes as $\tilde{w}ww'$ such that $|ww'|\leq p$. By the structure of our word, $ww'$ is a suffix of $\lambda_1^{-p}$. Next, $\tilde{w}w^2w'$ belongs to $L'$ and therefore has the form 
    $$\tilde{w}w^2w' = \lambda_3 \lambda_4^n \lambda_3^{-1} \lambda_1^{-k} = \lambda_3 \lambda_4^{p+1} \lambda_3^{-1}w_1w^2w',$$
    for some $k,l\in\N$ and $w_1\in S^*$. Because we are working over a Cayley graph, the labels of different edges starting from $u_0$ must be different and thus the first generators for $\lambda_4$ and $\lambda_3^{-1}$ are different. Therefore, $n = p+1$. This means, $\lambda_1^{-k} = w_1w^2w'$. Finally, as $\lambda_1^{-k}$ is strictly longer than $\lambda_1^{-p}$, $k \geq p+1$. But, this would imply $\lambda_3 \lambda_4^{p+1} \lambda_3^{-1} \lambda_1^{-k}$ belongs to $\lang(\Xs_{G,S})$ and is not self-avoiding, which is a contradiction.

\end{proof}

\begin{proposition}
        Let $G$ be a finitely generated group. If there exists $S$ such that $\Xs_{G,S}$ is sofic, then $G$ has more than one end, and $\Gamma(G,S)$ only has thin ends of size at most 2. 
\end{proposition}

\begin{proof}
    Let $G$ be a finitely generated group with generating set $S$ such that $\Xs_{G,S}$ is sofic. By Theorem \ref{thm:aperiodic}, $G$ is not a torsion group and therefore contains non-elliptic elements when seen as a subgroup of $\Aut(\Gamma(G,S))$. If $G$ is one-ended, then $\Gamma(G,S)$ has one end, which by Lemmas \ref{lem:parabolic_to_strip} and \ref{lem:2strip} is a contradiction. Thus, $\Gamma(G,S)$ has at least one thin end. By Lemma \ref{lem:ends_to_strip}, every thin end of size $m$ implies the existence of a strip of size $m$ in $\Gamma(G,S)$. By Lemma \ref{lem:2strip}, these strips -- and consequently their corresponding ends -- must have size at most $2$. Finally, if $\Gamma(G,S)$ had a thick end, from the proof of Theorem 4.1 in \cite{lindorfer2020language} we know it contains a one-ended subgraph. As before, this contradicts Lemmas~\ref{lem:parabolic_to_strip} and \ref{lem:2strip}.
\end{proof}

The converse of this proposition is not true: the group $\F_2\times \Z/2\Z$ along with the generating set $S = \{a^{\pm 1},b^{\pm1}, s\}$, given by the presentation $\langle a,b,s \mid s^2, [a,s], [b,s]\rangle$, has thin ends of size two, but its skeleton is not sofic. Similar to what we did in Proposition \ref{prop:not_sofic}, we can use the Pumping Lemma on the words $sa^{n+1}sa^{-n}$, with $n\in\N$, which are in $\lang(\Xs_{G,S})$ through the configuration $b^{\infty}sa^{n+1}sa^{-n}b^{\infty}$. The next Lemma captures this idea in the general setting.

\begin{lemma}
\label{lem:2ends_Z}
    Let $G$ be a finitely generated group. If there exists $S$ such that $\Xs_{G,S}$ is sofic and $\Gamma(G,S)$ has an end of size 2, then $G$ is virtually $\Z$. 
\end{lemma}

\begin{proof}
Suppose $\Gamma(G,S)$ has more than two ends, and take $\omega^+$ the end of size $2$. By Lemma \ref{lem:ends_to_strip}, there exists $g\in G$ and $Q$ a $g$-strip of size $2$. Then, there exist two $g$-invariant disjoint double rays $\pi_1 = (v_i)_{i\in\Z}$ and $\pi_2 = (u_i)_{i\in\Z}$. In the induced subgraph $P(\omega^+)$ we can find a path $p$ linking, without loss of generality, $v_0$ and $u_0$ with no other vertices from $\pi_1$ and $\pi_2$. Furthermore, the walks belonging to $\{g^n\cdot p\mid n\in\Z\}$ do not intersect each other. This way, the graph spanned by $\pi_1$, $\pi_2$ and $p$ is a periodic subdivision of the infinite $2$-ladder, $Q'\subseteq Q$. Now, take an end $\omega_1\neq\omega^{\pm}$ and $\pi_3 = (v'_i)_{i\in\N}$ a ray defining $\omega_1$. As $\pi_3$ defines an end different from $\omega^+$ there exists a smallest $N\in\N$ such that $v'_i\notin Q'$ for all $i>N$. Because $\Gamma(G,S)$ is transitive, we can take without loss of generality $v'_N$ to be equal to some $u_k$ with $k\in\N$, placed between $g\cdot u_0$ and $g^2\cdot u_0$ . This is all represented in Figure~\ref{fig:2ladder}.
    
    \begin{figure}[!ht]
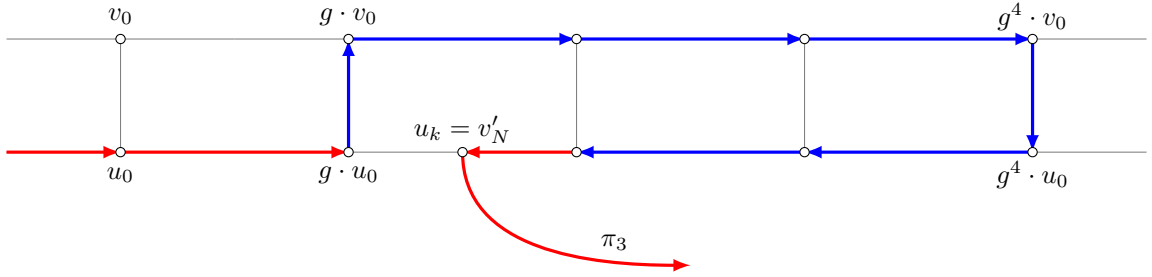

        \centering
        \include{figures/2ladder}
        \vspace{-0.5cm}
        \caption{The periodic subdivision of the $2$-ladder with the configuration $x$ highlighted in red and blue. The word $\lambda_2 \lambda_3^3 \lambda_2^{-1} \lambda_1^{-2}$ is marked in blue, whereas the infinite prefix and suffix of $x$ are marked in red.}
        \label{fig:2ladder}
    \end{figure}

Let us label the different sections of the bi-infinite ladder. We denote by $\lambda_1$ the label of the path from $u_0$ to $g\cdot u_0$, $\lambda_2$ the label from $u_0$ to $v_0$, $\lambda_3$ the label from $v_0$ to $g\cdot v_0$, $\lambda_4$ the label from $g^2\cdot u_0$ to $u_k$, and $\lambda\in S^\N$ the label of the ray $(v'_{N+i})_{i\in\N}$. Then, for every $n\in\N$ and $k<n$ the configuration
$$x = \lambda_1^{\infty}.\lambda_2\lambda_3^n\lambda_2^{-1}\lambda_1^{-k}\lambda_4\lambda\in S^{\Z},$$
belongs to the skeleton. Then, $\lambda_2\lambda_3^n\lambda_2^{-1}\lambda_1^{-k}\in\lang(\Xs_{G,S})$ for all $k<n$. Notice that the language given by $L = \{\lambda_2\lambda_3^n\lambda_2^{-1}\lambda_1^{-k} \mid n,k\in\N\}$ is regular. Therefore, $L' = L\cap\lang(\Xs_{G,S})$ is regular as we assume $\Xs_{G,S}$ is sofic. Take $p>0$ the pumping length of $L'$ given by the Pumping Lemma. If we pump the word $\lambda_2\lambda_3^{p+1}\lambda_2^{-1}\lambda_1^{-p}$ is in $L'$ as we did in the proof of Lemma~\ref{lem:2strip}, we conclude that there must exist $n,k\in\N$ with $k\geq n$ such that $\lambda_2\lambda_3^n\lambda_2^{-1}\lambda_1^{-k}\in L'$, which is a contradiction as it is not self-avoiding.
\end{proof}

Virtually $\Z$ groups have a very rigid structure. Epstein and Wall~\cite{epstein1961ends,wall1967ends} (see~\cite{lima2013virtually} for our current formulation) showed that a group is virtually $\Z$ if and only if it is of one of the following forms:
\begin{enumerate}
    \item $\Z\ltimes_{\phi} F$, for some finite group $F$ and $\phi\in\Aut(F)$,
    \item $G_1\ast_F G_2$, for $G_1$, $G_2$ and $F$ finite groups such that $[G_1:F]=[G_2:F]=2$.
\end{enumerate}

Groups of the second type, $G_1\ast_F G_2$, can be shown to be isomorphic to $\D_{\infty}\ltimes_{\psi}F$ for some homomorphism $\psi:\D_{\infty}\to\Aut(F)$ (see~\cite[Section 1.3]{gilabert2022virt}). Furthermore, every element $g\in\Z\ltimes_{\phi}F$ can be uniquely expressed as $ft^n$ with $f\in F$, $n\in\Z$ and $t$ the free generator of $\Z$. Similarly, every element $g\in\D_{\infty}\ltimes_{\phi}F$ can be uniquely expressed as $f\tt{r}^n\tt{s}^{b}$ with $f\in F$, $n\in\Z$, $b\in\{0,1\}$, and $\tt{r}$ and $\tt{s}$ generators for $\D_{\infty} = \langle \tt{r},\tt{s} \mid \tt{s}^2, \tt{rsrs}\rangle$.

\begin{lemma}
\label{lem:ends_virtualZ}
    Let $G = H\ltimes_{\phi} F$ be a group such that $F$ is a finite group, and $H$ is either $\Z$ or $\D_{\infty}$. Then, for any generating set $S$ the ends of the Cayley graph $\Gamma(G,S)$ have size at least $|F|$.
\end{lemma}

\begin{proof}
   Take $G$ as in the hypothesis. We will tackle the case when $H = \Z$ and $H = \D_{\infty}$ separately.\\
   
   \textbf{Case 1: }$H = \Z$:

   Let $S$ be a generating set for $G$. Then, there must exist at least one generator that does not belong to $F$, which we call $s$. This generator, must have the form $s = gt^n$ for some $g\in F$ and $n\in\Z$, and is thus a torsion-free element of the group. For each element $f\in F$ we define the ray $\pi_f = (f, fs,\ ...,\ fs^i,\ ...)$. These rays are all pair-wise disjoint because $s$ is torsion-free. Therefore, the end $D(s)$ has size at least $|F|$.\\

   \textbf{Case 2: }$H = \D_{\infty}$:

   Let $S$ be a generating set for $G$. As before, there must exist at least one generator that does not belong to $F$, which we call $s$. If $s$ is of the form $g\tt{r}^n$, it is a torsion free element, and by the argument for the previous case, $D(s)$ has size at least $|F|$. Suppose then that all elements $S\setminus F$ are of the form $g\tt{r}^n\tt{s}$. Because $S$ is a generating set, $S\setminus F$ must contain at least two elements which we will name $s = g\tt{r}^n\tt{s}$ and $s' = g'\tt{r}^{m}\tt{s}$. Without loss of generality take $n>m$. Then, $ss'$ is the torsion-free element $g_1\tt{r}^{n-m}$ for some $g_1\in F$. As before, for each $f\in F$ define the ray $\pi_f = (f, fs, fss',\ ...,\ f(ss')^i,\ ...)$. Let us prove these rays are disjoint. If $f(ss')^k = f'(ss')^l$ for $f,f'\in F$ and $k,l\in\N$, then $(ss')^{k-l}\in F$ has torsion, which is a contradiction. On the other hand, if $f(ss')^ks = f'(ss')^l$ for $f,f'\in F$ and $k,l\in\N$, then $\tt{r}^{(n-m)(k+l) + n}\tt{s}\in F$, which is also a contradiction. Thus, the rays $\pi_f$ are disjoint and therefore $D(ss')$ has size at least $|F|$.
\end{proof}

\begin{proposition}
\label{prop:virtZ}
    Let $G$ be a virtually $\Z$ group. Then, there exists $S$ such that $\Xs_{G,S}$ is sofic if and only if $G$ is either $\Z$, $\Z\times  \Z/2\Z$, $\D_{\infty}\times \Z/2\Z$ or $\D_{\infty}$.
\end{proposition}
\begin{proof}
    Let $G$ be a virtually $\Z$ group. Then, $G$ is of the form $H\ltimes_{\phi}F$ for $H\in\{\Z,\D_{\infty}\}$ and $F$ a finite group. Joining Lemma~\ref{lem:ends_virtualZ} and Lemma~\ref{lem:2strip}, if $\Xs_{G,S}$ is sofic for some generating set $S$, $|F|\leq 2$. If $|F| = 1$, then $G$ is either $\Z$ or $\D_{\infty}$. If $|F| = 2$, then $F\simeq \Z/2\Z$ and $\phi$ is the trivial automorphism. In this case $G$ is either $\Z\times \Z/2\Z$ or $\D_{\infty}\times \Z/2\Z$.

    Conversely, we already know $\Z$ and $\D_{\infty}$ admit sofic skeletons as they are plain groups. Similarly, by Lemma~\ref{lem:ZyZ2}, $\Z\times \Z/2\Z$ admits a sofic skeleton. Finally, if we take the presentation for $\D_{\infty}\times \Z/2\Z$ given by $\langle a, b, s \mid a^2, b^2, s^2, (sa)^2, (sb)^2\rangle$  the corresponding Cayley graph is the bi-infinite ladder, and therefore Lemma~\ref{lem:ZyZ2} can be adapted to show its skeleton is sofic.
\end{proof}

We now have all the ingredients to characterize groups that admit a sofic skeleton.

\begin{proof}[Proof of Theorem~\ref{thm:sofic}]
    Let $G$ be a finitely generated group that admits a sofic skeleton through the generating set $S$. From Lemma \ref{fig:2ladder}, $\Gamma(G,S)$ has only thin ends, all of size at most 2. If all ends are of size 1, $G$ is a plain group. Next, if $G$ has at least one end of size 2, it is virtually $\Z$ by Lemma~\ref{lem:2ends_Z}. Then, by Proposition~\ref{prop:virtZ} $G$ is either $\Z\times  \Z/2\Z$ or $\D_{\infty}\times \Z/2\Z$.
    For the other direction, if $G$ is a plain group by Theorem~\ref{thm:SFT} it admits a sofic skeleton (as SFTs are sofic). Finally, if $G$ is either $\Z\times  \Z/2\Z$ or $\D_{\infty}\times \Z/2\Z$, Proposition~\ref{prop:virtZ} tells us $G$ admits a sofic skeleton.
\end{proof}

%=-=-=-=-=-=-=-=-=-=-=-=-=-=-=-=-=-=-=-=-=-=-=-=-=-=-=-=-=-=-=-=-=-=-=-=-=-
%
%=-=-=-=-=-=-=-=-=-=-=-=-=-=-=-=-=-=-=-=-=-=-=-=-=-=-=-=-=-=-=-=-=-=-=-=-=-
\section{Approximating Entropy and Connective Constants}
\label{section.entropy_CC}

\subsection{Bridges and Periodic Points}

We saw in Corollary~\ref{cor:irreducible} that the connective constant of SFT skeletons of plain groups can be approximated by their periodic points. This is also the case of irreducible sofic shifts~\cite[Theorem 4.3.6.]{lind2021introduction}. The natural question that follows is if this is possible for skeletons that are not sofic. Clisby showed~\cite{clisby2013endless} that this is the case for $\Z^d$ with standard generating set, which by Theorem~\ref{thm:sofic} do not define sofic skeletons. Instead of periodic configurations, Clisby used the term \emph{endless SAWs}. By using the notion of a bridge, introduced by Hammersley and Welsh~\cite{hammersley1962constant} and latter expanded upon by Grimmett and Li~\cite{grimmet2018locality}, we can generalize this result to any Cayley graph admitting a particular kind of graph height function.

\begin{definition}
\label{def:graphHeight}
    Let $\Gamma$ be an infinite, connected, locally finite, quasi-transitive graph. A \define{graph height function} $(h, H)$ is composed of a function $h:V_\Gamma\to \Z$ and a subgroup $H\leq\Aut(\Gamma)$ acting quasi-transitively on $\Gamma$ such that
    \begin{itemize}
    \item ($H$-difference-preserving) for all $u,v\in V_\Gamma$ and $g\in H$   
    $$h(g\cdot v) - h(g\cdot u) = h(v) - h(u),$$
    \item for all $u\in V_\Gamma$, there exists $v,v'\in V_\Gamma$ adjacent to $u$ such that $h(v) < h(u) < h(v')$.
    \end{itemize}
\end{definition}

A \define{bridge} with respect to the height function $(h,H)$ is a self-avoiding walk $\pi = (e_0,\ ...,\ e_{n-1})$ that verifies
$$h(\init(e_0)) < h(\ter(e_i)) \leq h(\ter(e_{n-1})),$$
for all $i\in\{0,...,n-1\}$.  

\begin{example}
    Take $G = \Z^2$ with the standard generating set $\{\tt{a}^{\pm},\tt{b}^{\pm}\}$. If we look at the generators as $\tt{a} = (1,0)$ and $\tt{b}=(0,1)$, we define the map $h(g) = m$ for $g=(m,n)\in\Z^2$. This function defines a graph height function with respect to $H = \Z^2$ acting by left-translations. Further still, any elementary amenable group admits a graph height function~\cite{grimmet2018locality}.
\end{example}

\begin{lemma}
\label{lem:bridge}
    Let $G$ be a finitely generated group with generating set $S$. If $\Gamma(G,S)$ admits a graph height function $(h, H)$, then $\Xs_{G,S}$ contains periodic configurations. Moreover, if $\pi$ is a bridge such that $\init(\pi)$ and $\ter(\pi)$ lie in the same $H$-orbit, then $\lambda(\pi)^{\infty}\in \Xs_{G,S}$.
\end{lemma}

\begin{proof}
    Let $\pi$ and $\pi'$ be two bridges such that $\ter(\pi) = \init(\pi')$. Then, the concatenation of both paths, $\pi\pi'$, is a bridge. Furthermore, for every $g\in H$, $g\cdot \pi$ is also a bridge, as $h$ is $H$-difference-preserving. 

    Now, let $R$ be a finite right transversal for the action of $H$ on $\Gamma(G,S)$. Take a bridge $\pi$ such that $\init(\pi),\ter(\pi)\in H\cdot r$ with $r\in R$. If $\init(\pi) = h_1\cdot r$ and $\ter(\pi) = h_2\cdot r$, because $h$ is $H$-difference-preserving, $h_2h_1^{-1}\cdot \pi$ is a bridge starting at $h_2\cdot r$. We can then concatenate $\pi$ with $h_2h_1^{-1}\cdot \pi$ to create a bridge, which we denote by $\pi^2$, whose label is given by $\lambda(\pi)^2$. This process can be iterated indefinitely to obtain a bi-infinite SAW whose label is given by $\lambda(\pi)^\infty$.

    Next, take a bridge $\pi$ such that $\init(\pi)\in H\cdot r_1$ and $\ter(\pi)\in H\cdot r_2$, with $r_1, r_2 \in R$ distinct representatives. Up to translation by an element from $H$, we can take any bridge starting at a vertex in $H\cdot r_2$, say $\pi_1$ and concatenate to $\pi$ to obtain a new bridge $\pi\pi_1$. Such a bridge exists by the definition of a graph height function as there must exist at least one vertex $v$ next to $r_2$ such that $h(r_2) < h(v)$. Similarly, we can take any bridge in the $H$-orbit of $\ter(\pi_1)$, which we denote $\pi_2$, and concatenate it --up to translation by $H$-- to $\pi\pi_1$. Iterating this process, for all $n\in \N$ we obtain a bridge $\pi\pi_1\ ...\ \pi_n$. Because there is a finite number of $H$-orbits, we will have $i\leq j$ such that $\init(\pi_i),\ter(\pi_j)$ belong to the same $H$-orbit. Then, as previously stated $\pi' = \pi_i\pi_{i+1}\ ...\ \pi_{j}$ is a bridge that can be iterated to obtain the periodic point $\lambda(\pi')^{\infty}$. 
\end{proof}

We saw in Theorem \ref{thm:aperiodic} that torsion groups have aperiodic skeletons. By the previous lemma, graph height functions imply the existence of periodic points. Combining these two facts we can state the following.

\begin{theorem}[Theorem~\ref{intro:height_function}]
\label{theorem.torsion_group_no_graph_height_function}
        The Cayley graphs of infinite torsion f.g. groups do not admit graph height functions.
\end{theorem}
This generalizes a result from Grimmett and Li who showed that the Grigorchuk group (which is an infinite torsion group) does not admit a graph height function, and more generally, Cayley graphs of torsion groups with certain conditions on the stabilizer of the identity~\cite{grimmett2015self}. However, the converse of the previous theorem does not hold, as they also showed that the Higman group, which is torsion-free~\cite{higman1951finitely}, does not admit graph height functions.\\

Bridges are particularly useful to compute the connective constant of graphs, and have been used to obtain exact expressions for the constant (for instance, \cite{duminil2012connective}). Let us denote by $b_{n,g}$ the number of bridges of length $n$ starting at $g\in G$, and $b_n = \min_{g\in G}b_{n,g}$. As stated in the proof of Lemma \ref{lem:bridge}, we can concatenate bridges with corresponding endpoints. Then, $b_n b_m\leq b_{n+m}$ and by Fekete's sub-additive Lemma, there exists a constant $\beta(\Gamma, h)$, where $\Gamma = \Gamma(G,S)$, such that
$$\beta(\Gamma,h) = \lim_{n\to\infty}\sqrt[n]{b_n}.$$

This process can be done for a larger class of graphs \cite{grimmet2018locality}, and helps us compute connective constants.

\begin{theorem}[General Bridge Theorem \cite{lindorfer2020bridge}]
\label{thm:bridge}
    Let $\Gamma$ be an infinite, connected, locally finite, quasi-transitive graph. Then, if $\Gamma$ admits a graph height function $(h,H)$, 
    $$\mu(\Gamma) = \max\{\beta(\Gamma, h), \beta(\Gamma, -h)\}.$$
\end{theorem}

Using this result, we can find conditions under which periodic points approximate the connective constant. In other words, periodic points from $\Xs_{G,S}$ approximate its entropy.

\begin{theorem}[Theorem~\ref{intro:approximation}]
    Let $G$ be a finitely generated group and $S$ a finite generating set. If $\Gamma(G,S)$ admits a graph height function $(h, H)$ such that $H$ acts transitively on $\Gamma(G,S)$, then 
    $$\mu(G,S) = \lim_{n\to \infty} \sqrt[n]{e_n},$$
    where $e_n$ denotes the number of periodic points in $\Xs_{G,S}$ of period $n\in\N$.
\end{theorem}

\begin{proof}
Let us denote by $\bar{b}_n$ the minimum over all $g\in G$ of the number of bridges of length $n$ starting at $g$ for the graph height function $(H,-h)$. Because $H$ acts transitively on $\Gamma(G,S)$, there is a single $H$-orbit. Thus, by Lemma \ref{lem:bridge}, every bridge for $h$ and $-h$ can be iterated to obtain a periodic point. This means,
    $$\max\{b_n, \bar{b}_n\}\leq e_n \leq c_n.$$
By taking the $n$th root and limit, Theorem \ref{thm:bridge} implies,
$$\mu(G,S) = \max\{\beta(\Gamma, h), \beta(\Gamma, -h)\}\leq \lim_{n\to\infty}\sqrt[n]{e_n}\leq \mu(G,S).$$
\end{proof}

Examples of Cayley graphs with a graph height function $(h, H)$ such that $H$ acts transitively are given by Cayley graphs that admit \emph{strong} graph height function where $H = G$. Strong graph height functions are graph functions where we also ask for $H$ to be a finite index subgroup of $G$, and to act by left translations~\cite{grimmett2015self}. A class of groups that admit such functions are groups with strictly positive first Betty number~\cite{grimmett2017connective}. Other sufficient conditions can be found in~\cite{grimmett2020cubic}.

%=-=-=-=-=-=-=-=-=-=-=-=-=-=-=-=-=-=-=-=-=-=-=-=-=-=-=-=-=-=-=-=-=-=-=-=-=-

\subsection{Lower Bounds with Self-Avoiding Polygons}
\label{subsection.lower_bounds_SAP}

What other methods can we use when graph height functions are not available? We will make use of a counting argument popularized by Rosenfeld~\cite{rosenfeld2020colorings} to find lower bounds on the connective constant by studying the sets of forbidden patterns defining the skeleton. Rosenfeld found the following criterion for subshifts.

\begin{theorem}[\cite{rosenfeld2022finding}, Corollary 12]
    Let $A$ be a finite alphabet and $\Fo\subseteq A^{+}$ a set of connected forbidden patterns. If there exists a positive real number $\beta>1$ such that
    $$|A| \geq \beta + \sum_{n\geq 0}f_n\beta^{1-n},$$
    then $\alpha(\X_{\Fo})\geq\beta$, where $f_n$ is the number of forbidden patterns of length $n$, that is, $f_n = |\Fo\cap A^{n}|$.
\end{theorem}

Therefore, we can use the different forbidden patterns we have found so far for the skeleton to find lower bounds for the connective constant. From Lemma \ref{lem:ouro}, we know the set of SAPs along with words of the form $ss^{-1}$ define a set of forbidden patterns for the skeleton. 

\begin{proposition}
\label{prop:rosenfeld}
    Let $(G,S)$ be an infinite finitely generated group. If there exists a positive real number $\beta$ such that 
    $$|S| - 1 \geq \beta + \sum_{n\geq 0}\rho_n\beta^{1-n},$$
    then $\mu(G,S) \geq \beta$, where $\rho_n$ the number of SAPs of length $n$, that is, $\rho_n = |\Ou_{G,S}\cap S^{n}|$.
\end{proposition}

The proof of the proposition is essentially the same as the one from \cite{rosenfeld2022finding}, but we add it for completion.

 \begin{proof} 
    Let $L_k$ be the set of SAWs of length $k\in\N$. We prove by induction that $|L_{k}| \geq \beta |L_{k-1}|$, for $\beta > 1$ as in the statement. Notice $|L_0| = 1$ as it only contains the empty word, and $|L_1| = |S|$. By hypothesis, $\beta\leq |S|$, and therefore $|L_1|\geq\beta|L_0|$. 

     Suppose our statement is true up to some $k>0$. In particular, for $j \leq k$ 
     $$|L_{k-j}| \leq \frac{|L_k|}{\beta^{j}}.$$

     Now, because every SAW from $L_k$ can be extended in $|S|-1$ ways, we have that
     $$|L_{k+1}| = (|S|-1)|L_k| - |B|,$$
     where $B$ is the set of SAWs that when extended generate a path of length $k+1$ that self-intersects. Notice that if $u\in B$, it can be written in the form $u = u'v$, where $v$ is a SAP, as $u$ is the extension by one non-backtracking letter of a SAW. We define the sets $B_v = \{u\in B\mid u = u'v\}$ to obtain the upper bound $|B| \leq \sum_{v\in\Ou_{G,S}}|B_v|$.
     Then, every word in $B_v$ is determined by a word from $L_{k+1 - |v|}$, namely $u'$. 
     Therefore,
     $$|B_v| \leq |L_{k+1 - |v|}|\leq \frac{|L_k|}{\beta^{|v| - 1}},$$
     and consequently,
     $$|B| \leq |L_k|\sum_{v\in\Ou_{G,S}}\beta^{1-|v|} = |L_k|\sum_{n\geq 1}\rho_n\beta^{1-n}$$

     Finally, joining all the formulas we obtain:
     $$|L_{k+1}| \geq \left((|S|-1) - \sum_{n\geq 1}\rho_n\beta^{1-n}\right)|L_k| \geq \beta|L_k|.$$
     Our induction proven, we can iterate the identity to obtain $|L_k|\geq \beta^k$, and thus $\mu(G, S)\geq \beta$.
\end{proof}

This approach is different from the usual use of self-avoiding polygons to approximate $\mu(G,S)$ in the literature. We define the asymptotic growth rate for SAPs through,
$$\mu_{SAP} = \limsup_{n\to\infty}\sqrt{\rho_n}.$$
It has been shown that $\mu_{SAP} = \mu(G,S)$ for Euclidean lattices~\cite{hammersley1961number,kesten1963number}, but $\mu_{SAP} < \mu(G,S)$ for many non-euclidean lattices, including some Cayley graphs of surface groups~\cite{panagiotis2019self}.

\section{Geodesic skeletons}
\label{seciton.geodesic}

A geodesic is always a self-avoiding walk. It is then natural to see what changes when we restrict a group's skeleton to bi-infinite geodesics. We define the geodesic skeleton of $G$ with respect to $S$ by,
\[\Xs^g_{G,S} = \{x\in \Xs_{G,S} \mid \forall w\factor x, w'=_G w : \ |w|\leq |w'|\}.\]

This subshift is contained in the skeleton $\Xs_{G,S}$, and the locally admissible language given by its defining forbidden patterns is $\Geo(G,S)$. In particular, $\Xs^g_{G,S}$ is generated by taking $\Geo(G,S)^c$ as the set of forbidden patterns. As was the case with the skeleton, $\Xs_{G,S}^g$ is empty if and only if the group is finite; this is due to Watkins who showed that every transitive infinite graph contains a bi-infinite geodesic~\cite{watkins1986infinite}.

%=-=-=-=-=-=-=-=-=-=-=-=-=-=-=-=-=-=-=-=-=-=-=-=-=-=-=-=-=-=-=-=-=-=-=-=-=-
\paragraph{SFT Geodesics}
We have a sufficient condition for the geodesic skeleton to be an SFT coming from a result by Gilman, Hermiller, Holt and Rees~\cite{gilman2007characterisation} that characterizes virtually free groups. They showed that for a finitely generated group $G$, there exists a finite generating set $S$ such that $\Geo(G,S)$ is $k$-locally excluding, that is, there exists a set $F$ of words of length $k$ such that a word $w\in S^*$ is geodesic if no factor of length $k$ belongs to $F$, if and only if $G$ is virtually free. An immediate consequence is the following.
\begin{proposition}
    Let $G$ be a virtually free group. Then, there exists $S$ such that $\Xs_{G,S}^g$ is a SFT.
\end{proposition}

%=-=-=-=-=-=-=-=-=-=-=-=-=-=-=-=-=-=-=-=-=-=-=-=-=-=-=-=-=-=-=-=-=-=-=-=-=-
\paragraph{Effective Geodesics} 
\begin{lemma}
    Let $G$ be a recursively presented group. Then, $\Xs^g_{G,S}$ is effective for every finite generating set $S$.
\end{lemma}

\begin{proof}
    We describe a co-semi-algorithm for $\Geo(G,S)$. By using an enumeration for the word problem, we can test every word $w'$ of length $|w'| < |w|$ to see if they define the same group element, i.e. $w'w^{-1} =_{G} 1_G$. If one such $w'w^{-1}$ appears in the enumeration, we know $w$ is not geodesic and accept. If $w$ is not geodesic, $w'w^{-1}$ will eventually be enumerated, for some $w'$ of shorter length. When $w\in\Geo(G,S)$ the algorithm never stops.
\end{proof}

In other words, the effectiveness of $\Xs^{g}_{G,S}$ is a consequence of the fact that a recursively enumerable word problem implies that the language of geodesics is co-recursively enumerable.\\

%=-=-=-=-=-=-=-=-=-=-=-=-=-=-=-=-=-=-=-=-=-=-=-=-=-=-=-=-=-=-=-=-=-=-=-=-=-
\paragraph{Sofic Geodesics}
By their definition, we can easily obtain many examples of sofic geodesic skeletons.
\begin{proposition}
\label{prop:geod_sofic}
    Let $G$ be a finitely generated group and $S$ a generating set. If $\Geo(G,S)$ is regular, then $\Xs_{G,S}^{g}$ is sofic.
\end{proposition}

Because the complement of a regular language is regular, when $\Geo(G,S)$ is regular then $\Xs^g_{G,S}$ is defined by a regular set of forbidden words, and is therefore sofic. We know that $\Geo(G,S)$ is regular for all generating sets in abelian groups \cite{neumann1995automatic} and hyperbolic groups \cite{epstein1992word}. Also, there exists at least one generating set such that $\Geo(G,S)$ is regular for virtually abelian groups \cite{neumann1995automatic}, Coxeter groups \cite{howlett1993miscellaneous} and other classes \cite{charney2004language,holt2012artin,antolin2016finite}.

In order to find a characterization of groups that admit a geodesic skeleton that is sofic, we must look at geodesics that are not extendable. These elements are precisely the ones known as dead-ends. An element $g\in G$ is a \define{dead-end}\index{dead-end} with respect to the generating set $S$ if for all $s\in S$ we have $d(1_G, gs) \leq d(1_G, g)$.

\begin{proposition}
\label{prop:deadends}
      Let $G$ be a finitely generated group along with a generating set $S$. Then, $\Xs_{G,S}^g$ is sofic and the language of geodesics defining dead-ends is regular if and only if $\Geo(G,S)$ is regular.
\end{proposition}

\begin{proof}
Denote the language of geodesics defining dead-ends by $D$. If $\Geo(G,S)$ is regular, then $D$ is regular: it suffices to take the minimal deterministic finite state automaton with a single sink state for $\Geo(G,S)$ and only keep accepting states where every outgoing transitions goes to the sink state. Furthermore, $\Xs_{G,S}^g$ is sofic by Proposition~\ref{prop:geod_sofic}.

Conversely, suppose $\Xs_{G,S}^g$ is sofic and $D$ is regular. To prove $\Geo(G,S)$ is regular in this case, we make use of the following notation: for a language $L\subseteq S^*$, consider the languages $L^{-1} = \{w \mid w^{-1}\in L\}$ and 
 $$mL = \{w\in S^{*} \mid \exists v,u\in S^*: \ vwu\in L\}.$$
Notice that if $L$ is regular, then both $L^{-1}$ and $mL$ are regular.
We claim that 
$$\Geo(G,S) = \lang(\Xs^g_{G,S})\cup mD \cup m(D^{-1}).$$
Indeed, in this case $\Geo(G,S)$ is regular as it is the union of regular languages. Let us prove the equality. First, it is clear that $\lang(\Xs^g_{G,S})\subseteq \Geo(G,S)$. Next, because inverses of geodesics are geodesic, and factors of geodesics are geodesics, we have $mD\cup m(D^{-1})\subseteq\Geo(G,S)$. This proves the left inclusion. For the converse, take $w\in\Geo(G,S)$. If $w$ is bi-infinitely extendable as a geodesic, then $w\in\lang(\Xs_{G,S}^g)$. If $w$ is not bi-extendable, there exist $v_1,v_2\in S^*$ such that $v_1wv_2\in\Geo(G,S)$, and $v_1wv_2s\notin\Geo(G,S)$ or for all $s\in S$ $sv_1wv_2\notin\Geo(G,S)$. This means $w$ is either a factor of dead-end (first case) or a factor of the inverse of a dead-end (second case). Therefore, $w\in mD\cup m(D^{-1})$. This proves our claim, and concludes the proof.
\end{proof}

\begin{corollary}
    Let $G$ be any finitely generated group with generating set $S$, and $\Z=\langle t\rangle$. Then, $\Xs^g_{G\ast\Z, S\cup\{t^{\pm 1}\}}$ is sofic if and only if $\Geo(G,S)$ is regular.
\end{corollary}
\begin{proof}
    The language of dead-ends of $G\ast\Z$ is empty as any geodesic can be extended by $t^{\pm 1}$. Furthermore, any geodesic in $G\ast\Z$ can be decomposed as geodesics on $G$ separated by factors of the form $t^{\pm n}$. Therefore, $\Geo(G\ast\Z, S\cup\{t^{\pm 1}\})$ is regular if and only if $\Geo(G,S)$ is regular. By Proposition~\ref{prop:deadends} this happens if and only if $\Xs^g_{G\ast\Z, S\cup\{t^{\pm 1}\}}$ is sofic.
\end{proof}

We pose the following question for sofic geodesic skeleton.
\begin{question}
       Is $\Xs^{g}_{G,S}$ sofic if and only if $\Geo(G,S)$ is regular?
\end{question}

%=-=-=-=-=-=-=-=-=-=-=-=-=-=-=-=-=-=-=-=-=-=-=-=-=-=-=-=-=-=-=-=-=-=-=-=-=-
\paragraph{Periodic Geodesics}

As was the case for the skeleton (Theorem~\ref{thm:aperiodic}), the aperiodicity of the geodesic skeleton also characterizes torsion groups.

\begin{theorem}
    Let $G$ be a finitely generated group. Then, $G$ is a torsion group if and only if $\Xs_{G,S}^g$ is aperiodic for every (any) generating set $S$.
\end{theorem}

\begin{proof}
    Suppose $G$ contains a torsion-free element $g$. Then, by Proposition~\ref{prop:torsion-free} for any generating set $S$, there exists $k\geq 1$ and $w\in S^*$ a geodesic for $g^k$ such that $w^{\infty}\in\Xs^g_{G,S}$, which is a periodic configuration. Conversely, if there exists a periodic configuration $x=w^\infty\in\Xs_{G,S}^g$ for some generating set $S$, then $g =\overline{w}$ is a torsion-free element.
\end{proof}

\subsection{Entropy and Connective Constant for Geodesics}

The objective of this section is to define an analog of the connective constant for geodesics. This relies on finding the asymptotic growth rate of geodesics of a given length. The \define{geodesic growth} of $G$ with respect to $S$ is the map $\Gamma_{G,S}:\N\to\N$ given by 
$$\Gamma_{G,S}(n) = |\{w\in \Geo(G,S)\mid |w|\leq n\}|.$$

Because this function is sub-multiplicative we can define the \define{geodesic connective constant} of the Cayley graph $\Gamma(G,S)$ as 
$$\mu^g(G,S) = \lim_{n\to\infty}\sqrt[n]{\Gamma_{G,S}(n)}.$$

As we saw in Remark~\ref{rem:strict_vs_normal}, $\mu^g(G,S)$ can be shown to be equal to the growth rate of the number of geodesics of length exactly $n$. Thus, the geodesic growth is an upper bound on the complexity of $\Xs^g_{G,S}$. Because $\Geo(G,S)$ is the set of locally admissible words for the geodesic skeleton, we use Lemma \ref{lem:global_local} to obtain an expression for the entropy.

\begin{lemma}
Let $G$ be a finitely generated group along with a generating set $S$. Then,
$$ h(\Xs^g_{G,S}) =\log(\mu^g(G,S)).$$
In other words, the geodesic connective constant is equal to the connective constant of bi-extendable geodesics.
\end{lemma}

Same as with the connective constant, the geodesic version is a non-negative rational power of a Perron number when $\Xs^g_{G,S}$ is sofic, and a right-computable number when $\Xs^g_{G,S}$ is effective. It is also a lower bound of the connective constant, that is, $\mu^g(G,S)\leq\mu(G,S)$. This inequality may be strict: graphs may have geodesic connective constant equal to 1 without being finite. As shown in \cite{bridson2012groups}, the virtually $\Z^2$ group $H = \langle \tt{a},\tt{t} \mid [\tt{a}, \tt{tat}^{-1}], \tt{t}^2\rangle$, has geodesic growth of order $O(n^3)$ and therefore, 
$$\mu^g(H,\{\tt{a},\tt{t}\}) = 1 < \sqrt{3}\leq \mu(H,\{\tt{a},\tt{t}\}).$$
This is also the case for lattices with known (or well-approximated) connective constants.

\begin{proposition}
The geodesic connective constants of the square grid, ladder graph and hexagonal grid are as follows:
\begin{itemize}
    \item $\mu^{g}(\Z^2) = 2$,
    \item  $\mu^{g}(\mathbb{L}) = 1$,
    \item $\mu^{g}(\mathbb{H})=\sqrt{2}$.
\end{itemize}
\end{proposition}

\begin{proof}
\begin{itemize}
    \item For the square lattice, we know that $\Gamma_{\Z^2, \{\tt{a},\tt{b}\}}(n) \leq 2^{n+3}$ which implies $h(\Xs^g_{\Z^2, \{\tt{a},\tt{b}\}}) = \log(2)$, as $\Xs^g_{G,S}$ contains the full-shift $\{\tt{a},\tt{b}\}^\Z$.
    \item Recall that the ladder graph $\mathbb{L}$ is the Cayley graph of $\Z\times \Z/2\Z$ with generating set $\{\tt{t},\tt{s}\}$, where $\tt{s}^2 =_{G} \epsilon$ and $\tt{t}$ is the generator for $\Z$. In this case, the geodesic growth is given by $\Gamma(n) = n^2 + 3n$ when $n\geq 2$. Thus, the geodesic connective constant is $1$.
    \item Also recall that the hexagonal grid $\mathbb{H}$ is the Cayley graph of the Coxeter group $\tilde{A}_2$ with generating set $\{\tt{a}, \tt{b}, \tt{c}\}$ (see Example~\ref{ex:Coxeter}). From~\cite{avasjo2004automata} we know that the generating function for the geodesic growth of $\tilde{A}_2$ in this case is given by
    $$f(z) = \frac{2z^3 + z^2 + z + 1}{(1-z)(1-2z^2)}.$$
    Thus, the geodesic connective constant is given by the reciprocal of the smallest zero of the denominator, which is $\sqrt{2}$.
\end{itemize}

\end{proof}

On the other hand, if we take the infinite dihedral group $\mathcal{D}_{\infty}$ with the generating set $S = \{\tt{a}, \tt{b}\}$ as seen in Example~\ref{ex:dihedral}, we have that $\mu(\D_{\infty},S) = \mu^g(\D_{\infty},S) = 1$.

\begin{question}
   Under which conditions $\mu(G,S) = \mu^g(G,S)$? Under which conditions is the inequality strict?
\end{question}

% %=-=-=-=-=-=-=-=-=-=-=-=-=-=-=-=-=-=-=-=-=-=-=-=-=-=-=-=-=-=-=-=-=-=-=-

\section*{Acknowledgments}

We would like to thank Matthieu Rosenfeld for discussions on the counting method, especially when applied to $n$-power-free words. We would also like to thank Alex Bishop for pointing us towards references on the geodesic growth of Coxeter groups.

\printbibliography
\vspace{0.5cm}
\begin{tabular}{@{}l}\scshape Université Paris-Saclay, CNRS, LISN, Gif-sur-Yvette, France\\\textit{E-mail address: }\href{mailto:nathalie.aubrun@lisn.fr}{nathalie.aubrun@lisn.fr}\end{tabular}

\vspace{0.5cm}

\begin{tabular}{@{}l}\scshape Université Paris-Saclay, CNRS, LISN, Gif-sur-Yvette, France\\\textit{E-mail address: }\href{mailto:nicolas.bitar@lisn.fr}{nicolas.bitar@lisn.fr}\end{tabular}

%=-=-=-=-=-=-=-=-=-=-=-=-=-=-=-=-=-=-=-=-=-=-=-=-=-=-=-=-=-=-=-=-=-=-=-
\end{document}

%% file: preamble.tex
\usepackage[T1]{fontenc}
\usepackage[utf8]{inputenc}
\usepackage{lmodern}
\usepackage[french,english]{babel}
\usepackage{geometry} 
\usepackage{amsmath,amssymb,amsfonts,amsthm,mathtools}
\usepackage{mathrsfs}% pour rajouter un format de lettres façon "caligraphie" en math mode.

\usepackage{standalone}

\usepackage{graphicx}
\usepackage{caption}
\usepackage{subcaption} % permet de faire des subfigures (remplace le package subfig)

\makeatletter

\newskip\@bigflushglue \@bigflushglue = -100pt plus 1fil

\def\bigcentering{\let\\\@centercr\rightskip\@bigflushglue%
\leftskip\@bigflushglue
\parindent\z@\parfillskip\z@skip}

\makeatother

\usepackage{booktabs}
\usepackage{array}
\usepackage{paralist} % si besoin regarder "enumitem" qui est un peu différent
\usepackage{fancyhdr}
\usepackage{emptypage} % garantit que les pages blanches avant les débuts de chapitres soient vraiment blanches (pas d'entête ni de pied de page)
 
\fancypagestyle{plain}{
\fancyhf{}
\fancyfoot[RO,RE]{\thepage}

}
\usepackage{ stmaryrd }

\usepackage[Lenny]{fncychap}

\usepackage[nottoc]{tocbibind} % pour que la bibliographie apparaisse dans la table des matières (avec l'option pour que la table des matières elle-même n'apparaisse pas dans la table des matières).
\usepackage{tocloft}% pour pouvoir modifier les tailles d'espacement dans la table des matières
\usepackage{textcomp} % rajoute des symboles (notes de musiques, etc.)
\usepackage{xcolor} % pour ajouter de la couleur (si besoin)
\usepackage{epigraph} % pour rajouter des citations en début de chapitre
\usepackage[
colorlinks=true,
urlcolor=purple,
linkcolor=purple!87!black,
citecolor=green!60!black,
pdfborder={0 0 0},
]{hyperref}
\usepackage{bookmark}

\theoremstyle{plain}
\newtheorem{theorem}{Theorem}[section] % reset theorem numbering for each chapter
\newtheorem{proposition}[theorem]{Proposition}
\newtheorem{lemma}[theorem]{Lemma}
\newtheorem{corollary}[theorem]{Corollary}
\theoremstyle{definition}
\newtheorem{remark}[theorem]{Remark}
\newtheorem{definition}[theorem]{Definition} % definition numbers are dependent on theorem numbers
\newtheorem{example}[theorem]{Example} % same for example numbers
\newtheorem{question}[theorem]{Question}
\newtheorem{conjecture}[theorem]{Conjecture}

\theoremstyle{plain}
\newtheorem{thmx}{Theorem}

\newcommand{\WP}{\textnormal{WP}}

\newcommand{\Geo}{\textnormal{Geo}}

\newcommand{\Aut}{\textnormal{Aut}}
\newcommand{\SAW}{\textnormal{SAW}}
\newcommand{\N}{\mathbb{N}}

\newcommand{\Z}{\mathbb{Z}}
\newcommand{\C}{\mathbb{C}}
\newcommand{\R}{\mathbb{R}}

\newcommand{\X}{\mathcal{X}}
\newcommand{\Ou}{\mathcal{O}}

\newcommand{\M}{\mathcal{M}}
\newcommand{\lang}{\mathcal{L}}
\newcommand{\D}{\mathcal{D}}
\newcommand{\F}{\mathbb{F}}
\newcommand{\Fo}{\mathcal{F}}
\newcommand{\ter}{\mathfrak{t}}
\newcommand{\init}{\mathfrak{i}}
\newcommand{\llangle}{\langle\langle}
\newcommand{\rrangle}{\rangle\rangle}
\newcommand{\factor}{\sqsubseteq}
\newcommand{\Conv}{\mathop{\scalebox{2}{\raisebox{-0.2ex}{$\ast$}}}}
\renewcommand{\tt}[1]{\mathtt{#1}}
\newcommand{\define}[1]{\textbf{#1}}
\newcommand{\argmin}{\textnormal{argmin}}
\definecolor{rouge}{RGB}{255,77,77}
\definecolor{vert}{RGB}{0,178,102}
\definecolor{jaune}{RGB}{255,255,0}
\definecolor{violet}{RGB}{208,32,144}
\definecolor{orange}{RGB}{255,140,0}
\definecolor{bleu}{RGB}{0,0,205}

\newcommand{\wang}[6]{
\draw [black,fill=#3] (#1,#2+1)  -- (#1+0.5,#2+0.5) -- (#1+1,#2+1) -- cycle;
\draw [black,fill=#4] (#1+1,#2+1)  -- (#1+0.5,#2+0.5) -- (#1+1,#2) -- cycle;
\draw [black,fill=#5] (#1,#2)  -- (#1+0.5,#2+0.5) -- (#1+1,#2) -- cycle;
\draw [black,fill=#6] (#1,#2)  -- (#1+0.5,#2+0.5) -- (#1,#2+1) -- cycle;
}

\newcommand{\marker}{\vbox to 12pt{\hbox to 15pt{
\hspace{-2pt}\includegraphics[width=15pt]{figures/youRhere.png}
}}}

\newcommand\hexagone[8]{ 
\begin{scope}[rotate around={#3:(#7,#8)},shift={(#1,#2)},scale=#6]
        \draw[color=#4] (0:\R) -- (60:\R);
        \draw[color=#5] (60:\R) -- (120:\R);
        \draw[color=#4] (120:\R) -- (180:\R);
        \draw[color=#5] (180:\R) -- (240:\R);
        \draw[color=#4] (240:\R) -- (300:\R);
        \draw[color=#5] (300:\R) -- (360:\R);
\end{scope}
}

\providecommand{\keywords}[1]{{\small\textit{Keywords:} #1}}

%% file: figures/sawIntro.tex
\tikzset{every picture/.style={line width=0.75pt}} %set default line width to 0.75pt        

\begin{tikzpicture}[x=0.75pt,y=0.75pt,yscale=-1,xscale=1]
%uncomment if require: \path (0,475); %set diagram left start at 0, and has height of 475

%Straight Lines [id:da7004311569093905] 
\draw [color={rgb, 255:red, 155; green, 155; blue, 155 }  ,draw opacity=1 ]   (84.86,157.48) -- (67.43,147.46) ;
%Straight Lines [id:da35473504622098073] 
\draw [color={rgb, 255:red, 155; green, 155; blue, 155 }  ,draw opacity=1 ] [dash pattern={on 0.84pt off 2.51pt}]  (67.43,147.46) -- (50,137.43) ;
%Straight Lines [id:da05897134127542014] 
\draw [color={rgb, 255:red, 155; green, 155; blue, 155 }  ,draw opacity=1 ]   (211.86,157.48) -- (229.29,147.46) ;
%Straight Lines [id:da44169999913917224] 
\draw [color={rgb, 255:red, 155; green, 155; blue, 155 }  ,draw opacity=1 ] [dash pattern={on 0.84pt off 2.51pt}]  (229.29,147.46) -- (246.72,137.43) ;
%Straight Lines [id:da8444222795126579] 
\draw [color={rgb, 255:red, 155; green, 155; blue, 155 }  ,draw opacity=1 ]   (180.11,212.47) -- (180.11,230.04) ;
%Straight Lines [id:da22976098241960363] 
\draw [color={rgb, 255:red, 155; green, 155; blue, 155 }  ,draw opacity=1 ] [dash pattern={on 0.84pt off 2.51pt}]  (180.11,229.07) -- (180.11,246.63) ;
%Straight Lines [id:da8589129944840116] 
\draw [color={rgb, 255:red, 208; green, 2; blue, 27 }  ,draw opacity=1 ][line width=1.5]    (148.36,84.16) -- (116.61,102.49) ;
%Straight Lines [id:da3166702567511903] 
\draw [color={rgb, 255:red, 155; green, 155; blue, 155 }  ,draw opacity=1 ]   (116.61,102.49) -- (116.61,139.15) ;
%Straight Lines [id:da9169699215875521] 
\draw [color={rgb, 255:red, 155; green, 155; blue, 155 }  ,draw opacity=1 ]   (116.61,139.15) -- (148.36,157.48) ;
%Straight Lines [id:da1244097919669952] 
\draw [color={rgb, 255:red, 155; green, 155; blue, 155 }  ,draw opacity=1 ]   (116.61,139.15) -- (84.86,157.48) ;
%Straight Lines [id:da0016259888336493944] 
\draw [color={rgb, 255:red, 155; green, 155; blue, 155 }  ,draw opacity=1 ]   (180.11,139.15) -- (148.36,157.48) ;
%Straight Lines [id:da9048198600646858] 
\draw [color={rgb, 255:red, 208; green, 2; blue, 27 }  ,draw opacity=1 ][line width=1.5]    (180.11,102.49) -- (180.11,139.15) ;
%Straight Lines [id:da05863346414709014] 
\draw [color={rgb, 255:red, 208; green, 2; blue, 27 }  ,draw opacity=1 ][line width=1.5]    (148.36,84.16) -- (180.11,102.49) ;
%Straight Lines [id:da2875557266593798] 
\draw [color={rgb, 255:red, 208; green, 2; blue, 27 }  ,draw opacity=1 ][line width=1.5]    (180.11,139.15) -- (211.86,157.48) ;
%Straight Lines [id:da8764770744454441] 
\draw [color={rgb, 255:red, 208; green, 2; blue, 27 }  ,draw opacity=1 ][line width=1.5]    (148.36,157.48) -- (148.36,194.14) ;
%Straight Lines [id:da30610085028292033] 
\draw [color={rgb, 255:red, 155; green, 155; blue, 155 }  ,draw opacity=1 ]   (84.86,157.48) -- (84.86,194.14) ;
%Straight Lines [id:da8167650396446337] 
\draw [color={rgb, 255:red, 155; green, 155; blue, 155 }  ,draw opacity=1 ]   (84.86,194.14) -- (116.61,212.47) ;
%Straight Lines [id:da5605101771991859] 
\draw [color={rgb, 255:red, 155; green, 155; blue, 155 }  ,draw opacity=1 ]   (148.36,194.14) -- (116.61,212.47) ;
%Straight Lines [id:da9819826737007649] 
\draw [color={rgb, 255:red, 208; green, 2; blue, 27 }  ,draw opacity=1 ][line width=1.5]    (148.36,194.14) -- (180.11,212.47) ;
%Straight Lines [id:da2608321569691524] 
\draw [color={rgb, 255:red, 208; green, 2; blue, 27 }  ,draw opacity=1 ][line width=1.5]    (211.86,194.14) -- (180.11,212.47) ;
%Straight Lines [id:da9880079222382528] 
\draw [color={rgb, 255:red, 208; green, 2; blue, 27 }  ,draw opacity=1 ][line width=1.5]    (211.86,157.48) -- (211.86,194.14) ;
%Straight Lines [id:da002170120243295104] 
\draw [color={rgb, 255:red, 155; green, 155; blue, 155 }  ,draw opacity=1 ]   (116.61,212.47) -- (116.61,230.04) ;
%Straight Lines [id:da8436035220208586] 
\draw [color={rgb, 255:red, 155; green, 155; blue, 155 }  ,draw opacity=1 ] [dash pattern={on 0.84pt off 2.51pt}]  (116.61,230.04) -- (116.61,247.6) ;
%Straight Lines [id:da8546643326935182] 
\draw [color={rgb, 255:red, 155; green, 155; blue, 155 }  ,draw opacity=1 ]   (84.86,194.14) -- (67.43,204.17) ;
%Straight Lines [id:da769271306158256] 
\draw [color={rgb, 255:red, 155; green, 155; blue, 155 }  ,draw opacity=1 ] [dash pattern={on 0.84pt off 2.51pt}]  (67.43,204.17) -- (50,214.2) ;
%Straight Lines [id:da9622258270582758] 
\draw [color={rgb, 255:red, 155; green, 155; blue, 155 }  ,draw opacity=1 ]   (180.11,102.49) -- (197.54,92.46) ;
%Straight Lines [id:da7388051593526146] 
\draw [color={rgb, 255:red, 155; green, 155; blue, 155 }  ,draw opacity=1 ] [dash pattern={on 0.84pt off 2.51pt}]  (197.54,92.46) -- (214.97,82.44) ;
%Straight Lines [id:da9899251639706753] 
\draw [color={rgb, 255:red, 155; green, 155; blue, 155 }  ,draw opacity=1 ]   (211.86,194.14) -- (229.29,204.17) ;
%Straight Lines [id:da3782379162184534] 
\draw [color={rgb, 255:red, 155; green, 155; blue, 155 }  ,draw opacity=1 ] [dash pattern={on 0.84pt off 2.51pt}]  (229.29,204.17) -- (246.72,214.2) ;
%Straight Lines [id:da4911061287103313] 
\draw [color={rgb, 255:red, 155; green, 155; blue, 155 }  ,draw opacity=1 ]   (116.61,102.49) -- (99.18,92.46) ;
%Straight Lines [id:da47833039605749383] 
\draw [color={rgb, 255:red, 155; green, 155; blue, 155 }  ,draw opacity=1 ] [dash pattern={on 0.84pt off 2.51pt}]  (99.18,92.46) -- (81.75,82.44) ;
%Straight Lines [id:da06350311593917413] 
\draw [color={rgb, 255:red, 155; green, 155; blue, 155 }  ,draw opacity=1 ]   (148.36,66.6) -- (148.36,84.16) ;
%Straight Lines [id:da28213459824733267] 
\draw [color={rgb, 255:red, 155; green, 155; blue, 155 }  ,draw opacity=1 ] [dash pattern={on 0.84pt off 2.51pt}]  (148.36,50) -- (148.36,67.56) ;
%Shape: Circle [id:dp9650354867315024] 
\draw  [fill={rgb, 255:red, 255; green, 255; blue, 255 }  ,fill opacity=1 ] (177.61,139.15) .. controls (177.61,137.77) and (178.73,136.65) .. (180.11,136.65) .. controls (181.49,136.65) and (182.61,137.77) .. (182.61,139.15) .. controls (182.61,140.53) and (181.49,141.65) .. (180.11,141.65) .. controls (178.73,141.65) and (177.61,140.53) .. (177.61,139.15) -- cycle ;
%Shape: Circle [id:dp1401983738009427] 
\draw  [fill={rgb, 255:red, 255; green, 255; blue, 255 }  ,fill opacity=1 ] (209.36,157.48) .. controls (209.36,156.1) and (210.48,154.98) .. (211.86,154.98) .. controls (213.24,154.98) and (214.36,156.1) .. (214.36,157.48) .. controls (214.36,158.86) and (213.24,159.98) .. (211.86,159.98) .. controls (210.48,159.98) and (209.36,158.86) .. (209.36,157.48) -- cycle ;
%Shape: Circle [id:dp5158486224970242] 
\draw  [fill={rgb, 255:red, 255; green, 255; blue, 255 }  ,fill opacity=1 ] (145.86,194.14) .. controls (145.86,192.76) and (146.98,191.64) .. (148.36,191.64) .. controls (149.74,191.64) and (150.86,192.76) .. (150.86,194.14) .. controls (150.86,195.52) and (149.74,196.64) .. (148.36,196.64) .. controls (146.98,196.64) and (145.86,195.52) .. (145.86,194.14) -- cycle ;
%Shape: Circle [id:dp18930437660315136] 
\draw  [fill={rgb, 255:red, 255; green, 255; blue, 255 }  ,fill opacity=1 ] (177.61,212.47) .. controls (177.61,211.09) and (178.73,209.97) .. (180.11,209.97) .. controls (181.49,209.97) and (182.61,211.09) .. (182.61,212.47) .. controls (182.61,213.85) and (181.49,214.97) .. (180.11,214.97) .. controls (178.73,214.97) and (177.61,213.85) .. (177.61,212.47) -- cycle ;
%Shape: Circle [id:dp5560639456812285] 
\draw  [fill={rgb, 255:red, 255; green, 255; blue, 255 }  ,fill opacity=1 ] (209.36,194.14) .. controls (209.36,192.76) and (210.48,191.64) .. (211.86,191.64) .. controls (213.24,191.64) and (214.36,192.76) .. (214.36,194.14) .. controls (214.36,195.52) and (213.24,196.64) .. (211.86,196.64) .. controls (210.48,196.64) and (209.36,195.52) .. (209.36,194.14) -- cycle ;
%Shape: Circle [id:dp2823513295095571] 
\draw  [fill={rgb, 255:red, 255; green, 255; blue, 255 }  ,fill opacity=1 ] (114.11,139.15) .. controls (114.11,137.77) and (115.23,136.65) .. (116.61,136.65) .. controls (117.99,136.65) and (119.11,137.77) .. (119.11,139.15) .. controls (119.11,140.53) and (117.99,141.65) .. (116.61,141.65) .. controls (115.23,141.65) and (114.11,140.53) .. (114.11,139.15) -- cycle ;
%Shape: Circle [id:dp4434037615035916] 
\draw  [fill={rgb, 255:red, 255; green, 255; blue, 255 }  ,fill opacity=1 ] (114.11,212.47) .. controls (114.11,211.09) and (115.23,209.97) .. (116.61,209.97) .. controls (117.99,209.97) and (119.11,211.09) .. (119.11,212.47) .. controls (119.11,213.85) and (117.99,214.97) .. (116.61,214.97) .. controls (115.23,214.97) and (114.11,213.85) .. (114.11,212.47) -- cycle ;
%Shape: Circle [id:dp7829258307648314] 
\draw  [fill={rgb, 255:red, 255; green, 255; blue, 255 }  ,fill opacity=1 ] (82.36,194.14) .. controls (82.36,192.76) and (83.48,191.64) .. (84.86,191.64) .. controls (86.24,191.64) and (87.36,192.76) .. (87.36,194.14) .. controls (87.36,195.52) and (86.24,196.64) .. (84.86,196.64) .. controls (83.48,196.64) and (82.36,195.52) .. (82.36,194.14) -- cycle ;
%Shape: Circle [id:dp9384840098741157] 
\draw  [fill={rgb, 255:red, 255; green, 255; blue, 255 }  ,fill opacity=1 ] (82.36,157.48) .. controls (82.36,156.1) and (83.48,154.98) .. (84.86,154.98) .. controls (86.24,154.98) and (87.36,156.1) .. (87.36,157.48) .. controls (87.36,158.86) and (86.24,159.98) .. (84.86,159.98) .. controls (83.48,159.98) and (82.36,158.86) .. (82.36,157.48) -- cycle ;
%Shape: Circle [id:dp27871924243647683] 
\draw  [fill={rgb, 255:red, 255; green, 255; blue, 255 }  ,fill opacity=1 ] (177.61,102.49) .. controls (177.61,101.11) and (178.73,99.99) .. (180.11,99.99) .. controls (181.49,99.99) and (182.61,101.11) .. (182.61,102.49) .. controls (182.61,103.87) and (181.49,104.99) .. (180.11,104.99) .. controls (178.73,104.99) and (177.61,103.87) .. (177.61,102.49) -- cycle ;
%Shape: Circle [id:dp8118654049987479] 
\draw  [fill={rgb, 255:red, 255; green, 255; blue, 255 }  ,fill opacity=1 ] (145.86,84.16) .. controls (145.86,82.78) and (146.98,81.66) .. (148.36,81.66) .. controls (149.74,81.66) and (150.86,82.78) .. (150.86,84.16) .. controls (150.86,85.54) and (149.74,86.66) .. (148.36,86.66) .. controls (146.98,86.66) and (145.86,85.54) .. (145.86,84.16) -- cycle ;
%Shape: Circle [id:dp946776617455901] 
\draw  [fill={rgb, 255:red, 255; green, 255; blue, 255 }  ,fill opacity=1 ] (114.11,102.49) .. controls (114.11,101.11) and (115.23,99.99) .. (116.61,99.99) .. controls (117.99,99.99) and (119.11,101.11) .. (119.11,102.49) .. controls (119.11,103.87) and (117.99,104.99) .. (116.61,104.99) .. controls (115.23,104.99) and (114.11,103.87) .. (114.11,102.49) -- cycle ;
%Shape: Donut [id:dp5041938700809308] 
\draw  [fill={rgb, 255:red, 255; green, 255; blue, 255 }  ,fill opacity=1 ,even odd rule] (147.36,157.48) .. controls (147.36,156.93) and (147.81,156.48) .. (148.36,156.48) .. controls (148.91,156.48) and (149.36,156.93) .. (149.36,157.48) .. controls (149.36,158.03) and (148.91,158.48) .. (148.36,158.48) .. controls (147.81,158.48) and (147.36,158.03) .. (147.36,157.48)(145.86,157.48) .. controls (145.86,156.1) and (146.98,154.98) .. (148.36,154.98) .. controls (149.74,154.98) and (150.86,156.1) .. (150.86,157.48) .. controls (150.86,158.86) and (149.74,159.98) .. (148.36,159.98) .. controls (146.98,159.98) and (145.86,158.86) .. (145.86,157.48) ;

\end{tikzpicture}

%% file: figures/honeycomb.tex
\tikzset{every picture/.style={line width=0.75pt}} %set default line width to 0.75pt        

\begin{tikzpicture}[x=0.75pt,y=0.75pt,yscale=-1,xscale=1]
%uncomment if require: \path (0,475); %set diagram left start at 0, and has height of 475

%Straight Lines [id:da8818995249330956] 
\draw [color={rgb, 255:red, 208; green, 2; blue, 27 }  ,draw opacity=1 ]   (215.25,120) -- (183.5,138.33) ;
%Straight Lines [id:da7730383052787946] 
\draw [color={rgb, 255:red, 208; green, 2; blue, 27 }  ,draw opacity=1 ]   (247,138.33) -- (247,174.99) ;
%Straight Lines [id:da540077566613743] 
\draw [color={rgb, 255:red, 74; green, 144; blue, 226 }  ,draw opacity=1 ]   (183.5,138.33) -- (183.5,174.99) ;
%Straight Lines [id:da5732827649557044] 
\draw [color={rgb, 255:red, 74; green, 144; blue, 226 }  ,draw opacity=1 ]   (247,174.99) -- (215.25,193.32) ;
%Straight Lines [id:da034103912225434074] 
\draw [color={rgb, 255:red, 208; green, 2; blue, 27 }  ,draw opacity=1 ]   (183.5,174.99) -- (215.25,193.32) ;
%Straight Lines [id:da0057799779093603565] 
\draw [color={rgb, 255:red, 74; green, 144; blue, 226 }  ,draw opacity=1 ]   (215.25,120) -- (247,138.33) ;
%Straight Lines [id:da6213481868013367] 
\draw [color={rgb, 255:red, 126; green, 211; blue, 33 }  ,draw opacity=1 ]   (215.25,193.32) -- (215.25,229.98) ;
%Straight Lines [id:da65982385000216] 
\draw [color={rgb, 255:red, 74; green, 144; blue, 226 }  ,draw opacity=1 ]   (215.25,229.98) -- (247,248.31) ;
%Straight Lines [id:da4069564463460169] 
\draw [color={rgb, 255:red, 208; green, 2; blue, 27 }  ,draw opacity=1 ]   (215.25,229.98) -- (183.5,248.31) ;
%Straight Lines [id:da8757802739771909] 
\draw [color={rgb, 255:red, 126; green, 211; blue, 33 }  ,draw opacity=1 ]   (278.75,229.98) -- (247,248.31) ;
%Straight Lines [id:da12499743682492126] 
\draw [color={rgb, 255:red, 74; green, 144; blue, 226 }  ,draw opacity=1 ]   (278.75,193.32) -- (278.75,229.98) ;
%Straight Lines [id:da5825243012033922] 
\draw [color={rgb, 255:red, 126; green, 211; blue, 33 }  ,draw opacity=1 ]   (247,174.99) -- (278.75,193.32) ;
%Straight Lines [id:da9688677297765089] 
\draw [color={rgb, 255:red, 126; green, 211; blue, 33 }  ,draw opacity=1 ]   (151.75,229.98) -- (183.5,248.31) ;
%Straight Lines [id:da3611320701520859] 
\draw [color={rgb, 255:red, 208; green, 2; blue, 27 }  ,draw opacity=1 ]   (151.75,193.32) -- (151.75,229.98) ;
%Straight Lines [id:da04429217087590753] 
\draw [color={rgb, 255:red, 126; green, 211; blue, 33 }  ,draw opacity=1 ]   (183.5,174.99) -- (151.75,193.32) ;
%Straight Lines [id:da3117073963140926] 
\draw [color={rgb, 255:red, 208; green, 2; blue, 27 }  ,draw opacity=1 ]   (278.75,120) -- (310.49,138.33) ;
%Straight Lines [id:da10755922292050546] 
\draw [color={rgb, 255:red, 126; green, 211; blue, 33 }  ,draw opacity=1 ]   (310.49,138.33) -- (310.49,174.99) ;
%Straight Lines [id:da4449354963291968] 
\draw [color={rgb, 255:red, 208; green, 2; blue, 27 }  ,draw opacity=1 ]   (310.49,174.99) -- (278.75,193.32) ;
%Straight Lines [id:da18728235150278327] 
\draw [color={rgb, 255:red, 126; green, 211; blue, 33 }  ,draw opacity=1 ]   (278.75,120) -- (247,138.33) ;
%Straight Lines [id:da14268623131114666] 
\draw [color={rgb, 255:red, 74; green, 144; blue, 226 }  ,draw opacity=1 ]   (120,174.99) -- (151.75,193.32) ;
%Straight Lines [id:da1788631880248559] 
\draw [color={rgb, 255:red, 126; green, 211; blue, 33 }  ,draw opacity=1 ]   (120,138.33) -- (120,174.99) ;
%Straight Lines [id:da061287625083469854] 
\draw [color={rgb, 255:red, 126; green, 211; blue, 33 }  ,draw opacity=1 ]   (151.75,120) -- (183.5,138.33) ;
%Straight Lines [id:da44748471279866797] 
\draw [color={rgb, 255:red, 74; green, 144; blue, 226 }  ,draw opacity=1 ]   (151.75,120) -- (120,138.33) ;
%Straight Lines [id:da6214305191666188] 
\draw [color={rgb, 255:red, 208; green, 2; blue, 27 }  ,draw opacity=1 ]   (278.75,229.98) -- (310.49,248.31) ;
%Straight Lines [id:da8280462165836152] 
\draw [color={rgb, 255:red, 208; green, 2; blue, 27 }  ,draw opacity=1 ]   (247,248.31) -- (247,284.97) ;
%Straight Lines [id:da16891465447698473] 
\draw [color={rgb, 255:red, 74; green, 144; blue, 226 }  ,draw opacity=1 ]   (183.5,248.31) -- (183.5,284.97) ;
%Straight Lines [id:da21517110511219928] 
\draw [color={rgb, 255:red, 208; green, 2; blue, 27 }  ,draw opacity=1 ]   (183.5,284.97) -- (215.25,303.3) ;
%Straight Lines [id:da6538576377050014] 
\draw [color={rgb, 255:red, 74; green, 144; blue, 226 }  ,draw opacity=1 ]   (247,284.97) -- (215.25,303.3) ;
%Straight Lines [id:da5235687246292379] 
\draw [color={rgb, 255:red, 126; green, 211; blue, 33 }  ,draw opacity=1 ]   (247,284.97) -- (278.75,303.3) ;
%Straight Lines [id:da4640526696107621] 
\draw [color={rgb, 255:red, 208; green, 2; blue, 27 }  ,draw opacity=1 ]   (310.49,284.97) -- (278.75,303.3) ;
%Straight Lines [id:da17598779855370883] 
\draw [color={rgb, 255:red, 126; green, 211; blue, 33 }  ,draw opacity=1 ]   (310.49,248.31) -- (310.49,284.97) ;
%Straight Lines [id:da6597071721626545] 
\draw [color={rgb, 255:red, 74; green, 144; blue, 226 }  ,draw opacity=1 ]   (342.24,229.98) -- (310.49,248.31) ;
%Straight Lines [id:da1450187945201542] 
\draw [color={rgb, 255:red, 208; green, 2; blue, 27 }  ,draw opacity=1 ]   (342.24,193.32) -- (342.24,229.98) ;
%Straight Lines [id:da5977608645060055] 
\draw [color={rgb, 255:red, 74; green, 144; blue, 226 }  ,draw opacity=1 ]   (310.49,174.99) -- (342.24,193.32) ;
%Straight Lines [id:da6482838525227147] 
\draw [color={rgb, 255:red, 74; green, 144; blue, 226 }  ,draw opacity=1 ]   (342.24,120) -- (310.49,138.33) ;
%Straight Lines [id:da4887774064564542] 
\draw [color={rgb, 255:red, 126; green, 211; blue, 33 }  ,draw opacity=1 ]   (373.99,174.99) -- (342.24,193.32) ;
%Straight Lines [id:da7058093556615844] 
\draw [color={rgb, 255:red, 74; green, 144; blue, 226 }  ,draw opacity=1 ]   (373.99,138.33) -- (373.99,174.99) ;
%Straight Lines [id:da5352216077078682] 
\draw [color={rgb, 255:red, 126; green, 211; blue, 33 }  ,draw opacity=1 ]   (342.24,120) -- (373.99,138.33) ;
%Straight Lines [id:da8318113408511675] 
\draw [color={rgb, 255:red, 126; green, 211; blue, 33 }  ,draw opacity=1 ]   (215.25,303.3) -- (215.25,339.96) ;
%Straight Lines [id:da5512352460141743] 
\draw [color={rgb, 255:red, 74; green, 144; blue, 226 }  ,draw opacity=1 ]   (278.75,303.3) -- (278.75,339.96) ;
%Straight Lines [id:da4269191142459966] 
\draw [color={rgb, 255:red, 74; green, 144; blue, 226 }  ,draw opacity=1 ]   (215.25,339.96) -- (247,358.29) ;
%Straight Lines [id:da7251515398703893] 
\draw [color={rgb, 255:red, 126; green, 211; blue, 33 }  ,draw opacity=1 ]   (278.75,339.96) -- (247,358.29) ;
%Straight Lines [id:da24752933374380226] 
\draw [color={rgb, 255:red, 208; green, 2; blue, 27 }  ,draw opacity=1 ]   (247,358.29) -- (247,375.86) ;
%Straight Lines [id:da6838395008358663] 
\draw [color={rgb, 255:red, 208; green, 2; blue, 27 }  ,draw opacity=1 ]   (215.25,339.96) -- (197.82,349.99) ;
%Straight Lines [id:da21421589415696662] 
\draw [color={rgb, 255:red, 208; green, 2; blue, 27 }  ,draw opacity=1 ] [dash pattern={on 0.84pt off 2.51pt}]  (197.82,349.99) -- (180.38,360.02) ;

%Straight Lines [id:da015141680053041173] 
\draw [color={rgb, 255:red, 208; green, 2; blue, 27 }  ,draw opacity=1 ] [dash pattern={on 0.84pt off 2.51pt}]  (247,375.86) -- (247,393.42) ;
%Straight Lines [id:da4654227563768648] 
\draw [color={rgb, 255:red, 126; green, 211; blue, 33 }  ,draw opacity=1 ]   (183.5,284.97) -- (166.07,295) ;
%Straight Lines [id:da8293124703694876] 
\draw [color={rgb, 255:red, 126; green, 211; blue, 33 }  ,draw opacity=1 ] [dash pattern={on 0.84pt off 2.51pt}]  (166.07,295) -- (148.64,305.03) ;
%Straight Lines [id:da4203253127486267] 
\draw [color={rgb, 255:red, 74; green, 144; blue, 226 }  ,draw opacity=1 ]   (151.75,229.98) -- (134.32,240.01) ;
%Straight Lines [id:da8028964380389262] 
\draw [color={rgb, 255:red, 74; green, 144; blue, 226 }  ,draw opacity=1 ] [dash pattern={on 0.84pt off 2.51pt}]  (134.32,240.01) -- (116.89,250.04) ;
%Straight Lines [id:da9703107973661097] 
\draw [color={rgb, 255:red, 208; green, 2; blue, 27 }  ,draw opacity=1 ]   (120,174.99) -- (102.57,185.02) ;
%Straight Lines [id:da7173925192201864] 
\draw [color={rgb, 255:red, 208; green, 2; blue, 27 }  ,draw opacity=1 ] [dash pattern={on 0.84pt off 2.51pt}]  (102.57,185.02) -- (85.14,195.04) ;
%Straight Lines [id:da38421511237141626] 
\draw [color={rgb, 255:red, 208; green, 2; blue, 27 }  ,draw opacity=1 ]   (373.99,138.33) -- (391.42,128.3) ;
%Straight Lines [id:da2635107888453535] 
\draw [color={rgb, 255:red, 208; green, 2; blue, 27 }  ,draw opacity=1 ] [dash pattern={on 0.84pt off 2.51pt}]  (391.42,128.3) -- (408.86,118.28) ;
%Straight Lines [id:da8494137530255413] 
\draw [color={rgb, 255:red, 208; green, 2; blue, 27 }  ,draw opacity=1 ]   (278.75,339.96) -- (296.18,349.99) ;
%Straight Lines [id:da3497044554837009] 
\draw [color={rgb, 255:red, 208; green, 2; blue, 27 }  ,draw opacity=1 ] [dash pattern={on 0.84pt off 2.51pt}]  (296.18,349.99) -- (313.61,360.02) ;
%Straight Lines [id:da4897018716753814] 
\draw [color={rgb, 255:red, 74; green, 144; blue, 226 }  ,draw opacity=1 ]   (310.49,284.97) -- (327.93,295) ;
%Straight Lines [id:da5429400954695388] 
\draw [color={rgb, 255:red, 74; green, 144; blue, 226 }  ,draw opacity=1 ] [dash pattern={on 0.84pt off 2.51pt}]  (327.93,295) -- (345.36,305.03) ;
%Straight Lines [id:da950862942311638] 
\draw [color={rgb, 255:red, 126; green, 211; blue, 33 }  ,draw opacity=1 ]   (342.24,229.98) -- (359.67,240.01) ;
%Straight Lines [id:da029259257469373523] 
\draw [color={rgb, 255:red, 126; green, 211; blue, 33 }  ,draw opacity=1 ] [dash pattern={on 0.84pt off 2.51pt}]  (359.67,240.01) -- (377.11,250.04) ;
%Straight Lines [id:da6001653933783908] 
\draw [color={rgb, 255:red, 208; green, 2; blue, 27 }  ,draw opacity=1 ]   (373.99,174.99) -- (391.42,185.02) ;
%Straight Lines [id:da12747118517444056] 
\draw [color={rgb, 255:red, 208; green, 2; blue, 27 }  ,draw opacity=1 ] [dash pattern={on 0.84pt off 2.51pt}]  (391.42,185.02) -- (408.86,195.04) ;
%Straight Lines [id:da8145475928297862] 
\draw [color={rgb, 255:red, 208; green, 2; blue, 27 }  ,draw opacity=1 ]   (120,138.33) -- (102.57,128.3) ;
%Straight Lines [id:da27421103501203914] 
\draw [color={rgb, 255:red, 208; green, 2; blue, 27 }  ,draw opacity=1 ] [dash pattern={on 0.84pt off 2.51pt}]  (102.57,128.3) -- (85.14,118.28) ;
%Straight Lines [id:da9662806540584878] 
\draw [color={rgb, 255:red, 208; green, 2; blue, 27 }  ,draw opacity=1 ]   (151.75,102.44) -- (151.75,120) ;
%Straight Lines [id:da8754783500811363] 
\draw [color={rgb, 255:red, 208; green, 2; blue, 27 }  ,draw opacity=1 ] [dash pattern={on 0.84pt off 2.51pt}]  (151.75,84.88) -- (151.75,102.44) ;
%Straight Lines [id:da5306822789892658] 
\draw [color={rgb, 255:red, 208; green, 2; blue, 27 }  ,draw opacity=1 ]   (342.24,102.44) -- (342.24,120) ;
%Straight Lines [id:da8398768204024339] 
\draw [color={rgb, 255:red, 208; green, 2; blue, 27 }  ,draw opacity=1 ] [dash pattern={on 0.84pt off 2.51pt}]  (342.24,84.88) -- (342.24,102.44) ;
%Straight Lines [id:da9537432184752852] 
\draw [color={rgb, 255:red, 74; green, 144; blue, 226 }  ,draw opacity=1 ]   (278.75,102.44) -- (278.75,120) ;
%Straight Lines [id:da7021951202077764] 
\draw [color={rgb, 255:red, 74; green, 144; blue, 226 }  ,draw opacity=1 ] [dash pattern={on 0.84pt off 2.51pt}]  (278.75,84.88) -- (278.75,102.44) ;
%Straight Lines [id:da6807043250186969] 
\draw [color={rgb, 255:red, 126; green, 211; blue, 33 }  ,draw opacity=1 ]   (215.25,102.44) -- (215.25,120) ;
%Straight Lines [id:da8206603272259791] 
\draw [color={rgb, 255:red, 126; green, 211; blue, 33 }  ,draw opacity=1 ] [dash pattern={on 0.84pt off 2.51pt}]  (215.25,84.88) -- (215.25,102.44) ;
%Shape: Circle [id:dp9006219973469916] 
\draw  [fill={rgb, 255:red, 255; green, 255; blue, 255 }  ,fill opacity=1 ] (276.25,229.98) .. controls (276.25,228.6) and (277.36,227.48) .. (278.75,227.48) .. controls (280.13,227.48) and (281.25,228.6) .. (281.25,229.98) .. controls (281.25,231.36) and (280.13,232.48) .. (278.75,232.48) .. controls (277.36,232.48) and (276.25,231.36) .. (276.25,229.98) -- cycle ;
%Shape: Circle [id:dp7474700302568121] 
\draw  [fill={rgb, 255:red, 255; green, 255; blue, 255 }  ,fill opacity=1 ] (307.99,248.31) .. controls (307.99,246.93) and (309.11,245.81) .. (310.49,245.81) .. controls (311.88,245.81) and (312.99,246.93) .. (312.99,248.31) .. controls (312.99,249.69) and (311.88,250.81) .. (310.49,250.81) .. controls (309.11,250.81) and (307.99,249.69) .. (307.99,248.31) -- cycle ;
%Shape: Circle [id:dp9643748465912583] 
\draw  [fill={rgb, 255:red, 255; green, 255; blue, 255 }  ,fill opacity=1 ] (244.5,284.97) .. controls (244.5,283.59) and (245.62,282.47) .. (247,282.47) .. controls (248.38,282.47) and (249.5,283.59) .. (249.5,284.97) .. controls (249.5,286.35) and (248.38,287.47) .. (247,287.47) .. controls (245.62,287.47) and (244.5,286.35) .. (244.5,284.97) -- cycle ;
%Shape: Circle [id:dp16348395422884077] 
\draw  [fill={rgb, 255:red, 255; green, 255; blue, 255 }  ,fill opacity=1 ] (276.25,303.3) .. controls (276.25,301.92) and (277.36,300.8) .. (278.75,300.8) .. controls (280.13,300.8) and (281.25,301.92) .. (281.25,303.3) .. controls (281.25,304.68) and (280.13,305.8) .. (278.75,305.8) .. controls (277.36,305.8) and (276.25,304.68) .. (276.25,303.3) -- cycle ;
%Shape: Circle [id:dp2616684707561264] 
\draw  [fill={rgb, 255:red, 255; green, 255; blue, 255 }  ,fill opacity=1 ] (307.99,284.97) .. controls (307.99,283.59) and (309.11,282.47) .. (310.49,282.47) .. controls (311.88,282.47) and (312.99,283.59) .. (312.99,284.97) .. controls (312.99,286.35) and (311.88,287.47) .. (310.49,287.47) .. controls (309.11,287.47) and (307.99,286.35) .. (307.99,284.97) -- cycle ;
%Shape: Circle [id:dp29997742339470845] 
\draw  [fill={rgb, 255:red, 255; green, 255; blue, 255 }  ,fill opacity=1 ] (212.75,229.98) .. controls (212.75,228.6) and (213.87,227.48) .. (215.25,227.48) .. controls (216.63,227.48) and (217.75,228.6) .. (217.75,229.98) .. controls (217.75,231.36) and (216.63,232.48) .. (215.25,232.48) .. controls (213.87,232.48) and (212.75,231.36) .. (212.75,229.98) -- cycle ;
%Shape: Circle [id:dp26657723337909467] 
\draw  [fill={rgb, 255:red, 255; green, 255; blue, 255 }  ,fill opacity=1 ] (212.75,303.3) .. controls (212.75,301.92) and (213.87,300.8) .. (215.25,300.8) .. controls (216.63,300.8) and (217.75,301.92) .. (217.75,303.3) .. controls (217.75,304.68) and (216.63,305.8) .. (215.25,305.8) .. controls (213.87,305.8) and (212.75,304.68) .. (212.75,303.3) -- cycle ;
%Shape: Circle [id:dp9791442972113086] 
\draw  [fill={rgb, 255:red, 255; green, 255; blue, 255 }  ,fill opacity=1 ] (181,284.97) .. controls (181,283.59) and (182.12,282.47) .. (183.5,282.47) .. controls (184.88,282.47) and (186,283.59) .. (186,284.97) .. controls (186,286.35) and (184.88,287.47) .. (183.5,287.47) .. controls (182.12,287.47) and (181,286.35) .. (181,284.97) -- cycle ;
%Shape: Circle [id:dp8757789165020179] 
\draw  [fill={rgb, 255:red, 255; green, 255; blue, 255 }  ,fill opacity=1 ] (181,248.31) .. controls (181,246.93) and (182.12,245.81) .. (183.5,245.81) .. controls (184.88,245.81) and (186,246.93) .. (186,248.31) .. controls (186,249.69) and (184.88,250.81) .. (183.5,250.81) .. controls (182.12,250.81) and (181,249.69) .. (181,248.31) -- cycle ;
%Shape: Circle [id:dp4495363486215599] 
\draw  [fill={rgb, 255:red, 255; green, 255; blue, 255 }  ,fill opacity=1 ] (212.75,339.96) .. controls (212.75,338.58) and (213.87,337.46) .. (215.25,337.46) .. controls (216.63,337.46) and (217.75,338.58) .. (217.75,339.96) .. controls (217.75,341.34) and (216.63,342.46) .. (215.25,342.46) .. controls (213.87,342.46) and (212.75,341.34) .. (212.75,339.96) -- cycle ;
%Shape: Circle [id:dp6103892302736628] 
\draw  [fill={rgb, 255:red, 255; green, 255; blue, 255 }  ,fill opacity=1 ] (244.5,358.29) .. controls (244.5,356.91) and (245.62,355.79) .. (247,355.79) .. controls (248.38,355.79) and (249.5,356.91) .. (249.5,358.29) .. controls (249.5,359.68) and (248.38,360.79) .. (247,360.79) .. controls (245.62,360.79) and (244.5,359.68) .. (244.5,358.29) -- cycle ;
%Shape: Circle [id:dp4489511375221099] 
\draw  [fill={rgb, 255:red, 255; green, 255; blue, 255 }  ,fill opacity=1 ] (276.25,339.96) .. controls (276.25,338.58) and (277.36,337.46) .. (278.75,337.46) .. controls (280.13,337.46) and (281.25,338.58) .. (281.25,339.96) .. controls (281.25,341.34) and (280.13,342.46) .. (278.75,342.46) .. controls (277.36,342.46) and (276.25,341.34) .. (276.25,339.96) -- cycle ;
%Shape: Circle [id:dp8548592989673184] 
\draw  [fill={rgb, 255:red, 255; green, 255; blue, 255 }  ,fill opacity=1 ] (339.74,229.98) .. controls (339.74,228.6) and (340.86,227.48) .. (342.24,227.48) .. controls (343.62,227.48) and (344.74,228.6) .. (344.74,229.98) .. controls (344.74,231.36) and (343.62,232.48) .. (342.24,232.48) .. controls (340.86,232.48) and (339.74,231.36) .. (339.74,229.98) -- cycle ;
%Shape: Circle [id:dp24182879724132367] 
\draw  [fill={rgb, 255:red, 255; green, 255; blue, 255 }  ,fill opacity=1 ] (276.25,193.32) .. controls (276.25,191.94) and (277.36,190.82) .. (278.75,190.82) .. controls (280.13,190.82) and (281.25,191.94) .. (281.25,193.32) .. controls (281.25,194.7) and (280.13,195.82) .. (278.75,195.82) .. controls (277.36,195.82) and (276.25,194.7) .. (276.25,193.32) -- cycle ;
%Shape: Circle [id:dp8347536497529926] 
\draw  [fill={rgb, 255:red, 255; green, 255; blue, 255 }  ,fill opacity=1 ] (244.5,174.99) .. controls (244.5,173.61) and (245.62,172.49) .. (247,172.49) .. controls (248.38,172.49) and (249.5,173.61) .. (249.5,174.99) .. controls (249.5,176.37) and (248.38,177.49) .. (247,177.49) .. controls (245.62,177.49) and (244.5,176.37) .. (244.5,174.99) -- cycle ;
%Shape: Circle [id:dp5775220124613007] 
\draw  [fill={rgb, 255:red, 255; green, 255; blue, 255 }  ,fill opacity=1 ] (212.75,193.32) .. controls (212.75,191.94) and (213.87,190.82) .. (215.25,190.82) .. controls (216.63,190.82) and (217.75,191.94) .. (217.75,193.32) .. controls (217.75,194.7) and (216.63,195.82) .. (215.25,195.82) .. controls (213.87,195.82) and (212.75,194.7) .. (212.75,193.32) -- cycle ;
%Shape: Circle [id:dp7667678968088686] 
\draw  [fill={rgb, 255:red, 255; green, 255; blue, 255 }  ,fill opacity=1 ] (181,174.99) .. controls (181,173.61) and (182.12,172.49) .. (183.5,172.49) .. controls (184.88,172.49) and (186,173.61) .. (186,174.99) .. controls (186,176.37) and (184.88,177.49) .. (183.5,177.49) .. controls (182.12,177.49) and (181,176.37) .. (181,174.99) -- cycle ;
%Shape: Circle [id:dp4577731633328084] 
\draw  [fill={rgb, 255:red, 255; green, 255; blue, 255 }  ,fill opacity=1 ] (149.25,193.32) .. controls (149.25,191.94) and (150.37,190.82) .. (151.75,190.82) .. controls (153.13,190.82) and (154.25,191.94) .. (154.25,193.32) .. controls (154.25,194.7) and (153.13,195.82) .. (151.75,195.82) .. controls (150.37,195.82) and (149.25,194.7) .. (149.25,193.32) -- cycle ;
%Shape: Circle [id:dp37507756823071703] 
\draw  [fill={rgb, 255:red, 255; green, 255; blue, 255 }  ,fill opacity=1 ] (149.25,229.98) .. controls (149.25,228.6) and (150.37,227.48) .. (151.75,227.48) .. controls (153.13,227.48) and (154.25,228.6) .. (154.25,229.98) .. controls (154.25,231.36) and (153.13,232.48) .. (151.75,232.48) .. controls (150.37,232.48) and (149.25,231.36) .. (149.25,229.98) -- cycle ;
%Shape: Circle [id:dp5517023568751608] 
\draw  [fill={rgb, 255:red, 255; green, 255; blue, 255 }  ,fill opacity=1 ] (117.5,174.99) .. controls (117.5,173.61) and (118.62,172.49) .. (120,172.49) .. controls (121.38,172.49) and (122.5,173.61) .. (122.5,174.99) .. controls (122.5,176.37) and (121.38,177.49) .. (120,177.49) .. controls (118.62,177.49) and (117.5,176.37) .. (117.5,174.99) -- cycle ;
%Shape: Circle [id:dp48801529018916623] 
\draw  [fill={rgb, 255:red, 255; green, 255; blue, 255 }  ,fill opacity=1 ] (117.5,138.33) .. controls (117.5,136.95) and (118.62,135.83) .. (120,135.83) .. controls (121.38,135.83) and (122.5,136.95) .. (122.5,138.33) .. controls (122.5,139.71) and (121.38,140.83) .. (120,140.83) .. controls (118.62,140.83) and (117.5,139.71) .. (117.5,138.33) -- cycle ;
%Shape: Circle [id:dp23358118962663388] 
\draw  [fill={rgb, 255:red, 255; green, 255; blue, 255 }  ,fill opacity=1 ] (149.25,120) .. controls (149.25,118.62) and (150.37,117.5) .. (151.75,117.5) .. controls (153.13,117.5) and (154.25,118.62) .. (154.25,120) .. controls (154.25,121.38) and (153.13,122.5) .. (151.75,122.5) .. controls (150.37,122.5) and (149.25,121.38) .. (149.25,120) -- cycle ;
%Shape: Circle [id:dp7218786886068276] 
\draw  [fill={rgb, 255:red, 255; green, 255; blue, 255 }  ,fill opacity=1 ] (181,138.33) .. controls (181,136.95) and (182.12,135.83) .. (183.5,135.83) .. controls (184.88,135.83) and (186,136.95) .. (186,138.33) .. controls (186,139.71) and (184.88,140.83) .. (183.5,140.83) .. controls (182.12,140.83) and (181,139.71) .. (181,138.33) -- cycle ;
%Shape: Circle [id:dp5622353233953006] 
\draw  [fill={rgb, 255:red, 255; green, 255; blue, 255 }  ,fill opacity=1 ] (212.75,120) .. controls (212.75,118.62) and (213.87,117.5) .. (215.25,117.5) .. controls (216.63,117.5) and (217.75,118.62) .. (217.75,120) .. controls (217.75,121.38) and (216.63,122.5) .. (215.25,122.5) .. controls (213.87,122.5) and (212.75,121.38) .. (212.75,120) -- cycle ;
%Shape: Circle [id:dp4991725555176316] 
\draw  [fill={rgb, 255:red, 255; green, 255; blue, 255 }  ,fill opacity=1 ] (244.5,138.33) .. controls (244.5,136.95) and (245.62,135.83) .. (247,135.83) .. controls (248.38,135.83) and (249.5,136.95) .. (249.5,138.33) .. controls (249.5,139.71) and (248.38,140.83) .. (247,140.83) .. controls (245.62,140.83) and (244.5,139.71) .. (244.5,138.33) -- cycle ;
%Shape: Circle [id:dp20874928540687832] 
\draw  [fill={rgb, 255:red, 255; green, 255; blue, 255 }  ,fill opacity=1 ] (276.25,120) .. controls (276.25,118.62) and (277.36,117.5) .. (278.75,117.5) .. controls (280.13,117.5) and (281.25,118.62) .. (281.25,120) .. controls (281.25,121.38) and (280.13,122.5) .. (278.75,122.5) .. controls (277.36,122.5) and (276.25,121.38) .. (276.25,120) -- cycle ;
%Shape: Circle [id:dp8636880292019704] 
\draw  [fill={rgb, 255:red, 255; green, 255; blue, 255 }  ,fill opacity=1 ] (307.99,138.33) .. controls (307.99,136.95) and (309.11,135.83) .. (310.49,135.83) .. controls (311.88,135.83) and (312.99,136.95) .. (312.99,138.33) .. controls (312.99,139.71) and (311.88,140.83) .. (310.49,140.83) .. controls (309.11,140.83) and (307.99,139.71) .. (307.99,138.33) -- cycle ;
%Shape: Circle [id:dp464547237289872] 
\draw  [fill={rgb, 255:red, 255; green, 255; blue, 255 }  ,fill opacity=1 ] (339.74,120) .. controls (339.74,118.62) and (340.86,117.5) .. (342.24,117.5) .. controls (343.62,117.5) and (344.74,118.62) .. (344.74,120) .. controls (344.74,121.38) and (343.62,122.5) .. (342.24,122.5) .. controls (340.86,122.5) and (339.74,121.38) .. (339.74,120) -- cycle ;
%Shape: Circle [id:dp2884800545256926] 
\draw  [fill={rgb, 255:red, 255; green, 255; blue, 255 }  ,fill opacity=1 ] (371.49,138.33) .. controls (371.49,136.95) and (372.61,135.83) .. (373.99,135.83) .. controls (375.37,135.83) and (376.49,136.95) .. (376.49,138.33) .. controls (376.49,139.71) and (375.37,140.83) .. (373.99,140.83) .. controls (372.61,140.83) and (371.49,139.71) .. (371.49,138.33) -- cycle ;
%Shape: Circle [id:dp5041011192212181] 
\draw  [fill={rgb, 255:red, 255; green, 255; blue, 255 }  ,fill opacity=1 ] (371.49,174.99) .. controls (371.49,173.61) and (372.61,172.49) .. (373.99,172.49) .. controls (375.37,172.49) and (376.49,173.61) .. (376.49,174.99) .. controls (376.49,176.37) and (375.37,177.49) .. (373.99,177.49) .. controls (372.61,177.49) and (371.49,176.37) .. (371.49,174.99) -- cycle ;
%Shape: Circle [id:dp11932054789069368] 
\draw  [fill={rgb, 255:red, 255; green, 255; blue, 255 }  ,fill opacity=1 ] (339.74,193.32) .. controls (339.74,191.94) and (340.86,190.82) .. (342.24,190.82) .. controls (343.62,190.82) and (344.74,191.94) .. (344.74,193.32) .. controls (344.74,194.7) and (343.62,195.82) .. (342.24,195.82) .. controls (340.86,195.82) and (339.74,194.7) .. (339.74,193.32) -- cycle ;
%Shape: Circle [id:dp7281044385382143] 
\draw  [fill={rgb, 255:red, 255; green, 255; blue, 255 }  ,fill opacity=1 ] (307.99,174.99) .. controls (307.99,173.61) and (309.11,172.49) .. (310.49,172.49) .. controls (311.88,172.49) and (312.99,173.61) .. (312.99,174.99) .. controls (312.99,176.37) and (311.88,177.49) .. (310.49,177.49) .. controls (309.11,177.49) and (307.99,176.37) .. (307.99,174.99) -- cycle ;
%Shape: Donut [id:dp878222033493569] 
\draw  [fill={rgb, 255:red, 255; green, 255; blue, 255 }  ,fill opacity=1 ,even odd rule] (246,248.31) .. controls (246,247.76) and (246.44,247.31) .. (247,247.31) .. controls (247.55,247.31) and (248,247.76) .. (248,248.31) .. controls (248,248.86) and (247.55,249.31) .. (247,249.31) .. controls (246.44,249.31) and (246,248.86) .. (246,248.31)(244.5,248.31) .. controls (244.5,246.93) and (245.62,245.81) .. (247,245.81) .. controls (248.38,245.81) and (249.5,246.93) .. (249.5,248.31) .. controls (249.5,249.69) and (248.38,250.81) .. (247,250.81) .. controls (245.62,250.81) and (244.5,249.69) .. (244.5,248.31) ;

\end{tikzpicture}

%% file: figures/ladder.tex
 
\begin{tikzpicture}
    \draw [gray,thick,-latex](0,1.5) -- (1,1.5);
    \draw [gray,thick,-latex](0,3) -- (1,3);

\foreach \x in {0,...,4}{
    \draw [gray,thick,-latex](1+2*\x,1.5) -- (3+2*\x,1.5);
    \draw [gray,thick,-latex](1+2*\x,3) -- (3+2*\x,3);
}
\foreach \x in {0,...,5}{
     \draw [gray](1+2*\x,1.5) -- (1+2*\x,3);
 }

\draw [gray,thick](11,1.5) -- (12,1.5);
\draw [gray,thick](11,3) -- (12,3);
 
% \draw [red,very thick, -latex](0,1.5) -- (1.5,1.5);
% \draw [red,very thick, -latex](1.5,1.5) -- (4.5,1.5);
% \draw [blue,very thick, -latex](4.5,1.5) -- (4.5,3);
% \draw [blue,very thick, -latex](4.5,3) -- (7.5,3);
% \draw [blue,very thick, -latex](7.5,3) -- (10.5,3);
% \draw [blue,very thick, -latex](10.5,3) -- (13.5,3);
% \draw [blue,very thick, -latex](13.5,3) -- (13.5,1.5);
% \draw [blue,very thick, -latex](13.5,1.5) -- (10.5,1.5);
% \draw [blue,very thick, -latex](10.5,1.5) -- (7.5,1.5);
% \draw [red,very thick, -latex](7.5,1.5) -- (6,1.5);

%Points
\foreach \x in {0,...,5}{
    \draw[fill=white]   (1 +2*\x,1.5) circle [radius=1.75pt];
}
\foreach \x in {0,...,5}{
    \draw[fill=white]   (1 + 2*\x, 3) circle [radius=1.75pt];
}

% \draw[fill=white]   (6, 1.5) circle [radius=1.75pt];

\draw (2,1.7) node [anchor=center]{$\tt{t}$};
\draw (3.2,2.25) node [anchor=center]{$\tt{s}$};

\end{tikzpicture}

%% file: figures/finite_SAW_Z2_non_infinite.tex
\begin{tikzpicture}[scale=0.6]

\draw[black!75] (0,0) grid (7,4);
\draw[very thick, color = bleu] (3,1) -- (1,1) -- (1,3) -- (6,3) -- (6,1) -- (4,1) -- (4,2);
\filldraw[color = bleu] (3,1) circle (3pt);
\filldraw[color = bleu] (4,2) circle (3pt);

\end{tikzpicture}

%% file: figures/regular_spanning_tree.tex
\begin{tikzpicture}[scale=0.25]

\draw[black!50] (0,0) grid (15,15);
\draw[very thick, color = bleu] (0,0) -- (1,0) -- (1,1) -- (0,1) -- (0,3) -- (1,3) -- (1,2) -- (2,2) -- (2,3) -- (3,3) -- (3,1) -- (2,1) -- (2,0) -- (4,0) -- (4,1) -- (5,1) -- (5,0) -- (7,0) -- (7,1) -- (6,1) -- (6,2) -- (7,2) -- (7,3) -- (5,3) -- (5,2) -- (4,2) -- (4,5) -- (5,5) -- (5,4) -- (7,4) -- (7,5) -- (6,5) -- (6,6) -- (7,6) -- (7,7) -- (5,7) -- (5,6) -- (4,6) -- (4,7) -- (2,7) -- (2,6) -- (3,6) -- (3,4) -- (2,4) -- (2,5) -- (1,5) -- (1,4) -- (0,4) -- (0,6) -- (1,6) -- (1,7) -- (0,7);
\begin{scope}[shift={(7,8)},rotate=90]
\draw[very thick, color = bleu] (0,0) -- (1,0) -- (1,1) -- (0,1) -- (0,3) -- (1,3) -- (1,2) -- (2,2) -- (2,3) -- (3,3) -- (3,1) -- (2,1) -- (2,0) -- (4,0) -- (4,1) -- (5,1) -- (5,0) -- (7,0) -- (7,1) -- (6,1) -- (6,2) -- (7,2) -- (7,3) -- (5,3) -- (5,2) -- (4,2) -- (4,5) -- (5,5) -- (5,4) -- (7,4) -- (7,5) -- (6,5) -- (6,6) -- (7,6) -- (7,7) -- (5,7) -- (5,6) -- (4,6) -- (4,7) -- (2,7) -- (2,6) -- (3,6) -- (3,4) -- (2,4) -- (2,5) -- (1,5) -- (1,4) -- (0,4) -- (0,6) -- (1,6) -- (1,7) -- (0,7);
\end{scope}
\begin{scope}[shift={(15,7)},rotate=180]
\draw[very thick, color = bleu] (0,0) -- (1,0) -- (1,1) -- (0,1) -- (0,3) -- (1,3) -- (1,2) -- (2,2) -- (2,3) -- (3,3) -- (3,1) -- (2,1) -- (2,0) -- (4,0) -- (4,1) -- (5,1) -- (5,0) -- (7,0) -- (7,1) -- (6,1) -- (6,2) -- (7,2) -- (7,3) -- (5,3) -- (5,2) -- (4,2) -- (4,5) -- (5,5) -- (5,4) -- (7,4) -- (7,5) -- (6,5) -- (6,6) -- (7,6) -- (7,7) -- (5,7) -- (5,6) -- (4,6) -- (4,7) -- (2,7) -- (2,6) -- (3,6) -- (3,4) -- (2,4) -- (2,5) -- (1,5) -- (1,4) -- (0,4) -- (0,6) -- (1,6) -- (1,7) -- (0,7);
\end{scope}
\begin{scope}[shift={(15,8)},rotate=90]
\draw[very thick, color = bleu] (0,0) -- (1,0) -- (1,1) -- (0,1) -- (0,3) -- (1,3) -- (1,2) -- (2,2) -- (2,3) -- (3,3) -- (3,1) -- (2,1) -- (2,0) -- (4,0) -- (4,1) -- (5,1) -- (5,0) -- (7,0) -- (7,1) -- (6,1) -- (6,2) -- (7,2) -- (7,3) -- (5,3) -- (5,2) -- (4,2) -- (4,5) -- (5,5) -- (5,4) -- (7,4) -- (7,5) -- (6,5) -- (6,6) -- (7,6) -- (7,7) -- (5,7) -- (5,6) -- (4,6) -- (4,7) -- (2,7) -- (2,6) -- (3,6) -- (3,4) -- (2,4) -- (2,5) -- (1,5) -- (1,4) -- (0,4) -- (0,6) -- (1,6) -- (1,7) -- (0,7);
\end{scope}
\draw[thick, color = bleu] (0,7) -- (0,8);
\draw[thick, color = bleu] (7,8) -- (8,8);
\draw[thick, color = bleu] (15,7) -- (15,8);

\filldraw[color = vert] (6,8) circle (4pt);
\draw[line width = 2pt, color = vert] (6,8) -- (6,9) -- (7,9) -- (7,8) -- (8,8);
\begin{scope}[shift={(15,8)},rotate=90]
\draw[line width = 2pt, color = vert] (2,7) -- (2,6) -- (3,6) -- (3,4) -- (2,4) -- (2,5) -- (1,5) -- (1,4) -- (0,4) -- (0,6) -- (1,6) -- (1,7) -- (0,7);
\filldraw[color = vert] (2,7) circle (4pt);
\end{scope}

\end{tikzpicture}

%% file: figures/not_sofic.tex
 
\begin{tikzpicture}
\foreach \x in {1,...,10}{
    \draw [gray](0+\x -1,0) -- (0+\x,0);
    \draw [gray](-0.5+\x -1,1.5) -- (-0.5+\x,1.5);
}
\foreach \x in {0,...,8}{
    \draw [gray](0+\x,0) -- (1.5+\x,1.5);
    \draw [gray](-0.5+\x,1.5) -- (1+\x,0);
}
\draw [gray](8.5,1.5) -- (10,0);

%Colores
    \draw [-latex,red,thick](2,0) -- (3.5,1.5);
    \draw [-latex,red,thick](3,0) -- (1.5,1.5);
\foreach \x in {1,...,5}{
    \draw [-latex,blue,thick](3.5,1.5) -- (3.5+\x,1.5);
    }
\foreach \x in {1,...,4}{
    \draw [-latex,blue,thick](7,0) -- (7-\x,0);
}    

\draw [-latex,blue,thick](0,0) -- (1,0);
\draw [-latex,blue,thick](1,0) -- (2,0);
\draw [-latex,blue,thick](0.5,1.5) -- (-0.5,1.5);
\draw [-latex,blue,thick](1.5,1.5) -- (0.5,1.5);
\draw[-latex,red,thick](8.5,1.5) -- (7,0);

\draw [blue,dotted,thick](-0.5,1.5) -- (-1,1.5);
\draw [blue,dotted,thick](0,0) -- (-0.5,0);

\draw [gray,dotted,thick](9.5,1.5) -- (10,1.5);
\draw [gray,dotted,thick](10,0) -- (10.5,0);

%Points
\foreach \x in {0,...,10}{
    \draw[fill=white]   (0 + \x,0) circle [radius=1.75pt];
    \draw[fill=white]   (-0.5 + \x, 1.5) circle [radius=1.75pt];
}

\draw (2.6,1) node [anchor=center]{$\textcolor[rgb]{1,0,0}{t}$};
\draw (4,1.3) node [anchor=center]{$\textcolor[rgb]{0,0,1}{s}$};
\end{tikzpicture}

%% file: figures/3ladder.tex
 
\begin{tikzpicture}
\foreach \x in {0,...,4}{
    \draw [gray](0+3*\x,0) -- (3+3*\x,0);
    \draw [gray](0+3*\x,1.5) -- (3+3*\x,1.5);
    \draw [gray](0+3*\x,3) -- (3+3*\x,3);
}
 \foreach \x in {1,...,4}{
     \draw [gray](0+3*\x,0) -- (0+3*\x,1.5);
 }
\foreach \x in {0,...,4}{
     \draw [gray](1.5+3*\x,1.5) -- (1.5+3*\x,3);
 }
\draw [red,very thick, -latex](0,1.5) -- (3,1.5);
\draw [red,very thick, -latex](3,1.5) -- (4.5,1.5);
\draw [blue,very thick, -latex](4.5,1.5) -- (4.5,3);
\draw [blue,very thick, -latex](4.5,3) -- (7.5,3);
\draw [blue,very thick, -latex](7.5,3) -- (10.5,3);
\draw [blue,very thick, -latex](10.5,3) -- (13.5,3);
\draw [blue,very thick, -latex](13.5,3) -- (13.5,1.5);
\draw [blue,very thick, -latex](13.5,1.5) -- (10.5,1.5);
\draw [blue,very thick, -latex](10.5,1.5) -- (7.5,1.5);
\draw [red,very thick, -latex](7.5,1.5) -- (6,1.5);
\draw [red,very thick, -latex](6,1.5) -- (6,0);
\draw [red,very thick, -latex](6,0) -- (3,0);
\draw [red,very thick, -latex](3,0) -- (0,0);

%Points
\foreach \x in {1,...,4}{
    \draw[fill=white]   (0 + 3*\x,0) circle [radius=1.75pt];
    \draw[fill=white]   (0 +1.5*\x,1.5) circle [radius=1.75pt];
    }
\foreach \x in {1,...,5}{
    \draw[fill=white]   (6 +1.5*\x,1.5) circle [radius=1.75pt];
}
\foreach \x in {0,...,4}{
    \draw[fill=white]   (1.5 + 3*\x, 3) circle [radius=1.75pt];
}

\draw (1.5,1.2) node [anchor=center]{$u_0$};
\draw (1.5,3.3) node [anchor=center]{$v_0$};
\draw (3,1.8) node [anchor=center]{$u_k$};
\draw (3,-0.4) node [anchor=center]{$v'_k$};

\draw (4.5,3.3) node [anchor=center]{$g\cdot v_0$};
\draw (4.5,1.2) node [anchor=center]{$g\cdot u_0$};
\draw (6,1.8) node [anchor=center]{$g\cdot u_k$};
\draw (6,-0.4) node [anchor=center]{$g\cdot v'_k$};

\draw (13.5,3.3) node [anchor=center]{$g^4\cdot v_0$};
\draw (13.5,1.2) node [anchor=center]{$g^4\cdot u_0$};

\end{tikzpicture}

%% file: figures/2ladder.tex
 
\begin{tikzpicture}
\foreach \x in {0,...,4}{
    \draw [gray](0+3*\x,1.5) -- (3+3*\x,1.5);
    \draw [gray](0+3*\x,3) -- (3+3*\x,3);
}
\foreach \x in {0,...,4}{
     \draw [gray](1.5+3*\x,1.5) -- (1.5+3*\x,3);
 }
\draw [red,very thick, -latex](0,1.5) -- (1.5,1.5);
\draw [red,very thick, -latex](1.5,1.5) -- (4.5,1.5);
\draw [blue,very thick, -latex](4.5,1.5) -- (4.5,3);
\draw [blue,very thick, -latex](4.5,3) -- (7.5,3);
\draw [blue,very thick, -latex](7.5,3) -- (10.5,3);
\draw [blue,very thick, -latex](10.5,3) -- (13.5,3);
\draw [blue,very thick, -latex](13.5,3) -- (13.5,1.5);
\draw [blue,very thick, -latex](13.5,1.5) -- (10.5,1.5);
\draw [blue,very thick, -latex](10.5,1.5) -- (7.5,1.5);
\draw [red,very thick, -latex](7.5,1.5) -- (6,1.5);

\draw[red, very thick, -latex](6,1.5) to[out=270,in=180] (9,0);

%Points
\foreach \x in {0,...,4}{
    \draw[fill=white]   (1.5 +3*\x,1.5) circle [radius=1.75pt];
    }
\foreach \x in {1,...,5}{
    % \draw[fill=white]   (6 +1.5*\x,1.5) circle [radius=1.75pt];
}
\foreach \x in {0,...,4}{
    \draw[fill=white]   (1.5 + 3*\x, 3) circle [radius=1.75pt];
}

\draw[fill=white]   (6, 1.5) circle [radius=1.75pt];

\draw (1.5,1.2) node [anchor=center]{$u_0$};
\draw (1.5,3.3) node [anchor=center]{$v_0$};
% \draw (3,1.8) node [anchor=center]{$u_k$};

\draw (4.5,3.3) node [anchor=center]{$g\cdot v_0$};
\draw (4.5,1.2) node [anchor=center]{$g\cdot u_0$};
\draw (6,1.8) node [anchor=center]{$u_k = v'_N$};

\draw (13.5,3.3) node [anchor=center]{$g^4\cdot v_0$};
\draw (13.5,1.2) node [anchor=center]{$g^4\cdot u_0$};

\draw (8,0.3) node [anchor=center]{$\pi_3$};

\end{tikzpicture}

%% file: alexandria.bib
@book{lind2021introduction,
    AUTHOR = {Lind, Douglas and Marcus, Brian},
     TITLE = {An Introduction to Symbolic Dynamics and Coding},
    SERIES = {Cambridge Mathematical Library},
   EDITION = {Second},
 PUBLISHER = {Cambridge University Press, Cambridge},
      YEAR = {2021},
     PAGES = {xix+550},
      ISBN = {978-1-108-82028-8},
       DOI = {10.1017/9781108899727},
}

@article{Seward_2014, title={Burnside’s Problem, spanning trees and tilings}, volume={18}, ISSN={1364-0380, 1465-3060}, DOI={10.2140/gt.2014.18.179}, number={1}, journal={Geometry \& Topology}, author={Seward, Brandon}, year={2014}, month={1}, pages={179–210}, language={en} }

@article{gilman2007characterisation,
  title={A characterisation of virtually free groups},
  author={Gilman, Robert H and Hermiller, Susan and Holt, Derek F and Rees, Sarah},
  journal={Archiv der Mathematik},
  volume={89},
  number={4},
  pages={289--295},
  year={2007},
  publisher={Springer}
}

@article{charney2004language,
  title={The language of geodesics for Garside groups},
  author={Charney, Ruth and Meier, John},
  journal={Mathematische Zeitschrift},
  volume={248},
  number={3},
  pages={495--509},
  year={2004},
  publisher={Springer}
}

@article{neumann1995automatic,
  title={Automatic structures, rational growth, and geometrically finite hyperbolic groups},
  author={Neumann, Walter D and Shapiro, Michael},
  journal={Inventiones mathematicae},
  volume={120},
  number={1},
  pages={259--287},
  year={1995},
  publisher={Springer}
}

@article{holt2012artin,
  title={Artin groups of large type are shortlex automatic with regular geodesics},
  author={Holt, Derek F and Rees, Sarah},
  journal={Proceedings of the London Mathematical Society},
  volume={104},
  number={3},
  pages={486--512},
  year={2012},
  publisher={Oxford University Press}
}

@book{howlett1993miscellaneous,
  title={Miscellaneous facts about Coxeter groups},
  author={Howlett, Robert B},
  year={1993},
  publisher={School of Mathematics and Statistics, University of Sydney}
}

@article{antolin2016finite,
  title={Finite generating sets of relatively hyperbolic groups and applications to geodesic languages},
  author={Antolin, Yago and Ciobanu, Laura},
  journal={Transactions of the American Mathematical Society},
  volume={368},
  number={11},
  pages={7965--8010},
  year={2016}
}

@book {epstein1992word,
    AUTHOR = {Epstein, David B. A. and Cannon, James W. and Holt, Derek F.
              and Levy, Silvio V. F. and Paterson, Michael S. and Thurston,
              William P.},
     TITLE = {Word processing in groups},
 PUBLISHER = {Jones and Bartlett Publishers, Boston, MA},
      YEAR = {1992},
     PAGES = {xii+330},
      ISBN = {0-86720-244-0},
}

@book{ceccherini2010,
 address={Berlin, Heidelberg}, 
 series={Springer Monographs in Mathematics}, 
 title={Cellular Automata and Groups}, 
 ISBN={978-3-642-14033-4}, 
 DOI={10.1007/978-3-642-14034-1}, 
 publisher={Springer Berlin Heidelberg},
 author={Ceccherini-Silberstein, Tullio and Coornaert, Michel}, 
 year={2010}, 
 collection={Springer Monographs in Mathematics}, 
 language={en} }

@article{duminil2012connective,
  title={The connective constant of the honeycomb lattice equals $\sqrt{2+\sqrt{2}}$},
  author={Duminil-Copin, Hugo and Smirnov, Stanislav},
  journal={Annals of Mathematics},
  pages={1653--1665},
  year={2012},
  publisher={JSTOR}
}

@article{bridson2012groups,
  title={On groups whose geodesic growth is polynomial},
  author={Bridson, Martin R and Burillo, Jos{\'e} and Elder, Murray and {\v{S}}uni{\'c}, Zoran},
  journal={International Journal of Algebra and Computation},
  volume={22},
  number={05},
  pages={1250048},
  year={2012},
  publisher={World Scientific}
}

@article{myasnikov2011algorithmically,
  title={Algorithmically finite groups},
  author={Myasnikov, Alexei and Osin, Denis},
  journal={Journal of Pure and Applied Algebra},
  volume={215},
  number={11},
  pages={2789--2796},
  year={2011},
  publisher={Elsevier}
}

@misc{bowen108,
  author = {Rufus Bowen},  
  title = {{Problem 108, Rufus Bowen's Notebook}},
  howpublished = {\url{https://bowen.pims.math.ca/problems/108}},
  note = {Pacific Insititute for the Mathematical Sciences},
  year = {2017}
}

@incollection {grimmet2017self,
    AUTHOR = {Grimmett, Geoffrey R. and Li, Zhongyang},
     TITLE = {Self-avoiding walks and connective constants},
 BOOKTITLE = {Sojourns in probability theory and statistical physics. {III}.
              {I}nteracting particle systems and random walks, a
              {F}estschrift for {C}harles {M}. {N}ewman},
    SERIES = {Springer Proc. Math. Stat.},
    VOLUME = {300},
     PAGES = {215--241},
 PUBLISHER = {Springer, Singapore},
      YEAR = {2019},
      ISBN = {978-981-15-0302-3; 978-981-15-0301-6},
       DOI = {10.1007/978-981-15-0302-3\_8},
}

@article{alm1990random,
  title={Random self-avoiding walks on one-dimensional lattices},
  author={Alm, Sven Erick and Janson, Svante},
  journal={Stochastic Models},
  volume={6},
  number={2},
  pages={169--212},
  year={1990},
  publisher={Taylor \& Francis}
}

@article{hammersley1954poor,
  title={Poor man's Monte Carlo},
  author={Hammersley, John M. and Morton, K William},
  journal={Journal of the Royal Statistical Society: Series B (Methodological)},
  volume={16},
  number={1},
  pages={23--38},
  year={1954},
  publisher={Wiley Online Library}
}

@article{grimmett2014extendable,
  title={Extendable self-avoiding walks},
  author={Grimmett, Geoffrey R. and Holroyd, Alexander E and Peres, Yuval},
  journal={Annales de l’Institut Henri Poincar{\'e} D},
  volume={1},
  number={1},
  pages={61--75},
  year={2014}
}

@article{rosenfeld2022finding,
  title={Finding lower bounds on the growth and entropy of subshifts over countable groups},
  author={Rosenfeld, Matthieu},
  journal={arXiv preprint arXiv:2204.00394},
  year={2022}
}

@article{grimmett2014strict,
  title={Strict inequalities for connective constants of transitive graphs},
  author={Grimmett, Geoffrey R. and Li, Zhongyang},
  journal={SIAM Journal on Discrete Mathematics},
  volume={28},
  number={3},
  pages={1306--1333},
  year={2014},
  publisher={SIAM}
}

@inproceedings{aubrun2023domino,
  title={Domino snake problems on groups},
  author={Aubrun, Nathalie and Bitar, Nicol{\'a}s},
  booktitle={International Symposium on Fundamentals of Computation Theory},
  pages={46--59},
  year={2023},
  organization={Springer}
}

@article{gilch2017counting,
  title={Counting self-avoiding walks on free products of graphs},
  author={Gilch, Lorenz A and M{\"u}ller, Sebastian},
  journal={Discrete Mathematics},
  volume={340},
  number={3},
  pages={325--332},
  year={2017},
  publisher={Elsevier}
}

@inproceedings{jeandel2019characterization,
  title={A characterization of subshifts with computable language},
  author={Jeandel, Emmanuel and Vanier, Pascal},
  booktitle={36th International Symposium on Theoretical Aspects of Computer Science (STACS 2019)},
  year={2019},
  organization={Schloss Dagstuhl-Leibniz-Zentrum fuer Informatik}
}

@article{haring1983groups,
  title={Groups and simple languages},
  author={Haring-Smith, Robert H},
  journal={Transactions of the American Mathematical Society},
  volume={279},
  number={1},
  pages={337--356},
  year={1983}
}

@article{lindorfer2020language,
  title={The language of self-avoiding walks},
  author={Lindorfer, Christian and Woess, Wolfgang},
  journal={Combinatorica},
  volume={40},
  number={5},
  pages={691--720},
  year={2020},
  publisher={Springer}
}

@article{clisby2013endless,
  title={Endless self-avoiding walks},
  author={Clisby, Nathan},
  journal={Journal of Physics A: Mathematical and Theoretical},
  volume={46},
  number={23},
  pages={235001},
  year={2013},
  publisher={IOP Publishing}
}

@article {hammersley1962constant,
    AUTHOR = {Hammersley, John M. and Welsh, D. J. A.},
     TITLE = {Further results on the rate of convergence to the connective
              constant of the hypercubical lattice},
   JOURNAL = {Quart. J. Math. Oxford Ser. (2)},
  FJOURNAL = {The Quarterly Journal of Mathematics. Oxford. Second Series},
    VOLUME = {13},
      YEAR = {1962},
     PAGES = {108--110},
      ISSN = {0033-5606,1464-3847},
       DOI = {10.1093/qmath/13.1.108},
}

@article {grimmet2018locality,
    AUTHOR = {Grimmett, Geoffrey R. and Li, Zhongyang},
     TITLE = {Locality of connective constants},
   JOURNAL = {Discrete Math.},
  FJOURNAL = {Discrete Mathematics},
    VOLUME = {341},
      YEAR = {2018},
    NUMBER = {12},
     PAGES = {3483--3497},
      ISSN = {0012-365X,1872-681X},
       DOI = {10.1016/j.disc.2018.08.013},
}

@article {lindorfer2020bridge,
    AUTHOR = {Lindorfer, Christian},
     TITLE = {A general bridge theorem for self-avoiding walks},
   JOURNAL = {Discrete Math.},
  FJOURNAL = {Discrete Mathematics},
    VOLUME = {343},
      YEAR = {2020},
    NUMBER = {12},
     PAGES = {112092, 11},
      ISSN = {0012-365X,1872-681X},
   MRCLASS = {05C63 (05A16)},
  MRNUMBER = {4137630},
MRREVIEWER = {Lyuben\ R.\ Mutafchiev},
       DOI = {10.1016/j.disc.2020.112092},
}

@article{grimmett2015self,
  title={Self-Avoiding Walks and Amenability},
  author={Grimmett, Geoffrey R. and Li, Zhongyang},
  journal={The Electronic Journal of Combinatorics},
  pages={P4--38},
  year={2017}
}

@article{grimmett2017connective,
  title={Connective constants and height functions for Cayley graphs},
  author={Grimmett, Geoffrey R. and Li, Zhongyang},
  journal={Transactions of the American Mathematical Society},
  volume={369},
  number={8},
  pages={5961--5980},
  year={2017}
}

@article{higman1951finitely,
  title={A finitely generated infinite simple group},
  author={Higman, Graham},
  journal={Journal of the London Mathematical Society},
  volume={1},
  number={1},
  pages={61--64},
  year={1951},
  publisher={Oxford University Press}
}

@article {rosenfeld2020colorings,
    AUTHOR = {Rosenfeld, Matthieu},
     TITLE = {Another approach to non-repetitive colorings of graphs of
              bounded degree},
   JOURNAL = {Electron. J. Combin.},
  FJOURNAL = {Electronic Journal of Combinatorics},
    VOLUME = {27},
      YEAR = {2020},
    NUMBER = {3},
     PAGES = {Paper No. 3.43, 16},
      ISSN = {1077-8926},
       DOI = {10.37236/9667},
}

@article {elder2022rewriting,
    AUTHOR = {Elder, Murray and Piggott, Adam},
     TITLE = {Rewriting systems, plain groups, and geodetic graphs},
   JOURNAL = {Theoret. Comput. Sci.},
  FJOURNAL = {Theoretical Computer Science},
    VOLUME = {903},
      YEAR = {2022},
     PAGES = {134--144},
      ISSN = {0304-3975,1879-2294},
       DOI = {10.1016/j.tcs.2021.12.022},
}

@article {anisimov1971regular,
    AUTHOR = {Anīsīmov, A. V.},
     TITLE = {The group languages},
   JOURNAL = {Kibernetika (Kiev)},
  FJOURNAL = {Otdelenie Matematiki, Mekhaniki i Kibernetiki Akademii Nauk
              Ukrainsko\u{\i} SSR. Kibernetika},
      YEAR = {1971},
    NUMBER = {4},
     PAGES = {18--24},
      ISSN = {0023-1274},
}

@article{halin1973automorphisms,
    AUTHOR = {Halin, Rudolf},
     TITLE = {Automorphisms and endomorphisms of infinite locally finite
              graphs},
   JOURNAL = {Abh. Math. Sem. Univ. Hamburg},
  FJOURNAL = {Abhandlungen aus dem Mathematischen Seminar der
              Universit\"{a}t Hamburg},
    VOLUME = {39},
      YEAR = {1973},
     PAGES = {251--283},
      ISSN = {0025-5858,1865-8784},
       DOI = {10.1007/BF02992834},
}

@article {epstein1961ends,
    AUTHOR = {Epstein, David B. A.},
     TITLE = {Factorization of 3-manifolds},
   JOURNAL = {Comment. Math. Helv.},
  FJOURNAL = {Commentarii Mathematici Helvetici},
    VOLUME = {36},
      YEAR = {1961},
     PAGES = {91--102},
      ISSN = {0010-2571,1420-8946},
       DOI = {10.1007/BF02566894},
}

@article {wall1967ends,
    AUTHOR = {Wall, C. T. C.},
     TITLE = {Poincar\'{e} complexes. {I}},
   JOURNAL = {Ann. of Math. (2)},
  FJOURNAL = {Annals of Mathematics. Second Series},
    VOLUME = {86},
      YEAR = {1967},
     PAGES = {213--245},
      ISSN = {0003-486X},
       DOI = {10.2307/1970688},
}

@book {lima2013virtually,
    AUTHOR = {Lima Gon\c{c}alves, Daciberg and Guaschi, John},
     TITLE = {The classification of the virtually cyclic subgroups of the
              sphere braid groups},
    SERIES = {SpringerBriefs in Mathematics},
 PUBLISHER = {Springer, Cham},
      YEAR = {2013},
     PAGES = {x+102},
      ISBN = {978-3-319-00256-9; 978-3-319-00257-6},
       DOI = {10.1007/978-3-319-00257-6},
}

@mastersthesis{gilabert2022virt,
title={Subshifts sobre grupos virtualmente cíclicos},
author={Gilabert Vio, Martín},
year={2022},
school={Facultad de Ciencias Físicas y Matemáticas, Universidad de Chile},
address = {Santiago, Chile}
}

@article {jacobsen2016constant,
    AUTHOR = {Jacobsen, Jesper Lykke and Scullard, Christian R. and
              Guttmann, Anthony J.},
     TITLE = {On the growth constant for square-lattice self-avoiding walks},
   JOURNAL = {J. Phys. A},
  FJOURNAL = {Journal of Physics. A. Mathematical and Theoretical},
    VOLUME = {49},
      YEAR = {2016},
    NUMBER = {49},
     PAGES = {494004, 18},
      ISSN = {1751-8113,1751-8121},
       DOI = {10.1088/1751-8113/49/49/494004},
}

@article{grimmett2015bounds,
  title={Bounds on connective constants of regular graphs},
  author={Grimmett, Geoffrey R. and Li, Zhongyang},
  journal={Combinatorica},
  volume={35},
  number={3},
  pages={279--294},
  year={2015},
  publisher={Springer}
}

@article{hertling2008shifts,
  title={Shifts with decidable language and non-computable entropy},
  author={Hertling, Peter and Spandl, Christoph},
  journal={Discrete Mathematics \& Theoretical Computer Science},
  volume={10},
  number={Automata, Logic and Semantics},
  year={2008},
  publisher={Episciences. org}
}

@article {lind1984entropies,
    AUTHOR = {Lind, Douglas},
     TITLE = {The entropies of topological {M}arkov shifts and a related
              class of algebraic integers},
   JOURNAL = {Ergodic Theory Dynam. Systems},
  FJOURNAL = {Ergodic Theory and Dynamical Systems},
    VOLUME = {4},
      YEAR = {1984},
    NUMBER = {2},
     PAGES = {283--300},
      ISSN = {0143-3857,1469-4417},
       DOI = {10.1017/S0143385700002443},
}

@article{watkins1986infinite,
  title={Infinite paths that contain only shortest paths},
  author={Watkins, Mark E},
  journal={Journal of Combinatorial Theory, Series B},
  volume={41},
  number={3},
  pages={341--355},
  year={1986},
  publisher={Elsevier}
}

@phdthesis{avasjo2004automata,
  title={Automata and growth functions of Coxeter groups},
  author={Avasj{\"o}, Agneta},
  year={2004},
  school={Matematik}
}

@article{moller2024quotients,
  title={On quotients of Coxeter groups},
  author={M{\"o}ller, Philip and Varghese, Olga},
  journal={Journal of Algebra},
  volume={639},
  pages={516--531},
  year={2024},
  publisher={Elsevier}
}

@article{bishop2021geodesic,
  title={Geodesic growth in virtually abelian groups},
  author={Bishop, Alex},
  journal={Journal of Algebra},
  volume={573},
  pages={760--786},
  year={2021},
  publisher={Elsevier}
}

@article{bodart2023intermediate,
  title={Intermediate geodesic growth in virtually nilpotent groups},
  author={Bodart, Corentin},
  journal={arXiv preprint arXiv:2306.10381},
  year={2023}
}

@article{bishop2020virtually,
  title = {A virtually 2-step nilpotent group with polynomial geodesic growth},
  volume = {33},
  ISSN = {2415-721X},
  DOI = {10.12958/adm1667},
  number = {2},
  journal = {Algebra and Discrete Mathematics},
  publisher = {Luhansk Taras Shevchenko National University},
  author = {Bishop,  Alex and Elder,  Murray},
  year = {2022},
  pages = {21–28}
}

@article{halin1964unendliche,
  title={{\"U}ber unendliche wege in graphen},
  author={Halin, Rudolf},
  journal={Mathematische Annalen},
  volume={157},
  number={2},
  pages={125--137},
  year={1964},
  publisher={Springer}
}

@article{thomassen1993vertex,
  title={Vertex-transitive graphs and accessibility},
  author={Thomassen, Carsten and Woess, Wolfgang},
  journal={Journal of Combinatorial Theory, Series B},
  volume={58},
  number={2},
  pages={248--268},
  year={1993},
  publisher={Elsevier}
}

@article{freudenthal1944enden,
  title={{\"U}ber die enden diskreter r{\"a}ume und gruppen},
  author={Freudenthal, Hans},
  journal={Commentarii Mathematici Helvetici},
  volume={17},
  number={1},
  pages={1--38},
  year={1944},
  publisher={Springer}
}

@article{hopf1943enden,
  title={Enden offener R{\"a}ume und unendliche diskontinuierliche Gruppen},
  author={Hopf, Heinz},
  journal={Commentarii Mathematici Helvetici},
  volume={16},
  number={1},
  pages={81--100},
  year={1943},
  publisher={Springer}
}

@article {lehner2023saw,
    AUTHOR = {Lehner, Florian and Lindorfer, Christian},
     TITLE = {Self-avoiding walks and multiple context-free languages},
   JOURNAL = {Comb. Theory},
  FJOURNAL = {Combinatorial Theory},
    VOLUME = {3},
      YEAR = {2023},
    NUMBER = {1},
     PAGES = {Paper No. 18, 50},
      ISSN = {2766-1334},
       DOI = {10.5070/c63160431},
}

@article{flory1949configuration,
  title={The configuration of real polymer chains},
  author={Flory, Paul J},
  journal={The Journal of Chemical Physics},
  volume={17},
  number={3},
  pages={303--310},
  year={1949},
  publisher={American Institute of Physics}
}

@article{kesten1963number,
  title={On the number of self-avoiding walks},
  author={Kesten, Harry},
  journal={Journal of Mathematical Physics},
  volume={4},
  number={7},
  pages={960--969},
  year={1963},
  publisher={American Institute of Physics}
}

@inproceedings{hammersley1961number,
  title={The number of polygons on a lattice},
  author={Hammersley, John M.},
  booktitle={Mathematical Proceedings of the Cambridge Philosophical Society},
  volume={57},
  number={3},
  pages={516--523},
  year={1961},
  organization={Cambridge University Press}
}

@article{panagiotis2019self,
  title={Self-avoiding walks and polygons on hyperbolic graphs},
  author={Panagiotis, Christoforos},
  journal={arXiv preprint arXiv:1908.00127},
  year={2019}
}

@article{carrasco2023geometric,
  title={The geometric subgroup membership problem},
  author={Carrasco-Vargas, Nicanor},
  journal={arXiv preprint arXiv:2303.14820},
  year={2023}
}

@article {grimmett2020cubic,
    AUTHOR = {Grimmett, Geoffrey R. and Li, Zhongyang},
     TITLE = {Cubic graphs and the golden mean},
   JOURNAL = {Discrete Math.},
  FJOURNAL = {Discrete Mathematics},
    VOLUME = {343},
      YEAR = {2020},
    NUMBER = {1},
     PAGES = {111638, 32},
      ISSN = {0012-365X,1872-681X},
       DOI = {10.1016/j.disc.2019.111638},
}
